\PassOptionsToPackage{top=1cm, bottom=1.5cm, left=2cm, right=2cm}{geometry}
\documentclass[final,3p,times, leqno]{elsarticle}

\makeatletter
\def\ps@pprintTitle{%
   \let\@oddhead\@empty
   \let\@evenhead\@empty
   \def\@oddfoot{\reset@font\hfil\thepage\hfil}
   \let\@evenfoot\@oddfoot
}
\makeatother

\journal{}

\PassOptionsToPackage{dvipsnames,svgnames,x11names}{xcolor}

\usepackage{tikz}
\usetikzlibrary{arrows.meta,decorations.pathreplacing,calc}
\usepackage[colorlinks=true, linkcolor=blue, citecolor=red, urlcolor=blue]{hyperref}

\usepackage{geometry}

\usepackage{multirow}

\usepackage{amsthm, amssymb, amsfonts, mathrsfs, amsmath}
\usepackage{longtable}
\usepackage{lmodern}
\usepackage{hyperref}
\usepackage[T5]{fontenc}

\usepackage{amscd}
\allowdisplaybreaks[3] 
\usepackage{graphicx}

\usepackage[all]{xy}

\theoremstyle{plain}
\newtheorem{thm}{Theorem}[subsection]
\newtheorem{conj}[thm]{Conjecture}
\newtheorem{prop}[thm]{Proposition}
\newtheorem{lem}[thm]{Lemma}
\newtheorem{corl}[thm]{Corollary}

\theoremstyle{plain}
\newtheorem{therm}{Theorem}[section]

\newtheorem{corls}[therm]{Corollary}

\theoremstyle{definition}
\newtheorem{defns}[thm]{Definition}
\newtheorem{notas}[therm]{Notation}
\newtheorem{ex}[thm]{Example}
\newtheorem{nx}[therm]{Remark}
\newtheorem{rem}[thm]{Remark}

\theoremstyle{definition}

\theoremstyle{plain}

\everymath{\displaystyle}

\def\DD{D\kern-.7em\raise0.4ex\hbox{\char '55}\kern.33em}

\usepackage{sectsty}
\subsectionfont{\fontsize{11.5}{11.5}\selectfont}

\usepackage{tikz}
\usetikzlibrary{arrows.meta}
\newcommand{\myto}{\tikz[baseline] \draw[-{>[length=5pt, width=3.5pt]}] (0pt,0pt) (21pt,0pt) (3pt,2.5pt) -- (20pt,2.5pt);}

\begin{document}
\fontsize{11.5pt}{11.5}\selectfont

\begin{frontmatter}

\title{{\bf The fifth algebraic transfer in generic degrees and validation of a localized Kameko's conjecture}}
\author{{\bf \DD\d{\u A}NG V\~O PH\'UC}}

\address{{\fontsize{10pt}{10}\selectfont Department of Mathematics, FPT University, Quy Nhon AI Campus,\\ An Phu Thinh New Urban Area, Vietnam\\ ORCID: \url{https://orcid.org/0000-0002-6885-3996}}}
\cortext[]{\href{mailto:dangphuc150488@gmail.com, phucdv14@fpt.edu.vn}{\texttt{Email Address: dangphuc150488@gmail.com, phucdv14@fpt.edu.vn}}}

\begin{abstract}
This paper develops our previous works concerning the classical Peterson hit problem for the polynomial algebra on five variables over the mod--2 Steenrod algebra $\mathscr A$ in a generic family of degrees, together with applications to the fifth Singer algebraic transfer and a localized variation of Kameko's conjecture. As a topological illustration of the usefulness of the Steenrod algebra, we prove that $\mathbb{C}P^4/\mathbb{C}P^2$ and $\mathbb{S}^6\vee \mathbb{S}^8$ are not homotopy equivalent by showing that their mod--2 cohomologies are not isomorphic as $\mathscr A$-modules, and we further determine the homotopy type of the quotient $\mathbb{C}P^n/\mathbb{C}P^{\,n-2}$ for all $n\ge 3$. For the generic degrees under consideration, we determine the relevant cohit spaces and describe the associated $GL(5,\mathbb F_2)$-module structure. As a consequence, the fifth algebraic transfer is an isomorphism in an explicit infinite family of internal degrees. These results were independently verified by implementations in \texttt{SageMath} and \texttt{OSCAR}. We also study a localized form of Kameko's conjecture concerning the dimensions of the indecomposables $\mathbb F_2\otimes_{\mathscr A}\mathbb F_2[x_1,\ldots,x_m]$ relative to parameter vectors, and prove that this conjecture holds for all $m\ge 1$ in certain degrees.
\end{abstract}

\begin{keyword}

Primary cohomology operations in algebraic topology; Adams spectral sequences; Steenrod algebra; Algebraic transfer; Localized variation of Kameko's conjecture; \texttt{SageMath} and \texttt{OSCAR} (computer algebra systems).
\MSC[2020] 55S05, 55T15, 55S10, 55R12, 68W30
\end{keyword}

\end{frontmatter}

\tableofcontents

\subsection*{Acknowledgments}

The author thanks the handling editor and the anonymous referee for a careful reading of the manuscript and for helpful suggestions that improved the exposition.

\section{Context and Motivation}\label{s1}

For each integer $i\ge 0$, let $\mathcal{O}^S(i,\mathbb{F}_2,\mathbb{F}_2)$ denote the $\mathbb{F}_2$-vector space of stable cohomology operations of degree $i$ with coefficients in $\mathbb{F}_2$. The mod--2 Steenrod algebra is the graded $\mathbb{F}_2$-algebra
\[
\mathscr{A}=\bigoplus_{i\ge 0}\mathcal{O}^S(i,\mathbb{F}_2,\mathbb{F}_2).
\]
It has been studied extensively in algebraic topology and related areas; see, for example, \cite{Karaca,TK,VK,Ljungstrom,W.W,W.W2}. 
Beyond encoding stable operations, $\mathscr{A}$ is a Hopf algebra. Its conjugation (antipode) is closely connected to Milnor-basis computations and to dual descriptions relevant for transfer problems. 
From a broader algebraic perspective, the Steenrod algebra is closely related to other Hopf algebras that arise naturally in algebraic topology. In particular, Crossley and Turgay \cite{CrossleyTurgay} studied conjugation invariants in the Leibniz--Hopf algebra, while Turgay \cite{Turgay2020} investigated the Hopf algebra epimorphism from the mod-$p$ Leibniz--Hopf algebra to the Bockstein-free Steenrod algebra and its implications for conjugation and invariant theory. These works provide valuable structural insight into conjugation phenomena, invariant-theoretic aspects, and Hopf-algebra methods that are closely connected with the algebraic framework surrounding Steenrod-algebra computations. For alternative viewpoints on the Adem relations, conjugation, and structural embeddings of the dual Steenrod algebra, see~\cite{Turgay1,Turgay2,Turgay3}.

Since $\mathscr A$ acts as operations on mod-two cohomology, the cohomology of any space (or group or Lie algebra, or other sufficiently well-behaved object) carries a natural structure of $\mathscr A$-module. In many cases, this additional $\mathscr A$-module structure on $H^{*}(X;\mathbb F_2)$ reveals information that is invisible at the level of graded commutative cohomology rings alone. For instance, as we show in Subsect.~\ref{sub3-1}, the CW complexes $\mathbb{C}P^4/\mathbb{C}P^2$ and $\mathbb{S}^6\vee \mathbb{S}^8$ have isomorphic mod--2 cohomology rings as graded commutative $\mathbb F_2$-algebras, but their cohomologies are not isomorphic as $\mathscr A$-modules; consequently, these two spaces are not homotopy equivalent. More generally, as shown in Remark~\ref{rem312}, we determine the homotopy type of the quotient $\mathbb{C}P^n/\mathbb{C}P^{\,n-2}$ for every $n\ge 3$, thereby providing a broader topological illustration of the usefulness of the Steenrod algebra.

\medskip

Regarding the structure of modules over the Steenrod ring, as it is known, the mod--2 cohomology of the Eilenberg--MacLane space $K(\mathbb{F}_2,1)$ is a polynomial ring $\mathbb F_2[x]$ in one variable, and thus, $$ H^*((K(\mathbb F_2, 1))^{\times m}; \mathbb F_2) = \underset{\text{ $m$ times}}{\underbrace{ \mathbb{F}_2[x_1]\otimes_{\mathbb{F}_2}  \mathbb{F}_2[x_2]\otimes_{\mathbb{F}_2}\cdots\otimes_{\mathbb{F}_2} \mathbb{F}_2[x_m]}}=\mathbb{F}_2[x_1, \ldots, x_m],$$ where $|x_1| = |x_2| = \cdots = |x_m| = 1.$ As the polynomial ring $\mathcal P_m:=\mathbb{F}_2[x_1, \ldots, x_m]$ is the cohomology of a CW-complex, it is equipped with a structure of an unstable module over $\mathscr A.$ The action of $\mathscr A$ on $\mathcal P_m$ is determined by the rule $Sq^{k}(x_j^{n}) = \binom{n}{k}x_j^{n+k},$ extended to all of $\mathcal P_m$ by the Cartan formula and linearity.

\medskip

Denote by $(\mathcal P_m)_n$ the $\mathbb F_2$-subspace of $\mathcal P_m$ consisting of all homogeneous polynomials of degree $n$ in $\mathcal P_m.$ Then we have homogeneous $\mathbb N$-graded modules $\mathcal P_m = \bigoplus_{n}(\mathcal P_m)_n.$ i.e., $(\mathcal P_m)_n  = 0$ for all $n < 0,$ with an action of $\mathscr A$ such that $(Sq^k, f)\longmapsto 0$ for any $f\in (\mathcal P_m)_n$ with $n < k.$ Let $G(m):=GL(m,\mathbb{F}_2)$ denote the general linear group of rank $m$ over $\mathbb{F}_2$. This group acts on $\mathcal P_m$ by matrix substitution, and so, in addition to $\mathscr A$-module structure, $\mathcal P_m$ is also a (right) $G(m)$-module. Let $\overline{\mathscr A}$ denote the positive degree part of $\mathscr A.$ A polynomial in $\mathcal P_m$ is called $\mathscr A$-decomposable (or "hit"), if it is a combination of elements in the images of the Steenrod squares in $\overline{\mathscr A}.$ The classical "hit problem" for the algebra $\mathcal P_m,$ which is concerned with seeking a minimal set of generators for $\mathcal P_m$ over $\overline{\mathscr A},$ has been initiated in a variety of contexts by  Peterson \cite{F.P}, Singer \cite{W.S2}, and Wood \cite{R.W}. The geometric significance of this hit problem's solution lies in its ability to describe how cells in a $CW$-complex at the prime 2 are attached to cells of lower dimensions. The study of modules over the Steenrod algebra $\mathscr A$ and related hit problems is a central topic in algebraic topology and has received extensive attention from numerous authors, including Anick and Peterson \cite{Anick}, Crabb and Hubbuck \cite{C.H}, Janfada \cite{Janfada2-1}, Kameko \cite{M.K}, Mothebe and collaborators \cite{M.M2, M.M, MKR}, Nam \cite{Nam}, Repka and Selick \cite{R.S}, Walker and Wood \cite{W.W, W.W2}, Nguyen Sum \cite{N.S1}, the present author and Nguyen Sum \cite{P.S1, P.S2}, the present author \cite{D.P1, D.P2, D.P3, D.P5, D.P6, D.P10, D.P8-0, D.P9}, etc. For background and further references, we refer the reader to the monographs of Walker and Wood \cite{W.W, W.W2}.

\medskip

When $\mathbb F_2$ is an $\mathscr A$-module concentrated in degree 0, one variation of the hit problem is to give a module basis for the cohit module $Q^{\otimes m} :=  \mathbb F_2\otimes_{\mathscr A}\mathcal P_m \cong \mathcal P_m/\overline{\mathscr A}\mathcal P_m.$  This implies that the set of the hit polynomials forms a submodule $\overline{\mathscr A}\mathcal P_m$ of $\mathcal P_m.$ Here and subsequently, we will denote the homogeneous components of degree $n$ in $Q^{\otimes m}$ by 
$$Q^{\otimes m}_n:= (Q^{\otimes m})_n = (\mathcal P_m)_n/(\mathcal P_m)_n\cap \overline{\mathscr A}\mathcal P_m.$$ 
Consequently, we have a decomposition $Q^{\otimes m} = \bigoplus_{n\geq 0}Q^{\otimes m}_n.$ Since the action of the linear group $G(m)$ and the action of $\mathscr A$ on $\mathcal P_m$ commute, there is an induced action of $G(m)$ on $Q^{\otimes m}.$ Therefore, $Q^{\otimes m}$ can be viewed as a representation of $G(m).$  The $\mathscr A$-indecomposables are presently only ascertainable for $m\leq 4$ (see \cite{J.B, Janfada2-1, M.K, F.P, N.S1}). The general case remains a stimulating and unsolved problem.  It must be noted that an essential aspect of this hit problem is verifying whether a given polynomial in $\mathcal P_m$ is indeed hit or not. The dearth of comprehensive information available for larger values of $m$ has instigated a substantial number of researchers to undertake investigations into the $\mathscr A$-decomposables $\overline{\mathscr A}\mathcal P_m$. In fact, research into $\overline{\mathscr A}\mathcal P_m$ may offer potential avenues for progress towards a more comprehensive understanding of the $\mathscr A$-indecomposables for larger values of $m.$ For instance, in 2008, Ali Janfada \cite{Janfada2-1} accomplished the work by Kameko \cite{M.K} by introducing a criterion for monomials with generic degrees in $\mathcal P_3$ to be hit. In addition, in our recent work \cite{Phuc25f}, we study the dimension of $\overline{\mathscr A}\mathcal P_m$ by means of graph theory and computational algorithms in the \texttt{SageMath} computer algebra system.

\medskip

We will write $(\mathcal P_m)^{*}:=H_*((K(\mathbb F_2, 1))^{\times m}; \mathbb F_2)$ as the dual of $\mathcal P_m.$ This homology is, in fact, the divided power algebra in $m$ generators $a^{(1)}_j,\, 1\leq j\leq m,$ where each $a_j^{(1)} \equiv a_j$ denotes the linear dual to $x_j.$ The dual of the hit problem is equivalent to determining the submodule $P_{\mathscr A}(\mathcal P_m)^{*} = (Q^{\otimes m})^*$ of $(\mathcal P_m)^{*}$ that are annihilated by all elements of $\overline{\mathscr A}.$ Here we allow $\mathscr A$ to act from left on $(\mathcal P_m)^{*}$ by means of the dual Steenrod operations 
\xymatrix{Sq^k_{*}: H_{\bullet}((K(\mathbb F_2, 1))^{\times m}; \mathbb F_2)\ar[r] & H_{\bullet-k}((K(\mathbb F_2, 1))^{\times m}; \mathbb F_2)}, $a^{(\bullet)}_j\longmapsto \binom{\bullet-k}{k}a^{(\bullet-k)}_j$ which are induced by \xymatrix{Sq^k: H^{\bullet}((K(\mathbb F_2, 1))^{\times m}; \mathbb F_2)\ar[r] & H^{\bullet+k}((K(\mathbb F_2, 1))^{\times m}; \mathbb F_2)}; this action is equivalent to the right action of $\mathscr A$ on $(\mathcal P_m)^{*}$ by means of the opposite algebra of $\mathscr A.$ The dual viewpoint is naturally expressed in terms of the dual Steenrod algebra and its structural properties; see also~\cite{Turgay3} for a related structural embedding perspective. 

\medskip

At odd primes, related annihilated-element questions have also been studied in the homology of powers of infinite complex projective space. In his Ph.D. thesis \cite{AlHajjaj}, Al-Hajjaj investigated the odd-primary analogue of the annihilated problem, proving nonvanishing results for the spaces $M_n(k)$ below the first vanishing degree, computing dimensions of $M_n(3)$ in a substantial range, and studying the associated subring of lines $L^*(k)$. While these results lie in a different prime and geometric setting, they illustrate the broader relevance of annihilated-element methods in problems adjacent to the hit problem and algebraic transfer.

\medskip

It is well known that the determination of the cohit space $Q^{\otimes m}$ in each positive degree $n$ is considered important in comprehending the $E_{2}$-term of the Adams Spectral Sequence (Adams SS, for short), ${\rm Ext}_{\mathscr A}^{m, m+n}(\mathbb F_2, \mathbb F_2)$, through the $m$-th transfer homomorphism,
\[
\xymatrix{
 Tr_m^{\mathscr A}:  (\mathbb F_2\otimes_{G(m)} P_{\mathscr A}(\mathcal P_m)^{*})_n\ar[r] &{\rm Ext}_{\mathscr A}^{m, m+n}(\mathbb F_2, \mathbb F_2).
}
\]
Here the domain of the transfer map $Tr_m^{\mathscr A}$ is dual to the $G(m)$-invariants space $(Q^{\otimes m}_n)^{G(m)},$ where $m$ corresponds to the Adams filtration and $n$ represents the stem. Note that the $G(m)$-coinvariants $(\mathbb F_2\otimes_{G(m)} P_{\mathscr A}(\mathcal P_m)^{*})_n$ form a bigraded algebra and the algebraic transfers $Tr_*^{\mathscr A}$ yield a morphism of bigraded algebras with values in ${\rm Ext}_{\mathscr A}^{*,*}(\mathbb F_2, \mathbb F_2)$ \cite{W.S1}. The homomorphism $ Tr_m^{\mathscr A},$ constructed by Singer \cite{W.S1}, has garnered significant attention from numerous mathematicians. It is an isomorphism for $m\leq 3$: see \cite{W.S1} for $m = 1,\, 2$ and \cite{J.B} for $m = 3.$ While substantial progress has also been made for $m=4$ in certain families of degrees (see, e.g., \cite{Phuc26c} and related references), the rank $m = 5$ is widely viewed as the first genuinely difficult case where both the structure of $Q^{\otimes m}$ and the invariant theory become significantly more intricate. Singer also conjectured in \cite{W.S1} that \textit{$Tr_m^{\mathscr A}$ is a one-to-one homomorphism for arbitrary $m$}. Approaching this conjecture is highly nontrivial, chiefly because both the domain and codomain of $Tr_m^{\mathscr A}$ are difficult to compute. The low-dimensional cases $m\in \{1,\, 2,\, 3\}$ were completely resolved in the early 1990s, as discussed above. After nearly four decades, our works in \cite{Phuc26a, Phuc26b, Phuc26c} successfully extended these results to $m=4.$ However, our most recent investigation \cite{Phuc25d} demonstrates that the conjecture fails at $m = 6.$ 

\section{Outline of Our Contributions}\label{s2}

\textbf{Summary of contributions.} The following are the key contributions of our work.

\begin{itemize}
\item[(\textbf{A})] As presented in Sect.\ref{s1}, the hit problem for the $\mathscr A$-module $\mathcal P_m$ is still open for any $m \geq 5$ and positive degrees $n.$ Additionally, understanding the cohit modules $Q^{\otimes m}$ and the domain of the Singer transfer for larger $m$ is very much at the research frontier; calculations rapidly become exceedingly difficult, if not impossible, even when computer assisted. Hence motivated by these contexts and numerous previous interesting results, the first principal objective that we strive to achieve in this paper is to extend our previous works \cite{D.P3, D.P10-2} on the hit problem for $\mathcal{P}_5$ to the generic degree $n_s := d(2^{s}-1) + k\cdot 2^{s},$ where $d = 5,\, k  =18$ and $s$ is an arbitrary non-negative integer. The results are used to describe the representations of $G(5)$ over $\mathbb F_2.$ As a result, the algebraic transfer is an isomorphism in bidegree $(5, 5+n_s)$ for all $s\geq 0.$

\item[(\textbf{B})] As a direct result, a localized form of the Kameko conjecture that focuses on the dimension of $Q^{\otimes m}$ connected to parameter vectors has been demonstrated to be true for the specific case of $m=5$ and degree $n_s.$ Moreover, by using previous results in our work \cite{D.P10-2} and those of other authors, we show that the localized variation of the Kameko conjecture remains valid for all $m\geq 1$ in positive degrees $\leq 12$.
\end{itemize}

\medskip

We also wish to stress how the present paper differs from our previous works. For example, in \cite{D.P3}, we studied a different generic family of degrees for $\mathcal P_5$, namely $5(2^s-1)+6\cdot 2^s$, whereas the present paper treats the new family $5(2^s-1)+18\cdot 2^s.$
The present computations are not a reformulation of \cite{D.P3}, but require a different admissibility analysis and a new determination of the relevant $G(5)$-module structure. We further note that the results of the present paper were independently checked by computations in \texttt{OSCAR} and \texttt{SageMath}, using algorithms developed in our recent works \cite{Phuc25d,Phuc25e,Phuc25f,Phuc25g,Phuc25h}. The construction of these algorithms draws on several ingredients, including matrix and linear-algebraic criteria for detecting hit elements, computations in the lambda algebra and the use of Adams relations, invariant-theoretic methods for determining $G(m)$-invariant spaces, graph-theoretic techniques such as weight interaction graphs and global cluster analysis, as well as methods from modular representation theory. Taken together, these independent computational verifications provide further evidence for the accuracy of the dimension calculations and invariant-theoretic results established here.

\medskip
\noindent\textbf{Significance of the main theorems.}
Theorems~\ref{dlc1}, \ref{dlc2} and \ref{dlc3} below provide explicit and verifiable computations in rank~$5$ for an infinite generic family of degrees, where both the hit problem and the $G(5)$-invariant analysis are substantially more delicate than in the known cases $m\le 4$.
More precisely:
\begin{itemize}
\item[(i)] Theorem~\ref{dlc1} determines the dimension of the kernel of the Kameko morphism in degree $n_1$, and hence yields a uniform dimension formula for $Q^{\otimes 5}_{n_s}$ for all $s\ge 1$.
\item[(ii)] Theorem~\ref{dlc2} produces an explicit $G(5)$-invariant generator in the relevant degrees; this directly implies that the fifth algebraic transfer is an isomorphism in bidegree $(5,5+n_s)$ for every $s\ge 0$.
\item[(iii)] Theorem~\ref{dlc3} complements these results by confirming a localized variation of Kameko's conjecture in low degrees for all $m\ge 1$, providing further evidence for the effectiveness of the parameter-vector filtration.
\end{itemize}

\medskip

\textbf{The Kameko morphism.} Before detailing our chief results, we would like to mention the Kameko morphism and its related aspects, which are among the important tools in establishing our results. The Kameko morphism \cite{M.K} is of the form:
$$ \begin{array}{llll}
(\widetilde {Sq^0_*})_{(m, n)}: &Q^{\otimes m}_{n}  &\myto Q^{\otimes m}_{\frac{n-m}{2}}\ (n - m\, even)\\
&\mbox{[}x_1^{a_1}x_2^{a_2}\ldots x_m^{a_m}\mbox{]}&\longmapsto \left\{\begin{array}{ll}
\mbox{[}x_1^{\frac{a_1-1}{2}}x_2^{\frac{a_2-1}{2}}\ldots x_m^{\frac{a_m-1}{2}}\mbox{]}& \text{if $a_j$ odd, $j = 1, 2, \ldots, m$},\\[1mm]
0& \text{otherwise}.
\end{array}\right.
\end{array}$$ 
It shows that $(\widetilde {Sq^0_*})_{(m, n)}$ is always an epimorphism of $\mathbb F_2G(m)$-modules. This characteristic suggests a descent approach for tackling the hit problem. Kameko utilized this extensively in his computation of $Q^{\otimes m}$ for $m = 3$.
Due to the complexity of calculating $Q^{\otimes m}$, two useful approaches would be to either measure its global dimension by taking the maximum dimension of $Q_n^{\otimes m}$ for all positive values of $n$, or to examine $Q^{\otimes m}$ in small degrees of $n$, and then generalize the results obtained. In addition, there is also a homotopical approach, as seen in \cite{Anick}, but it will not be discussed in this paper. For related problems, readers can refer, for example, to \cite{Giambalvo2}. Besides, the well-established results of Wood \cite{R.W} (also referred to as Peterson's conjecture \cite{F.P}) and Kameko \cite{M.K} (Theorem \ref{dlWK}) can help to streamline the calculation process for $Q^{\otimes m}$ in specific degrees.

\begin{therm}[see \cite{R.W}, \cite{M.K}]\label{dlWK}
Given the arithmetic function 
$$ \begin{array}{llll}
 \beta: \mathbb N&\myto &\mathbb N\\
\ \ \ \ \  n&\longmapsto & {\rm min}\big\{k\in \mathbb N:\ n = \sum_{1\leq i\leq k}(2^{d_i}-1),\, d_i > 0\big\}= {\rm min}\big\{k\in \mathbb N:\ \alpha(n+k)\leq k\big\},\\
\end{array}$$
where the function $\alpha$ counts the number of ones in the binary expansion of its argument.  
\begin{itemize}
\item[(I)] $Q_n^{\otimes m}$ is trivial if $\beta(n) > m.$ Consequently $Q^{\otimes m}$ is trivial in degrees $n$ unless $n$ is of the form $n = (2^{d_1}-1) + (2^{d_2}-1)+ \cdots + (2^{d_m}-1)$ with $d_1\geq d_2\geq \cdots\geq d_m \geq 0$  (see Wood \cite{R.W}).
\item[(II)] The map $(\widetilde {Sq^0_*})_{(m, n)}$ is an isomorphism of $\mathbb F_2G(m)$-modules if $\beta(n) = m$ (see Kameko \cite{M.K}).
\end{itemize}
\end{therm}
This result in point $(I)$ serves as the foundation for ongoing research focused on determining the "hit" elements and the conditions that govern them.  Since then, numerous authors have attempted to refine and extend Wood's approach, including Meyer and Silverman \cite{Meyer}, Monks \cite{Monks}, Silverman \cite{Silverman1, Silverman2}, Singer \cite{W.S2}, etc. Particularly, the application of Theorem \ref{dlWK} can reduce the calculation of $Q^{\otimes m}$ in positive degrees $n$ for cases where $\beta(n) \leq m.$ In these cases, $n$ is of the "generic" form $d(2^{s}-1) + k\cdot 2^{s}$, where $d$, $s$, and $k$ are non-negative integers, and $0\leq \beta(k)< d \leq m$ (see \cite{P.S1}).

\medskip

{\bf Statement of main results and applications.} As aforementioned, our first goal in the present study is to extend our previous work \cite{D.P3} on the hit problem of five variables to the generic degree $n_s:=d(2^{s}-1) + k\cdot 2^{s}$ with $d = 5,\, k  =18$ and $s$ is a positive integer. One checks that $\beta(k) = \beta(18) = 2 < d = 5.$ Moreover, the degree $n_s$ can be represented as $$n_s = (2^{s+4} -1)+ (2^{s+2}-1) + (2^{s+1}-1) + (2^{s-1}-1) + (2^{s-1} - 1).$$
So, $\beta(n_s) = 5$ for any $s > 1.$ Thus, in view of Theorem \ref{dlWK}(II), the iterated Kameko morphism 
\xymatrix{
((\widetilde {Sq^0_*})_{(5, n_s)})^{s-1}:  Q^{\otimes 5}_{n_s} \ar[r] &Q^{\otimes 5}_{n_1}
}
is an isomorphism of $\mathbb F_2G(5)$-modules for all $s\geq 1.$ Thus, it is only necessary to calculate the dimension of $Q^{\otimes 5}_{n_s}$ for $s = 0$ and $s = 1.$ The case $s = 0$ was completely solved in \cite{D.P3}, where the dimension was determined to be $730.$ For $s=1,$ the Kameko morphism 
\xymatrix{
(\widetilde {Sq^0_*})_{(5, n_1)}: Q^{\otimes 5}_{n_1}  \ar[r] & Q^{\otimes 5}_{n_0}
}
 is an epimorphism, which leads to an isomorphism $$Q^{\otimes 5}_{n_1}\cong {\rm Ker}((\widetilde {Sq^0_*})_{(5, n_1)})\bigoplus {\rm Im}((\widetilde {Sq^0_*})_{(5, n_1)}) \cong {\rm Ker}((\widetilde {Sq^0_*})_{(5, n_1)})\bigoplus Q^{\otimes 5}_{n_0}.$$ 
Moreover, by using the monomorphism 
\xymatrix{
\varphi: \mathcal P_5\ar[r] &\mathcal P_5,\, u\longmapsto x_1x_2\ldots x_5u^{2},
} 
one derives

\begin{corls}\label{hq0}
We have an isomorphism of $\mathbb F_2$-vector spaces:
$$ Q^{\otimes 5}_{n_0} \cong \big\langle \big\{[\varphi(u)]\in Q^{\otimes 5}_{n_1}:\, \mbox{$u$ belongs to a minimal set of $\mathscr A$-generators for $\mathcal P_5$ in degree $n_0$}\big\} \big\rangle \subset Q^{\otimes 5}_{n_1}.$$
\end{corls}

By the above data, in order to find the dimension of $Q^{\otimes 5}_{n_1},$ we need to compute the dimension of ${\rm Ker}((\widetilde {Sq^0_*})_{(5, n_1)}),$ hence leading to our first main result. More precisely, the following theorem holds.

\begin{therm}\label{dlc1}
The kernel of the Kameko $(\widetilde {Sq^0_*})_{(5, n_1)}$ is an $\mathbb F_2$-vector space of dimension $1900.$
\end{therm}

From the isomorphism $Q^{\otimes 5}_{n_s}\cong Q^{\otimes 5}_{n_1}$ for all $s\geq 1$ and the relation $\dim Q^{\otimes 5}_{n_1} = 730+ \dim {\rm Ker}((\widetilde {Sq^0_*})_{(5, n_1)}),$ in conjunction with Theorem \ref{dlc1}, we obtain

\begin{corls}\label{hq1}
The cohit space $Q^{\otimes 5}_{n_s}$ has dimension $2630$ for every positive integer $s.$
\end{corls}

\begin{nx}
As shown above, $Q^{\otimes 5}_{n_s} \cong Q^{\otimes 5}_{n_1}$ as $G(5)$-modules, for all $s\geq 1.$ So, by dualizing the invariant spaces, we get
$$ (\mathbb F_2\otimes_{G(5)} P_{\mathscr A}(\mathcal P_5)^{*})_{n_s}\cong (\mathbb F_2\otimes_{G(5)} P_{\mathscr A}(\mathcal P_5)^{*})_{n_1},$$
for any $s\geq 1.$ Therefore, as an application of a previous result in our work \cite{D.P3} on the space $Q^{\otimes 5}_{n_0}$ and Theorem \ref{dlc1}, we only need to investigate the behavior of the fifth Singer transfer, $Tr_5^{\mathscr A},$ at degree $n_s$ for $s = 0$ and $s = 1.$
\end{nx}

From this remark, we have the following technical theorem, which is the second main result of the paper.

\begin{therm}\label{dlc2}
We have 
$$ (Q^{\otimes 5}_{n_0})^{G(5)} = \langle [\widetilde{\xi}_{n_0}] \rangle, \ \ \mbox{and}\ \ (Q^{\otimes 5}_{n_1})^{G(5)} = \langle [\varphi(\widetilde{\xi}_{n_0}) + \widetilde{\xi}_{n_1}] \rangle,$$
where $\varphi$ is the up Kameko map 
\xymatrix{
\mathcal P_5\ar[r]& \mathcal P_5,\, u\longmapsto x_1x_2\ldots x_5u^{2},
} 
and the polynomials $\widetilde{\xi}_{n_0}$ and $\widetilde{\xi}_{n_1}$ are given explicitly in Appendix~\ref{ap5.0}.
\end{therm}

\begin{nx}\label{nxta}

\begin{itemize}
\item[i)] As $(\mathbb F_2\otimes_{G(5)} P_{\mathscr A}(\mathcal P_5)^{*})_{n_s}$ is dual to $(Q^{\otimes 5}_{n_s})^{G(5)}$ for all $s \geq 0$, we have
$$ \dim (\mathbb F_2\otimes_{G(5)} P_{\mathscr A}(\mathcal P_5)^{*})_{n_s} = \dim (Q^{\otimes 5}_{n_s})^{G(5)} = 1,$$
for all $s\geq 0.$

\item[ii)] Using the results of Chen \cite{T.C} and Lin \cite{W.L}, we obtain
$$ {\rm Ext}_{\mathscr A}^{4, 21.2^{s}}(\mathbb F_2, \mathbb F_2) = \langle h_{s}f_s \rangle,\ \ h_{s}f_s\neq 0,\ \ \forall s\geq 0.$$
On the other hand, we note that by Singer \cite{W.S1}, $h_{s}\in {\rm Im}(Tr_1^{\mathscr A}),$ and by Nam \cite{Nam}, $f_{s}\in {\rm Im}(Tr_4^{\mathscr A}),$ for all $s.$ Moreover, since the ``total'' transfer $\bigoplus_{m\geq 0}Tr_m^{\mathscr A}$ is known to be an algebra homomorphism, it follows that $h_{s}f_s\in {\rm Im}(Tr_5^{\mathscr A})$ for all $s.$
\end{itemize}
\end{nx}

The results in Theorems \ref{dlc1} and \ref{dlc2} have been computationally verified by our algorithms implemented in \texttt{SageMath} and \texttt{OSCAR} \cite{Phuc25d, Phuc25e}. Detailed computational output is provided in \textbf{Appendices \ref{ap5.3} and \ref{ap5.4}}.

\subsection*{Computational verification and reproducibility}
Although the main theorems are proved by theoretical arguments using the Kameko morphism, admissibility criteria, and $G(5)$-invariant theory, we additionally verified several intermediate steps and final dimension/invariant computations by independent computer algebra implementations.

\smallskip
\noindent\textbf{Theoretical proofs.}
The proofs of Theorems~\ref{dlc1}, \ref{dlc2} and \ref{dlc3} rely on explicit $\mathscr A$-module arguments, filtration by parameter vectors, and invariance checks via the generators of $G(5)$.

\smallskip
\noindent\textbf{Computational verification.}
The following computations were independently verified by programs implemented in \texttt{SageMath} and \texttt{OSCAR}:
(i) bases and dimensions of $Q^{\otimes 5}_{n_0}$ and $Q^{\otimes 5}_{n_1}$;
(ii) the dimension of the kernel in Theorem~\ref{dlc1};
(iii) invariant generators in Theorem~\ref{dlc2}.
The corresponding outputs and logs are deposited at Zenodo (Appendices~\ref{ap5.3}--\ref{ap5.4}).

\medskip

As an immediate consequence of Theorem \ref{dlc2} and Remark \ref{nxta}, we get the following.

\begin{corls}
The fifth algebraic transfer 
\[\xymatrix{
 Tr_5^{\mathscr A}:  (\mathbb F_2\otimes_{G(5)} P_{\mathscr A}(\mathcal P_5)^{*})_{n_s}\ar[r]& {\rm Ext}_{\mathscr A}^{5, 5+n_s}(\mathbb F_2, \mathbb F_2)
}\]
 is an isomorphism for arbitrary $s\geq 0.$ 
\end{corls}

As a complementary result to our primary focus on the generic degree $n_s$ for $m=5$, we also give a broader validation of the localized variation of Kameko's conjecture. This is Conjecture \ref{gtK}, which we mentioned in Sect.\ \ref{s3}. We show that this conjecture holds true for the general case of all $m \ge 1$ in low degrees. This is our fourth main result.

\begin{therm}\label{dlc3}
Conjecture \ref{gtK} holds true for all $m\geq 1$ and parameter vectors of degrees not exceeding $12.$
\end{therm}

\smallskip
\noindent\textbf{Organization of the paper.}
Sect.~\ref{s3} collects the algebraic-topological background used throughout the paper, including the Steenrod action on $\mathcal P_m$, the Kameko morphism, and the parameter-vector filtration.
Sect.~\ref{s4} proves Theorems~\ref{dlc1}--\ref{dlc3}: first the kernel computation for $(\widetilde{Sq^0_*})_{(5,n_1)}$, then the $G(5)$-invariant analysis leading to the isomorphism of the fifth transfer, and finally the low-degree verification of the localized conjecture.
Sect.~\ref{s5} summarizes the main conclusions and outlines further directions.
Finally, Sect.~\ref{s6} is an appendix containing explicit bases and links to the computational outputs (Zenodo) supporting the computations in Sect.~\ref{s4}.

\section{Preliminaries}\label{s3}

This preliminary section starts with a brief overview of the Steenrod algebra over $\mathbb F_2$ and concludes with a concise summary of some well-known homomorphisms in the work of the author and Nguyen Sum \cite{P.S1}. The majority of the results presented in this section serve as basic building blocks that will be utilized throughout the paper.

\subsection{The 2-Primary Steenrod Algebra and Related Applications}\label{sub3-1}
The 2-primary Steenrod algebra, $\mathscr A,$ is defined as the algebra of stable cohomology operations from $\mathbb F_2$ cohomology to itself, generated by the Steenrod squares
\xymatrix{
Sq^i: H^{\bullet}(\mathscr U; \mathbb{F}_2) \ar[r] &H^{\bullet+i}(\mathscr U; \mathbb{F}_2)} in grading $i\geq 0,$ subject to the Adem relations and the condition $Sq^{0} = 1$ (for more information, refer to \cite{S.E}). Here $H^{\bullet}(\mathscr U; \mathbb{F}_2)$ denotes the $\mathbb F_2$-singular cohomology group of a topological space $\mathscr U.$  One can define the cup product $\smile,$ which takes the form 
\xymatrix{
H^{n}(\mathscr U; \mathbb{F}_2) \times H^{i}(\mathscr U; \mathbb{F}_2)\ar[r]& H^{n+i}(\mathscr U; \mathbb{F}_2),}\\
$(u, v)\longmapsto u\smile v.$ This cup product  also gives a multiplication on the graded cohomology ring $H^*(\mathscr U; \mathbb{F}_2) = \bigoplus_{n\geq 0}H^{n}(\mathscr U; \mathbb F_2).$  It is worth noticing that the structure of $H^*(\mathscr U; \mathbb{F}_2)$ is not only as a graded commutative $\mathbb F_2$-algebra, but also as an $\mathscr A$-module. The $\mathscr A$-module structure on $H^{*}(\mathscr U; \mathbb F_2)$ is often found to yield useful insights into the nature of $\mathscr U$ in many cases. As evidenced by the example below.

\begin{ex}\label{vd311}
The spaces $\mathbb{C}P^4/\mathbb{C}P^2$ and $\mathbb{S}^6\vee \mathbb{S}^8$ have cohomology rings that agree as graded commutative $\mathbb F_2$-algebras, but are different as modules over $\mathscr A.$ Hence
\[
\mathbb{C}P^4/\mathbb{C}P^2\not\simeq \mathbb{S}^6\vee \mathbb{S}^8.
\]
To the best of our knowledge, a direct proof of Example \ref{vd311} does not seem to be explicitly available in the literature. Accordingly, in order to make the paper self-contained, we include a detailed proof here. Our argument is based on the standard CW decomposition of complex projective spaces, the induced quotient map in cohomology, and the naturality of the Steenrod operations.

Recall that $\mathbb{C}P^n$ has the standard CW decomposition
\[
\mathbb{C}P^n=e^0\cup e^2\cup e^4\cup \cdots \cup e^{2n}
\]
with exactly one cell in each even dimension; see, for instance, \cite[Example 0.6]{Hat}. Therefore
\[
\mathbb{C}P^4/\mathbb{C}P^2 \simeq e^0\cup e^6\cup e^8,
\]
so
\[
H^k(\mathbb{C}P^4/\mathbb{C}P^2;\mathbb F_2)\cong
\begin{cases}
\mathbb F_2,& k=0,6,8,\\
0,& \text{otherwise}.
\end{cases}
\]
On the other hand, by the standard decomposition of reduced cohomology on wedges,
\[
\widetilde H^*(\mathbb S^6\vee \mathbb S^8;\mathbb F_2)
\cong
\widetilde H^*(\mathbb S^6;\mathbb F_2)\oplus
\widetilde H^*(\mathbb S^8;\mathbb F_2),
\]
and hence
\[
H^k(\mathbb S^6\vee \mathbb S^8;\mathbb F_2)\cong
\begin{cases}
\mathbb F_2,& k=0,6,8,\\
0,& \text{otherwise}.
\end{cases}
\]
Since there are no nonzero products for dimensional reasons, these two spaces have isomorphic cohomology rings as graded commutative $\mathbb F_2$-algebras.

We now compare their $\mathscr A$-module structures. Let
\[
q:\mathbb{C}P^4\longrightarrow \mathbb{C}P^4/\mathbb{C}P^2
\]
be the quotient map, and let $x\in H^2(\mathbb{C}P^4;\mathbb F_2)$ be the canonical generator. Since
\[
H^*(\mathbb{C}P^4;\mathbb F_2)\cong \mathbb F_2[x]/\langle x^5\rangle,
\]
we only need to analyze the induced maps in degrees $6$ and $8$. By the standard identification
\[
\widetilde H^*(\mathbb{C}P^4/\mathbb{C}P^2;\mathbb F_2)
\cong
H^*(\mathbb{C}P^4,\mathbb{C}P^2;\mathbb F_2),
\]
the map $q^*$ agrees with the canonical map
\[
H^*(\mathbb{C}P^4,\mathbb{C}P^2;\mathbb F_2)\longrightarrow H^*(\mathbb{C}P^4;\mathbb F_2)
\]
in the long exact sequence of the pair $(\mathbb{C}P^4,\mathbb{C}P^2)$. Since
\[
H^6(\mathbb{C}P^2;\mathbb F_2)=H^7(\mathbb{C}P^2;\mathbb F_2)=H^8(\mathbb{C}P^2;\mathbb F_2)=0,
\]
it follows that the induced maps
\[
q^*:H^6(\mathbb{C}P^4/\mathbb{C}P^2;\mathbb F_2)\longrightarrow H^6(\mathbb{C}P^4;\mathbb F_2),
\qquad
q^*:H^8(\mathbb{C}P^4/\mathbb{C}P^2;\mathbb F_2)\longrightarrow H^8(\mathbb{C}P^4;\mathbb F_2)
\]
are isomorphisms. Choose
\[
a\in H^6(\mathbb{C}P^4/\mathbb{C}P^2;\mathbb F_2),
\qquad
b\in H^8(\mathbb{C}P^4/\mathbb{C}P^2;\mathbb F_2)
\]
such that
\[
q^*(a)=x^3,\qquad q^*(b)=x^4.
\]
Using naturality and the standard formula
\[
Sq^{2r}(x^k)=\binom{k}{r}x^{k+r}
\qquad
\text{in }H^*(\mathbb{C}P^n;\mathbb F_2),
\]
we obtain
\[
q^*(Sq^2a)=Sq^2(q^*a)=Sq^2(x^3)=\binom31 x^4=x^4=q^*(b).
\]
Since $q^*$ is an isomorphism in degree $8$, it follows that
\[
Sq^2(a)=b\neq 0.
\]
Thus the homomorphism
\[
\xymatrix{
Sq^2:H^6(\mathbb{C}P^4/\mathbb{C}P^2;\mathbb F_2)\ar[r]&
H^8(\mathbb{C}P^4/\mathbb{C}P^2;\mathbb F_2)
}
\]
is nontrivial.

By contrast, the corresponding Steenrod square on $\mathbb S^6\vee \mathbb S^8$ is trivial. Indeed, let
\[
i:\mathbb S^6\hookrightarrow \mathbb S^6\vee \mathbb S^8
\]
be the canonical inclusion. Then
\[
\widetilde H^*(\mathbb S^6\vee \mathbb S^8;\mathbb F_2)
\cong
\widetilde H^*(\mathbb S^6;\mathbb F_2)\oplus
\widetilde H^*(\mathbb S^8;\mathbb F_2).
\]
Moreover, more generally, for every sphere $\mathbb S^m$ and every coefficient field $\Bbbk$,
\[
H^q(\mathbb S^m;\Bbbk)\cong
\begin{cases}
\Bbbk,& q=0,m,\\
0,& \text{otherwise}.
\end{cases}
\]
Hence
\[
H^6(\mathbb S^6\vee \mathbb S^8;\mathbb F_2)\cong H^6(\mathbb S^6;\mathbb F_2)\cong \mathbb F_2,
\qquad
H^8(\mathbb S^6;\mathbb F_2)=0.
\]
Therefore the induced map
\[
i^*:H^6(\mathbb S^6\vee \mathbb S^8;\mathbb F_2)\longrightarrow H^6(\mathbb S^6;\mathbb F_2)
\]
is an isomorphism. Furthermore, by the naturality of the Steenrod squares, for every continuous map
$f:X\to Y$ and every $r\geq 0,$ one has
\[
f^*\circ Sq^r = Sq^r\circ f^*.
\]
Applying this to the inclusion $i$ and $r=2$, we obtain the commutative diagram
\[
\xymatrix{
H^6(\mathbb S^6\vee \mathbb S^8;\mathbb F_2)\ar[r]^{Sq^2}\ar[d]_{i^*}^{\cong} &
H^8(\mathbb S^6\vee \mathbb S^8;\mathbb F_2)\ar[d]^{i^*} \\
H^6(\mathbb S^6;\mathbb F_2)\ar[r]_{Sq^2} &
H^8(\mathbb S^6;\mathbb F_2)=0.
}
\]
It follows that
\[
Sq^2:H^6(\mathbb S^6\vee \mathbb S^8;\mathbb F_2)\longrightarrow
H^8(\mathbb S^6\vee \mathbb S^8;\mathbb F_2)
\]
is the zero homomorphism.

Hence $H^*(\mathbb{C}P^4/\mathbb{C}P^2;\mathbb F_2)$ and
$H^*(\mathbb S^6\vee \mathbb S^8;\mathbb F_2)$ are not isomorphic as $\mathscr A$-modules. Therefore
\[
\mathbb{C}P^4/\mathbb{C}P^2\not\simeq \mathbb S^6\vee \mathbb S^8.
\]
\end{ex}

\medskip
\begin{rem}\label{rem312}

The above example is not specific to the case $n=4$, but in fact reflects a more general phenomenon for the quotients
\[
X_n:=\mathbb CP^n/\mathbb CP^{\,n-2}.
\]
To the best of our knowledge, we are not aware of a direct reference for the precise homotopy classification stated below. For this reason, and in order to keep the paper self-contained, we record the following generalization together with a brief proof sketch. The argument combines the standard CW structure of $\mathbb CP^n$, the fact that $X_n$ is a two-cell complex, and the behavior of the Steenrod square $Sq^2$ on its top two nonzero cohomology groups.

\smallskip
\noindent
{\it Let $n\ge 3$ and set $X_n:=\mathbb CP^n/\mathbb CP^{\,n-2}$. Then
\[
X_n\simeq
\begin{cases}
\mathbb S^{2n-2}\vee \mathbb S^{2n},& \text{if $n$ is odd},\\[1mm]
\Sigma^{2n-4}\mathbb CP^2,& \text{if $n$ is even}.
\end{cases}
\]
In particular,
\[
\mathbb CP^n/\mathbb CP^{\,n-2}\not\simeq \mathbb S^{2n-2}\vee \mathbb S^{2n}
\quad\text{for $n$ even}.
\]}

\smallskip
\noindent
We sketch a proof. By the CW structure of $\mathbb CP^n$, the quotient $X_n$ is a two-cell complex:
\[
X_n\simeq \mathbb S^{2n-2}\cup_f e^{2n},
\]
where
\[
f\in \pi_{2n-1}(\mathbb S^{2n-2}).
\]
Since $2n-2\ge 4$ for $n\ge 3$, we have
\[
\pi_{2n-1}(\mathbb S^{2n-2})\cong \mathbb Z/2,
\]
generated by the iterated suspension $\Sigma^{2n-4}\eta$ of the Hopf map
$\eta:\mathbb S^3\to \mathbb S^2$; see \cite[Corollary 4J.4]{Hat}. Therefore there are only two possibilities:
\[
f\simeq *,
\qquad\text{or}\qquad
f\simeq \Sigma^{2n-4}\eta.
\]
Accordingly,
\[
X_n\simeq \mathbb S^{2n-2}\vee \mathbb S^{2n}
\qquad\text{or}\qquad
X_n\simeq \Sigma^{2n-4}\mathbb CP^2,
\]
because $\mathbb CP^2$ is the mapping cone of $\eta$.

Now let
\[
q:\mathbb CP^n\longrightarrow X_n=\mathbb CP^n/\mathbb CP^{\,n-2}
\]
be the quotient map, and let $x\in H^2(\mathbb CP^n;\mathbb F_2)$ be the canonical generator. Choose
\[
a\in H^{2n-2}(X_n;\mathbb F_2),\qquad
b\in H^{2n}(X_n;\mathbb F_2)
\]
so that
\[
q^*(a)=x^{n-1},\qquad q^*(b)=x^n.
\]
By naturality,
\[
q^*(Sq^2a)=Sq^2(x^{n-1})=\binom{n-1}{1}x^n.
\]
Since all cohomology groups are taken with $\mathbb F_2$-coefficients, we have
\[
\binom{n-1}{1}x^n=((n-1)\bmod 2)\,x^n=q^*\bigl(((n-1)\bmod 2)\,b\bigr).
\]
Hence
\[
Sq^2(a)=((n-1)\bmod 2)\,b=
\begin{cases}
b,& \text{if $n$ is even},\\
0,& \text{if $n$ is odd}.
\end{cases}
\]

If $n$ is even, then $Sq^2(a)=b\neq 0$. But on the wedge
$\mathbb S^{2n-2}\vee \mathbb S^{2n}$, the Steenrod square
\[
Sq^2:H^{2n-2}(\mathbb S^{2n-2}\vee \mathbb S^{2n};\mathbb F_2)\longrightarrow
H^{2n}(\mathbb S^{2n-2}\vee \mathbb S^{2n};\mathbb F_2)
\]
is zero, exactly as in Example \ref{vd311}. Therefore $X_n$ cannot be homotopy equivalent to
$\mathbb S^{2n-2}\vee \mathbb S^{2n}$ when $n$ is even.

If $n$ is odd, then $Sq^2(a)=0$. On the other hand, suppose that
\[
X_n\simeq \Sigma^{2n-4}\mathbb CP^2.
\]
Let $\bar x\in H^2(\mathbb CP^2;\mathbb F_2)$ be the canonical generator, and let
\[
u\in \widetilde H^{2n-2}(\Sigma^{2n-4}\mathbb CP^2;\mathbb F_2),
\qquad
v\in \widetilde H^{2n}(\Sigma^{2n-4}\mathbb CP^2;\mathbb F_2)
\]
be the classes corresponding, under the $(2n-4)$-fold suspension isomorphism in reduced cohomology, to $\bar x$ and $\bar x^2$, respectively. Since
\[
Sq^2(\bar x)=\bar x^2\neq 0
\]
in $H^*(\mathbb CP^2;\mathbb F_2)$, and reduced Steenrod squares commute with the suspension isomorphism, it follows that
\[
Sq^2(u)=v\neq 0.
\]
This contradicts $Sq^2(a)=0$. Hence the nontrivial attaching map is impossible, so necessarily
\[
X_n\simeq \mathbb S^{2n-2}\vee \mathbb S^{2n}.
\]
This proves the claim.
\end{rem}

\subsection{Some Definitions and Related Results}\label{sub3-2}
 
In this subsection, we hereby present fundamental facts concerning the classical hit problem for the Steenrod algebra, which has an extensive historical background. We also review some essential results whose proofs are widely available in the existing literature. (In addition, we would like to direct the reader to the monographs written by Walker and Wood \cite{W.W, W.W2}  for a considerable amount of information on this subject.)

We would like to reiterate that the graded polynomial algebra $\mathcal P_m=\mathbb F_2[x_1, \ldots, x_m] = \bigoplus_{n\geq 0}(\mathcal P_m)_n$ realizes the (mod-2) cohomology of the product of $m$ copies of the Eilenberg-MacLan complex $K(\mathbb F_2, 1) = \mathbb RP^{\infty}=\mbox{colim}_n\mathbb RP^{n}.$ Here, $\mathbb RP^{\infty}$ is defined as the infinite-dimensional real projective space, which can be constructed as the sequential colimit of $\mathbb RP^{n}$ with the canonical inclusion maps. The grading is by the homogeneous components $(\mathcal P_m)_n = H^{n}((\mathbb RP^{\infty})^{\times m}; \mathbb F_2)$ of degree $n$ in the $m$ variables $x_1, \ldots, x_m$  of grading 1. In other point of view,  the algebra $\mathcal P_m$ can be identified with the (mod-2) cohomology of the product of $m$ copies of the infinite rank 1 Grassmannian $G_1(\mathbb R^{\infty}).$ This $\mathcal P_m$ is considered as an $\mathscr A[G(m)]$-module.  Dual to the hit problem is the "$\overline{\mathscr A}$-annihilated problem", which involves finding a basis for the space of $\overline{\mathscr A}$-annihilateds, $P_{\mathscr A}(\mathcal P_m)^{*} = \bigg\{u\in (\mathcal P_m)^{*}|\, Sq^k_*(u) = 0,\, \forall k \geq 1\bigg\}.$ Studying the hit problem and its dual is important because they lead to an understanding of the Singer transfer $Tr_{m}^{\mathscr A}$, as mentioned in Sect. \ref{s1}.

\medskip

{\bf $\bullet$ \textit{Parameter and exponent vector}}. If $n\geq 0$ is an integer, we may represent it by its binary expansion $n = \sum_{t\geq 0}\alpha_t(n)2^t$. For a monomial $X\in \mathcal P_m$, that is, $X = x_1^{u_1}x_2^{u_2}\ldots x_m^{u_m}$, we define $\mathsf{Param}_i(X)$ to be an integer $\sum_{1\leq j\leq m}\alpha_{i-1}(u_j)$, where $i\geq 1$. We also associate two sequences with $X$: $u(X) = (u_1, u_2, \ldots, u_m)$ and $ \mathsf{Param}(X) = (\mathsf{Param}_1(X), \ldots, \mathsf{Param}_i(X),\ldots)$. These sequences are referred to as the {\it exponent vector} and the \textit{parameter vector} of $X$, respectively. Let $\mathsf{Param} = (\mathsf{Param}_1, \ldots, \mathsf{Param}_i,\ldots)$ be a sequence of non-negative integers. This sequence is called a \textit{parameter vector} if $\mathsf{Param}_i  = 0$ for $i\gg 0.$ One also defines $\deg(\mathsf{Param}) = \sum_{i\geq 1}2^{i-1}\mathsf{Param}_i.$ We hereby establish the convention that \textit{the sets of all the parameter vectors and the exponent vectors are given the left lexicographical order}.

\medskip

{\bf $\bullet$ \textit{Linear order on \boldmath{$\mathcal P_m$}}}. Let us consider the monomials $X = x_1^{u_1}x_2^{u_2}\ldots x_m^{u_m}$ and $Y = x_1^{v_1}x_2^{v_2}\ldots x_m^{v_m}$ in $\mathcal P_m$ that have the same degree. We denote by $u(X)$ and $v(Y)$ the exponent vectors of $X$ and $Y$, respectively. We define $u < v$ if there exists a positive integer $d$ such that $u_j = v_j$ for all $j < d$ and $u_d < v_d$. We say that $X < Y$ if and only if either $\mathsf{Param}(X) < \mathsf{Param}(Y)$ or $\mathsf{Param}(X) = \mathsf{Param}(Y)$ and $u(X) < v(Y)$.

\medskip

{\bf $\bullet$ \textit{Binary relations on \boldmath{$\mathcal P_m$}}}. Let $\mathsf{Param}$ be a parameter vector. We define two subspaces of $\mathcal P_m$ associated with $\mathsf{Param}$ as follows: 
 $\mathcal P^{\leq \mathsf{Param}}_m = \langle\{ X\in\mathcal P_m|\, \deg(X) = \deg(\mathsf{Param}),\  \mathsf{Param}(X)\leq \mathsf{Param}\}\rangle,$ and $\mathcal P^{<\mathsf{Param}}_m = \langle \{ X\in\mathcal P_m|\, \deg(X) = \deg(\mathsf{Param}),\  \mathsf{Param}(X) < \mathsf{Param}\}\rangle.$  Given homogeneous polynomials $F$ and $G$ in $\mathcal P_m$, $\deg(F) = \deg(G),$ one defines the following binary relations "$\sim$" and "$\sim_{\mathsf{Param}}$" on $\mathcal P_m$:
\begin{enumerate}
\item [(i)]$F \sim G $ if and only if $F - G(= F+G)\in \overline{\mathscr {A}}\mathcal P_m.$ If $F \sim 0$ then $F$ is "hit", that is, $F$ is in the image of the action 
\xymatrix{
\overline{\mathscr A}\otimes_{\mathbb F_2} \mathcal P_m\ar[r]& \mathcal P_m;}
\item[(ii)] $F \sim_{\mathsf{Param}} G$ if and only if $F, \, G\in \mathcal P^{\leq \mathsf{Param}}_m$ and $F-G(= F+G)\in (\overline{\mathscr {A}}\mathcal P_m+ \mathcal P_m^{<\mathsf{Param}}).$
\end{enumerate}
It follows that the binary relations "$\sim$" and "$\sim_{\mathsf{Param}}$" fulfill the criteria for an equivalence relation. Let $(Q^{\otimes m})^{\mathsf{Param}}$ denote the factor space of $\mathcal P_m$  by the equivalence relation "$\sim_{\mathsf{Param}}$". According to \cite{D.P2, N.S3, W.W}, $(Q^{\otimes m})^{\mathsf{Param}}$ is also a $G(m)$-module, and $Q^{\otimes m}_n\cong \bigoplus_{\deg(\mathsf{Param})  =n}(Q^{\otimes m})^{\mathsf{Param}}.$ Although $Q^{\otimes m}_n$ can be expressed as a direct sum of $(Q^{\otimes m})^{\mathsf{Param}}$, it is crucial to note that these $(Q^{\otimes m})^{\mathsf{Param}}$ are merely filtration quotients, not natural components of $Q^{\otimes m}_n$. In other words, the $(Q^{\otimes m})^{\mathsf{Param}}$ spaces are typically not intrinsic subspaces or quotient spaces of $Q^{\otimes m}_n$.

In \cite{M.K}, Kameko put forth a conjecture stating that the cardinality of a minimal set of generators for the $\mathscr A$-module $\mathcal P_m$  is dominated by an explicit quantity depending on the number of the polynomial algebra's variables $m.$ By way of equivalence, Kameko's conjecture implies that $\dim Q^{\otimes m}_{n} \leq \prod_{1\leq j\leq m}(2^{j}-1)$ for all $n.$ While the statement is valid for $m\leq 4,$ it does not hold in general, as there are illustrative counterexamples (see \cite{N.S1, W.W2}). Nonetheless, the local version of this conjecture is still unresolved and can be formulated as follows.

\begin{conj}[see \cite{W.W2}]\label{gtK}
For each parameter vector $\mathsf{Param}$ of degree $n,$ we have $$ \dim (Q^{\otimes m})^{\mathsf{Param}}\leq  \prod_{1\leq j\leq m}(2^{j}-1).$$
\end{conj}
By the results in \cite{M.K, F.P, N.S1}, the conjecture holds for $m\leq 4$.

\medskip

{\bf $\bullet$ \textit{(Non)admissible monomial}}. A monomial $X\in\mathcal P_m$ is said to be {\it nonadmissible } if there exist monomials $Y_1, Y_2,\ldots, Y_k$ such that $Y_j < X$ for $1\leq j\leq k$ and $X \sim  \sum_{1\leq j\leq k}Y_j.$ In particular, if $X \sim_{\mathsf{Param}(X)}  \sum_{1\leq j\leq k}Y_j,$ then we say that $X$ is \textit{$\mathsf{Param}(X)$-nonadmissible }. Therefrom $X$ is said to be {\it admissible} when it is not nonadmissible.

It is worth noting that as stated in \cite[Proposition 1]{M.M2}, if $X$ is admissible, then it must take the form $X = x_1^{2^{a_1} - 1}x_2^{a_2}x_3^{a_3}\ldots x_m^{a_m}.$

We use the notation and definition of strictly nonadmissible monomials below following \cite{N.S4}. Given any non-negative integer $r,$ let $\mathscr A_{r} = \langle \{Sq^{i}:\, 0\leq i\leq 2^{r}\} \rangle$ denote a sub-Hopf algebra of $\mathscr A.$ We put $\overline{\mathscr A_r} = \overline{\mathscr A}\cap \mathscr A_{r}.$ Given polynomials $F$ and $G$ in $\mathcal P_m,$ where $\deg(F) = \deg(G),$ let $\mathsf{Param}$ be a parameter vector such that $\deg(\mathsf{Param}) = \deg(F) = \deg(G).$ We say that $F\sim_{(r,\, \mathsf{Param})}G$ if and only if $F+G\in \overline{\mathscr A_r}\mathcal P_m + \mathcal P^{<\mathsf{Param}}_m.$ It is also straightforward to check that $\sim_{(r,\, \mathsf{Param})}$ is an equivalence relation on $\mathcal P_m.$ 

{\bf $\bullet$ \textit{Strictly nonadmissible  monomial}}. A monomial $X\in\mathcal P_m$ is said to be {\it strictly nonadmissible } if and only if there exist monomials $Y_1, Y_2,\ldots, Y_k$ such that $Y_j < X$ for $1\leq j \leq k$ and $$X \sim_{(r-1,\, \mathsf{Param}(X))} Y_1 + Y_2 + \cdots + Y_k ,\ \mbox{where $r = {\rm max}\{i\in\mathbb Z: \mathsf{Param}_i(X) > 0\}.$}$$

\medskip

Thus the set of all the admissible monomials of degree $n$ in $\mathcal P_m$ is \textit{a minimal set of $\mathscr {A}$-generators for $\mathcal P_m$ in degree $n.$} 
Hereafter, we write $\mathscr {C}^{\otimes m}_n$ for the set of all admissible monomials of degree $n$ in the $\mathscr A$-module $\mathcal P_m.$

\begin{thm}[see \cite{M.K}]\label{dlK}
Let $X$ be a monomial in $\mathcal P_m.$ We consider a monomial $Z,$ which assigns to a $s\times m$-matrix $(\epsilon_{ij}(Z))$ such that for some non-negative integer $r,\ \epsilon_{ij}(Z) = \epsilon_{(i+r)j}(X)$ for $1\leq i\leq s$ and $1\leq j\leq m.$ If $Z$ is nonadmissible, then so is $X.$
\end{thm} 

\begin{thm}[see \cite{N.S1}]\label{dlS}
Let $X, Y$ and $Z$ be monomials in $\mathcal P_m$ such that $\mathsf{Param}_{i}(Z)=0$ for $i  >  t \geq 1.$ If $Z$ is strictly nonadmissible, then so is $ZY^{2^t}.$ 
\end{thm}

\begin{defns} A monomial $Z = \prod_{1\leq j\leq m}x_j^{u_j}$ in $\mathcal P_m$ is called a {\it spike} if the powers $u_j$ can be written as $2^{v_j} - 1$ for $v_j\in\mathbb Z,\ 1\leq j\leq m.$ If $Z$ is a spike with $u_1 > u_2 > \ldots > u_{s-1}\geq u_s \geq 1$ and $u_j = 0$ for $j  \geq s+1,$ then it is called a {\it minimal spike}.
\end{defns}

It is well-established that spikes cannot appear in the image of any Steenrod square, making them an inseparable part of any generating set of the $\mathscr A$-module $\mathcal P_m$. In addition, a spike of a certain positive degree is the minimal spike if its parameter vector order is minimal with respect to other spikes of that positive degree (see \cite{P.S1}). Particularly, the following key theorem regarding spikes will play a significant role in identifying hit monomials.

\begin{thm}[see \cite{W.S2}]\label{dlsig}
Suppose that $X\in \mathcal P_m$ is a monomial of degree $n,$ where $\beta(n)\leq m.$ Let $Z$ be the minimal spike of degree $n$ in $\mathcal P_m.$ If $\mathsf{Param}(X) < \mathsf{Param}(Z),$ then $X$ is a hit monomial.
\end{thm}

Singer \cite{W.S2} also pointed out that in general, the converse of this theorem does not hold. For the convenience of the reader, we consider the following example:  Let $m = 5,\, n = 37$ and the monomials $Z = x_1^{31}x_2^{3}x_3^{3}x_4^{0}x_5^{0}\in (\mathcal P_5)_{37}$ and $X = x_1(x_2x_3x_4x_5)^{9}\in (\mathcal P_5)_{37}.$ We have $\beta(37) = 3 < 5.$ It is evident that $X$ can be expressed as $fg^{2^{3}},$ where $f = x_1x_2x_3x_4x_5$ and $g =x_2x_3x_4x_5.$ It follows that $\deg(f) = 5 < (2^{3}-1)\beta(\deg(g)),$ and thus, as per Silverman \cite[Theorem 1.2]{Silverman2}, $X$ is a hit monomial. Despite the fact that $Z$ is the minimal spike, it can be observed that $\mathsf{Param}(X) = (5,0,0,4,0) > \mathsf{Param}(Z) = (3,3,1,1,1).$ For further details about a basis of $Q^{\otimes 5}_{37},$ we refer the reader to our recent work \cite{D.P10}.

\subsection{A Review of Several Known Homomorphisms}\label{sub3-3}

In this subsection, we review some useful homomorphisms that will appear a number of times in the proofs of our main results. 

First, for each $1\leq l\leq m,$ one defines the homomorphism \xymatrix{
\mathsf{q}_{(l,\,m)}: \mathcal P_{m-1}\ar[r]&  \mathcal P_m} of algebras by performing the following substitution:
$$ \mathsf{q}_{(l,\,m)}(x_j) = \left\{ \begin{array}{ll}
{x_j}&\text{if }\;1\leq j \leq l-1, \\
x_{j+1}& \text{if}\; l\leq j \leq m-1.
\end{array} \right.$$

In connection with this $\mathsf{q}_{(l,\,m)},$ one has a result due to Moetele and Mothebe \cite{M.M}, which is rather interesting and very helpful for the sequel.

\begin{thm}\label{dlMM}
Let $l$ and $d$ be two positive integers with $1\leq l\leq m.$ If $X\in \mathscr {C}^{\otimes (m-1)}_{n},$ then $x_l^{2^{d}-1}\mathsf{q}_{(l,\,m)}(X)\in \mathscr {C}^{\otimes m}_{n + 2^{d}-1}.$
\end{thm}

Next, we set
$$ \mathcal N_m := \{(l, \mathscr L)\;|\; \mathscr L = (l_1,l_2,\ldots, l_r), 1\leqslant l < l_1< l_2 < \ldots < l_r\leq m, \ 0\leq r \leq m-1\},$$  
where by convention, $\mathscr L = \emptyset,$ if $r = 0.$ Denote by $r = \ell(\mathscr L)$ the length of $\mathscr L.$ For each $(l, \mathscr L)\in\mathcal{N}_m,$ the homomorphism 
\xymatrix{
\mathsf{p}_{(l, \mathscr L)}: \mathcal P_m\ar[r]& \mathcal P_{m-1}} of algebras is defined by the following rule:
$$ \mathsf{p}_{(l, \mathscr L)}(x_j) = \left\{ \begin{array}{ll}
{x_j}&\text{if }\;1\leq j \leq l-1, \\
\sum_{p\in \mathscr L}x_{p-1}& \text{if}\; j = l,\\
x_{j-1}&\text{if}\; l+1 \leq j \leq m.
\end{array} \right.$$
It can also be easily verified that $\mathsf{q}_{(l,\,m)}$ and $\mathsf{p}_{(l; \mathscr L)}$ are also the homomorphisms of  $\mathscr {A}$-modules. In particular, one has $\mathsf{p}_{(l, \emptyset)}(x_l) = 0$ for $1\leq l\leq m$ and $\mathsf{p}_{(l, \mathscr L)}(\mathsf{q}_{(l,\,m)}(X)) = X$ for any $X\in \mathcal P_{m-1}.$ 

Now, let  $(l, \mathscr L)\in\mathcal{N}_m,\ 1 \leq r \leq m-1,$ and let $$ X_{(\mathscr L,\,u)} = x_{l_u}^{2^{r-1} + 2^{r-2} +\, \cdots\, + 2^{r-u}}\prod_{u < d\leq r}x_{l_d}^{2^{r-d}}\ \mbox{for $1\leq u\leq r,\; X_{(\emptyset, 1)} = 1.$}$$
 We consider the following $\mathbb F_2$-linear map, due to \cite{N.S1}:
$$ \begin{array}{ll}
\psi_{(l, \mathscr L)}: \mathcal P_{m-1}&\myto \mathcal P_m\\
 \prod_{1\leq s\leq m-1}x_s^{t_s} &\longmapsto  \left\{ \begin{array}{ll}
\dfrac{x_l^{2^{r} - 1}\mathsf{q}_{(l,\,m)}(\prod\limits_{1\leq s\leq m-1}x_s^{t_s})}{X_{(\mathscr L,\,u)}}&\text{if there exist $u$ such that:} \\
&t_{l_1 - 1} +1= \ldots = t_{l_{(u-1)} - 1} +1 = 2^{r},\\
& t_{l_{u} - 1} + 1 > 2^{r},\\
&\alpha_{r-d}(t_{l_{u} - 1}) -1 = 0,\, 1\leq d\leq u, \\
&\alpha_{r-d}(t_{l_{d}-1}) -1 = 0,\, \ u+1 \leq d \leq r,\\
0&\text{otherwise}.
\end{array} \right.
\end{array}$$
We should emphasize that this $\psi_{(l, \mathscr L)}$ is generally not a homomorphism of $\mathscr A$-modules. See our recent work \cite{D.P8-0} for an example that illustrates this claim.

Moreover, we demonstrated in \cite{P.S1} the following technical finding.

\begin{thm}\label{dlbs}
Let $X$ be a monomial in $\mathcal P_m.$ Then, $\mathsf{p}_{(l, \mathscr L)}(X)\in \mathcal P^{\leq \mathsf{Param}(X)}_{m-1}.$
\end{thm}

It follows from this observation that when $\mathsf{Param}$ is a parameter vector, the homomorphism induced by $\mathsf{p}_{(l, \mathscr L)}$ from $(Q^{\otimes m})^{\mathsf{Param}}$ to $(Q^{\otimes (m-1)})^{\mathsf{Param}}$ can serve as effective tools for establishing the linear independence of certain subsets of $(Q^{\otimes m})^{\mathsf{Param}}.$
\medskip

We end this subsection with a few pivotal rules for our work proofs in the next sections. Let $\mathcal P^{0}_m$ and $\mathcal P^{>0}_m$ denote the $\mathscr A$-submodules of $\mathcal P_m$ spanned by all the monomials $x_1^{t_1}x_2^{t_2}\ldots x_m^{t_m}$ such that $\prod_{1\leq j\leq m}t_j = 0$ and $\prod_{1\leq j\leq m}t_j > 0,$ respectively. By setting $(Q^{\otimes m})^0:= \mathcal P_m^{0}/\overline{\mathscr A}\mathcal P_m^{0} ,\ \mbox{and}\ (Q^{\otimes m})^{>0}:= \mathcal P_m^{>0}/\overline{\mathscr A}\mathcal P_m^{>0},$ one has an isomorphism: $Q^{\otimes m} \cong (Q^{\otimes m})^0\,\bigoplus\, (Q^{\otimes m})^{>0}.$ For a subset $\mathscr V\subset \mathcal P_{m-1},$ we put $$\widetilde {\Phi^0}(\mathscr  V) = \bigcup_{1\leq l \leq m}\psi_{(l, \emptyset)}(\mathscr  V) =  \bigcup_{1\leq l \leq m}\mathsf{q}_{(l, m)}(\mathscr V),\ \ \ \widetilde{\Phi^{>0}}(\mathscr  V) = \bigcup_{(l; \mathscr L)\in\mathcal{N}_m,\;1 \leq \ell(\mathscr L) \leq m-1}(\psi_{(l, \mathscr L)}(\mathscr V)\setminus \mathcal P_m^0),$$
and $\widetilde{\Phi_*}(\mathscr  V) = \widetilde {\Phi^0}(\mathscr  V) \bigcup \widetilde {\Phi^{>0}}(\mathscr V).$ Since $\mathsf{q}_{(l, m)}$ is a homomorphism of the $\mathscr A$-modules, if $\mathscr V$ is a minimal set of generators for the $\mathscr A$-module $\mathcal P_{m-1}$ in a certain positive degree, then $\widetilde{\Phi}^0(\mathscr V) $ is also a minimal set of generators for the $\mathscr A$-module $\mathcal P_m^0$ in that positive degree. 

\section{Proofs of Theorems \ref{dlc1}, \ref{dlc2}, and \ref{dlc3}}\label{s4}

The aim of this section is to prove each of the chief results (namely, Theorems \ref{dlc1}, \ref{dlc2}, and \ref{dlc3}) that were presented in Sect.\ref{s2}. To facilitate the reader's understanding, we provide a brief summary of the notational conventions that we will employ throughout this paper.

\begin{notas}\label{kh41}
\begin{itemize}
\item[(i)] For a polynomial $F\in \mathcal P_m,$ we denote by $[F]$ the classes in $Q^{\otimes m}$ represented by $F.$ If $\mathsf{Param}$ is a parameter vector and $F\in \mathcal P^{\leq \mathsf{Param}},$ then we denote by $[F]_\mathsf{Param}$ the classes in $(Q^{\otimes m})^{\mathsf{Param}}$ represented by $F.$ For a subset $\mathscr{C}\subset \mathcal P_m,$ as usual, we write $|\mathscr C|$ for the cardinal of $\mathscr C;$ at the same time, we put $[\mathscr C] = \{[F]\, :\, F\in \mathscr C\}.$ If $\mathscr C\subset  \mathcal P_m^{\leq \mathsf{Param}},$ then we set $[\mathscr C]_{\mathsf{Param}} = \{[F]_{\mathsf{Param}}\, :\, F\in \mathscr C\}.$ 

\item[(ii)]  Let $\mathsf{Param}$ be a parameter vector of degree $n.$ We set  
$$ \begin{array}{ll}
\medskip
& (\mathscr{C}^{\otimes m}_{n})^{\mathsf{Param}} := \mathscr {C}^{\otimes m}_n\cap \mathcal P_m^{\leq \mathsf{Param}},\ \ (\mathscr{C}^{\otimes m}_{n})^{\mathsf{Param}^{0}} := (\mathscr{C}^{\otimes m}_{n})^{\mathsf{Param}}\cap  \mathcal P^{0}_m,\\
\medskip
 & (\mathscr{C}^{\otimes m}_{n})^{\mathsf{Param}^{>0}} := (\mathscr{C}^{\otimes m}_{n})^{\mathsf{Param}}\cap  \mathcal P^{>0}_m,\\
&(Q_n^{\otimes m})^{\mathsf{Param}^{0}}:= (Q^{\otimes m})^{\mathsf{Param}}\cap (Q^{\otimes m}_n)^{0},\ \mbox{and}\ (Q_n^{\otimes m})^{\mathsf{Param}^{>0}} := (Q^{\otimes m})^{\mathsf{Param}}\cap (Q^{\otimes m}_n)^{>0}.
\end{array}$$
Then Observe that the sets $[(\mathscr{C}^{\otimes m}_{n})^{\mathsf{Param}}]_\mathsf{Param},\, [(\mathscr{C}^{\otimes m}_{n})^{\mathsf{Param}^{0}}]_\mathsf{Param}$ and $[(\mathscr{C}^{\otimes m}_{n})^{\mathsf{Param}^{>0}}]_\mathsf{Param}$ are respectively the bases of the $\mathbb F_2$-vector spaces $(Q_n^{\otimes m})^{\mathsf{Param}},\ (Q_n^{\otimes m})^{\mathsf{Param}^{0}}$ and $(Q_n^{\otimes m})^{\mathsf{Param}^{>0}}.$

\item[(iii)] Putting $U_m = \{1,2,\ldots, m\},\ \mathscr X_{(V,\, m)} = \prod_{u\in U_m\setminus V}x_u,$ where $V \subseteq  U_m.$ In particular, 
$$ \begin{array}{ll}
\medskip
& \mathscr X_{(U_m,\, m)} = 1,\ \ \mathscr X_{({\emptyset},\, m)} = x_1x_2\ldots x_m,\\
 & \mathscr X_{(\{u\},\, m)} = x_1\ldots \hat{x}_u\ldots x_m,\ \mbox{for $1\leq u\leq m.$}
\end{array}$$
Given any $X = x_1^{t_1}x_2^{t_2}\ldots x_m^{t_m}\in \mathcal P_m,$ let $V_k(X) = \{j\in U_m:\; \alpha_k(t_j) = 0\}$ for all $k\geq 0.$ Then, 
$X = \prod_{k\geq 0}\mathscr X_{(V_{k}(X),\,m)}^{2^{k}}$ and $\deg(\mathsf{Param}_k(X)) = \deg(\mathscr X_{(V_{k-1}(X),\,m)})$ for $k\geq 1.$
Noting also that due to \cite{M.K}, one has an identify $X = \prod_{k\geq 0}\mathscr X_{(V_{k}(X),\,m)}^{2^{k}} = \prod_{k,\, j}x_j^{2^{k-1}\epsilon_{kj}(X)}$ with  $\epsilon_{kj}(X)\in \{0, 1\}.$ For instance, with $m  =2,$ 

\medskip

\centerline{\begin{tabular}{ccccc}
the monomials in $\mathcal P_2$: &$x_1^2x_2^2$ & $x_1^2$ & $x_2^4$\cr
matrix: & $\begin{pmatrix}0&0\\1&1\end{pmatrix}$ & $\begin{pmatrix}0&0\\1&0\end{pmatrix}$ & $\begin{pmatrix}0&0\\0&0\\0&1\end{pmatrix}$
\end{tabular}}

\item[(iv)] The subsequent homomorphism is of significant interest and proves to be extremely valuable in achieving our objectives. For $1 \le j \le m$, we define the $\mathscr A$-homomorphism $\rho_j: \mathcal{P}_m \to \mathcal{P}_m$ by its action on the variables $\{x_1, \ldots, x_m\}$. The definition is split into two cases:

\begin{itemize}
    \item \textbf{Adjacent transpositions ($1 \le j \le m-1$):} The operator $\rho_j$ swaps the adjacent variables $x_j$ and $x_{j+1}$, and fixes all others:
    $$\rho_j(x_i) = \begin{cases} x_{j+1} & \text{if } i=j \\ x_j & \text{if } i=j+1 \\ x_i & \text{otherwise.} \end{cases}$$
   
    \item \textbf{A transvection ($j=m$):} The operator $\rho_m$ adds the variable $x_{m-1}$ to $x_m$ and fixes all others:
    $$\rho_m(x_i) = \begin{cases} x_m + x_{m-1} & \text{if } i=m \\ x_i & \text{if } i < m. \end{cases}$$
\end{itemize}

The action of any $\rho_j$ is extended to all polynomials in $\mathcal{P}_m$ as an algebra homomorphism.

The symmetric group $\Sigma_m \subset G(m)$ is generated by the set of adjacent transpositions $\{\rho_1, \ldots, \rho_{m-1}\}$. The general linear group $G(m)$ is generated by the set of operators $\{\rho_j \mid 1 \le j \le m\}$.

Let $[F]_{\mathsf{Param}}$ be a class in $(Q^{\otimes m})^{\mathsf{Param}}$ represented by a homogeneous polynomial $F \in \mathcal{P}_m^{\le \mathsf{Param}}$.

\begin{itemize}
    \item The class $[F]_{\mathsf{Param}}$ is \textbf{$\Sigma_m$-invariant} if and only if it is invariant under the action of all adjacent transpositions:
    $$\rho_j(F) + F \sim_{\mathsf{Param}} 0 \quad \text{for all } j \in \{1, \ldots, m-1\}.$$
    
    \item The class $[F]_{\mathsf{Param}}$ is \textbf{$G(m)$-invariant} if and only if it is $\Sigma_m$-invariant and is also invariant under the action of the transvection $\rho_m$. This is equivalent to the single, comprehensive condition:
    $$\rho_j(F) + F \sim_{\mathsf{Param}} 0 \quad \text{for all } j \in \{1, \ldots, m\}.$$
\end{itemize}
\end{itemize}
\end{notas}

\subsection{Proof of Theorem \ref{dlc1}}\label{sub1}

The proof proceeds in three steps. First, we use the minimal-spike criterion and the Kameko map to restrict the possible parameter vectors occurring in the kernel. Second, we separate the zero-part and positive-part contributions and reduce the problem to the parameter vector $\widetilde{\mathsf{Param}}=(3,3,2,1,1)$. Third, we determine the dimensions of these two pieces by combining known lower-rank calculations with an explicit admissible-basis computation in rank $5$.

Let $X$ be an admissible monomial of degree $n_1$ in the $\mathscr A$-module $\mathcal P_5$ such that $[X]$ belongs to ${\rm Ker}((\widetilde {Sq^0_*})_{(5, n_1)}).$ Observe that $Y = x_1^{2^{5}-1}x_2^{2^{3}-1}x_3^{2^{2}-1}x_4^{2^{0}-1}x_5^{2^{0}-1} = x_1^{31}x_2^{7}x_3^{3}\in \mathcal P_5$ is the minimal spike monomial of degree $n_1,$ and $\mathsf{Param}(Y) = (3,3,2,1,1).$ Hence $Y$ belongs to $\mathscr C_{n_1}^{\otimes 5}.$ Since $[X]\neq [0]$ and $\deg(X)$ is odd, in view of Theorem \ref{dlsig}, either $\mathsf{Param}_1(X) = 3$ or $\mathsf{Param}_1(X) = 5.$ If $\mathsf{Param}_1(X) = 5,$ then $X = \mathscr X_{(\emptyset,\, 5)}Z^2$ with $Z\in (\mathcal P_5)_{n_0}.$  Since $X\in \mathscr C_{n_1}^{\otimes 5}$ and $\mathsf{Param}_i(\mathscr X_{(\emptyset,\, 5)})  = 0$ for all $i > 1,$ by Theorem \ref{dlK}, one derives $Z\in \mathscr C_{n_0}^{\otimes 5},$ and so,  $(\widetilde {Sq_*^0})_{(5,n_1)}([X]) = [Z]\neq [0].$ This contradicts the fact that $[X]\in \mbox{Ker}(\widetilde {Sq_*^0})_{(5,n_1)}.$ Thus, $\mathsf{Param}_1(X)$ is equal to $3.$ Due to this and Theorem \ref{dlS}, we must have that $X = \mathscr X_{(\{i, j\},\,5)}Z_1^2$ with $1\leq i< j  \leq 5$ and $Z_1\in \mathscr C_{19}^{\otimes 5}$. Owing to Tin's thesis \cite{N.T},  $\mathsf{Param}(Z_1)\in \{(3, 2,1,1),\, (3,2,3),\, (3,4,2)\}.$ Now, by the usage of Theorem \ref{dlS}, we shall show that $[X]_{\mathsf{Param}(X)} = [0]_{\mathsf{Param}(X)},$ if either $\mathsf{Param}(Z_1) = (3,2,3)$ or $\mathsf{Param}(Z_1) = (3,4,2).$ 

\medskip

\underline{\textbf{Case} \mbox{\boldmath $\mathsf{Param}(Z_1) = (3,2,3).$}} The lemma below is a direct corollary of the preceding outcomes established in \cite{N.S1} and \cite{D.P1}.

\begin{lem}\label{bd1}
Let $i, j, k, l$ and $m$ be five distinct integers and  $1\leq i,\, j,\, k,\, l,\,m\leq 5.$ Then, the following monomials of degree 17 are nonadmissible:
\begin{itemize}
\item[i)] $x_i^7x_j^2x_kx_l^7, \ x_i^3x_j^6x_kx_l^7, $ for $j < k,$\\[2mm]
 $x_i^3x_j^7x_k^2x_l^5, \ x_i^7x_j^3x_k^4x_l^3,\ x_i^3x_j^3x_k^6x_l^5,$ for $j < k < l,$ \\[2mm] 
$x_i^3x_j^4x_k^2x_lx_m^7,\ x_i^3x_j^2x_kx_l^4x_m^7,\ x_i^3x_j^2x_k^4x_lx_m^7,\ x_i^3x_j^4x_k^2x_l^5x_m^3,$ for $j < k <l;$

\item[ii)] $x_1^7x_2^6x_kx_lx_m^2, \ x_1^7x_2^2x_kx_l^6x_m,\ x_1^3x_2^6x_kx_lx_m^6, \ x_1^7x_2^2x_kx_l^2x_m^5,\ x_1^3x_2^4x_k^2x_lx_m^7, $ \\[2mm]
$x_1^3x_2^2x_kx_l^4x_m^7, \ x_1^3x_2^2x_k^4x_lx_m^7, \ x_1^3x_2^6x_kx_l^2x_m^5,\ x_1^3x_2^6x_kx_l^4x_m^3,\ x_1^3x_2^4x_k^3x_l^4x_m^3;$ 

\item[iii)] $x_1^{3}x_2^{4}x_3x_4^{6}x_5^{3}, \ x_1^{3}x_2^{4}x_3^{3}x_4^{6}x_5,\ x_1^{3}x_2^{4}x_3^{6}x_4x_5^{3},\ x_1^{3}x_2^{4}x_3^{6}x_4^{3}x_5.$
\end{itemize}
\end{lem}

As an illustration, let us examine the monomials $X = x_1^3x_2^4x_k^3x_l^4x_m^3$ and $Y = x_i^3x_j^4x_k^6x_lx_m^3.$ Using the Cartan formula, one can derive the following equalities:
$$ \begin{array}{ll}
\medskip
 X &\sim_{(2,\, \mathsf{Param}(X))}x_1^3x_2^3x_k^4x_l^4x_m^3 + x_1^3x_2^3x_k^3x_l^4x_m^4 + x_1^2x_2^5x_k^3x_l^4x_m^3 + x_1^2x_2^3x_k^5x_l^4x_m^3 + x_1^2x_2^3x_k^3x_l^4x_m^5,\\
Y &\sim_{\mathsf{Param}(Y)} x_i^3x_j^4x_k^5x_l^{2}x_m^3,
\end{array}$$
which consequently establish that $X$ and $Y$ are strictly nonadmissible and $\mathsf{Param}(Y)$-nonadmissible, respectively. Hence, $X$ and $Y$ are nonadmissible monomials.

\begin{lem}\label{bd2}
All permutations of the following monomials are strictly nonadmissible : 

\begin{center}
\begin{tabular}{llrr}
$x_1^{3}x_2^{4}x_3^{9}x_4^{10}x_5^{15}$, & $x_1^{3}x_2^{4}x_3^{9}x_4^{11}x_5^{14}$, & \multicolumn{1}{l}{$x_1^{3}x_2^{4}x_3^{10}x_4^{11}x_5^{13}$,} & \multicolumn{1}{l}{$x_1^{3}x_2^{5}x_3^{8}x_4^{10}x_5^{15}$,} \\
$x_1^{3}x_2^{5}x_3^{8}x_4^{11}x_5^{14}$, & $x_1^{3}x_2^{5}x_3^{9}x_4^{10}x_5^{14}$, & \multicolumn{1}{l}{$x_1^{3}x_2^{5}x_3^{10}x_4^{10}x_5^{13}$,} & \multicolumn{1}{l}{$x_1^{3}x_2^{5}x_3^{10}x_4^{11}x_5^{12}$,} \\
$x_1^{3}x_2^{7}x_3^{8}x_4^{8}x_5^{15}$, & $x_1^{3}x_2^{7}x_3^{8}x_4^{9}x_5^{14}$, & \multicolumn{1}{l}{$x_1^{3}x_2^{7}x_3^{8}x_4^{10}x_5^{13}$,} & \multicolumn{1}{l}{$x_1^{3}x_2^{7}x_3^{8}x_4^{11}x_5^{12}$,} \\
$x_1^{3}x_2^{7}x_3^{9}x_4^{10}x_5^{12}$, & $X:=x_1^{7}x_2^{7}x_3^{8}x_4^{8}x_5^{11}$, & \multicolumn{1}{l}{$Y:=x_1^{7}x_2^{7}x_3^{8}x_4^{9}x_5^{10}$,} & \multicolumn{1}{l}{$Z:=x_1^{7}x_2^{11}x_3^{4}x_4^{9}x_5^{10}$,} \\
$x_1^{7}x_2^{11}x_3^{5}x_4^{8}x_5^{10},$ & $x_1^{7}x_2^{11}x_3^{11}x_4^{4}x_5^{8}.$ & &

\end{tabular}%
\end{center}
\end{lem}

\begin{proof}
Each monomial in the lemma is of the parameter vector $\mathsf{Param}^{*}:=(3,3,2,3).$ The proof of the lemma for the given set of monomials $x_1^{3}x_2^{4}x_3^{9}x_4^{10}x_5^{15},\,  x_1^{3}x_2^{4}x_3^{9}x_4^{11}x_5^{14}, \ldots, x_1^{3}x_2^{7}x_3^{9}x_4^{10}x_5^{12},\, x_1^{7}x_2^{11}x_3^{5}x_4^{8}x_5^{10}$ and $x_1^{7}x_2^{11}x_3^{11}x_4^{4}x_5^{8}$ is rather straightforward. We thus inspect that the monomials $X,\, Y,$ and $Z$ are strictly nonadmissible. Indeed, through the use of the Cartan formula, we obtain the following expression:
$$ \begin{array}{ll}
X & =  Sq^{1}\big(x_1^{7}x_2^{3}x_3x_4^{3}x_5^{26}+
x_1^{7}x_2^{3}x_3x_4^{10}x_5^{19}+
x_1^{7}x_2^{7}x_3^{8}x_4^{5}x_5^{13}+
x_1^{7}x_2^{9}x_3x_4^{10}x_5^{13}\\
&\quad\quad\quad\quad\quad\quad\quad +  x_1^{7}x_2^{9}x_3x_4^{12}x_5^{11}+
x_1^{11}x_2^{3}x_3x_4^{3}x_5^{22}+
x_1^{11}x_2^{3}x_3x_4^{6}x_5^{19}\big)\\
&\quad + Sq^{2}\big(x_1^{7}x_2^{3}x_3x_4^{6}x_5^{22}+
x_1^{7}x_2^{3}x_3^{2}x_4^{5}x_5^{22}+
x_1^{7}x_2^{3}x_3^{2}x_4^{6}x_5^{21}+
 x_1^{7}x_2^{6}x_3x_4^{6}x_5^{19}\\
& \quad\quad\quad\quad\quad\quad\quad + x_1^{7}x_2^{6}x_3x_4^{18}x_5^{7}+
x_1^{7}x_2^{7}x_3^{8}x_4^{3}x_5^{14}+
x_1^{7}x_2^{7}x_3^{8}x_4^{6}x_5^{11}+
x_1^{7}x_2^{9}x_3^{2}x_4^{10}x_5^{11}\big)\\
&\quad + Sq^{4}\big(x_1^{5}x_2^{7}x_3^{8}x_4^{3}x_5^{14}+
x_1^{5}x_2^{7}x_3^{8}x_4^{6}x_5^{11}+
x_1^{7}x_2^{3}x_3^{2}x_4^{3}x_5^{22}+
x_1^{7}x_2^{3}x_3^{2}x_4^{6}x_5^{19}+
x_1^{11}x_2^{5}x_3x_4^{6}x_5^{14}\\
& \quad\quad\quad\quad\quad\quad\quad + x_1^{11}x_2^{5}x_3^{4}x_4^{3}x_5^{14}+
x_1^{11}x_2^{5}x_3^{4}x_4^{6}x_5^{11}+
x_1^{11}x_2^{6}x_3x_4^{6}x_5^{13}+
x_1^{11}x_2^{6}x_3x_4^{12}x_5^{7}\big)\\
&\quad + Sq^{8}\big(x_1^{7}x_2^{5}x_3x_4^{6}x_5^{14}+
x_1^{7}x_2^{5}x_3^{4}x_4^{3}x_5^{14}+
x_1^{7}x_2^{5}x_3^{4}x_4^{6}x_5^{11}+ x_1^{7}x_2^{6}x_3x_4^{6}x_5^{13}+
x_1^{7}x_2^{6}x_3x_4^{12}x_5^{7}\big)\\
&\quad + \sum_{1\leq i\leq 4}X_i \mod(\mathcal P_5^{< \mathsf{Param}^{*}}),
\end{array}$$
where $X_1 = x_1^{5}x_2^{7}x_3^{8}x_4^{10}x_5^{11} < X,\ X_2= x_1^{5}x_2^{11}x_3^{8}x_4^{3}x_5^{14}<X,\ X_3=x_1^{5}x_2^{11}x_3^{8}x_4^{6}x_5^{11}<X,\ X_4= x_1^{7}x_2^{5}x_3^{8}x_4^{10}x_5^{11}<X.$ Next, we have
$$ \begin{array}{ll}
Y & =\sum_{1\leq i\leq 11}Y_i +Sq^{1}\big(x_1^{7}x_2^{3}x_3x_4^{11}x_5^{18}+
x_1^{7}x_2^{5}x_3^{8}x_4^{7}x_5^{13}+
x_1^{7}x_2^{5}x_3^{8}x_4^{9}x_5^{11}+
x_1^{7}x_2^{7}x_3^{8}x_4^{7}x_5^{11}+\\
& \quad\quad\quad\quad\quad + x_1^{7}x_2^{9}x_3x_4^{11}x_5^{12}+
x_1^{7}x_2^{9}x_3x_4^{13}x_5^{10}+
x_1^{11}x_2^{3}x_3x_4^{7}x_5^{18}\big)\\
&\quad + Sq^{2}\big(x_1^{7}x_2^{3}x_3x_4^{22}x_5^{6}+
x_1^{7}x_2^{3}x_3^{2}x_4^{7}x_5^{20}+
x_1^{7}x_2^{3}x_3^{8}x_4^{7}x_5^{14}+
x_1^{7}x_2^{6}x_3x_4^{7}x_5^{18}\\
&\quad\quad\quad\quad\quad+ x_1^{7}x_2^{6}x_3x_4^{19}x_5^{6}+
x_1^{7}x_2^{6}x_3^{8}x_4^{7}x_5^{11}+
x_1^{7}x_2^{7}x_3^{8}x_4^{7}x_5^{10}+
x_1^{7}x_2^{9}x_3^{2}x_4^{11}x_5^{10}\big)\\
&\quad +Sq^{4}\big(x_1^{4}x_2^{7}x_3^{8}x_4^{7}x_5^{11}+
x_1^{5}x_2^{3}x_3^{8}x_4^{7}x_5^{14}+
x_1^{5}x_2^{6}x_3^{8}x_4^{7}x_5^{11}+
x_1^{5}x_2^{7}x_3^{8}x_4^{7}x_5^{10}\\
& \quad\quad\quad\quad\quad+ x_1^{7}x_2^{3}x_3^{2}x_4^{7}x_5^{18}+
x_1^{11}x_2^{5}x_3x_4^{14}x_5^{6}+
x_1^{11}x_2^{5}x_3^{4}x_4^{7}x_5^{10}+
x_1^{11}x_2^{6}x_3x_4^{7}x_5^{12}+
x_1^{11}x_2^{6}x_3x_4^{13}x_5^{6}\big)\\
&\quad+ Sq^{8}\big(x_1^{7}x_2^{5}x_3x_4^{14}x_5^{6}+
x_1^{7}x_2^{5}x_3^{4}x_4^{7}x_5^{10}+
x_1^{7}x_2^{6}x_3x_4^{7}x_5^{12}+
x_1^{7}x_2^{6}x_3x_4^{13}x_5^{6}\big) \mod(\mathcal P_5^{< \mathsf{Param}^{*}}),
\end{array}$$
where the monomials $Y_i< X,\ 1\leq i\leq 11,$ are determined as follows:
\begin{center}
\begin{tabular}{lll}
$Y_{1}=x_1^{4}x_2^{7}x_3^{8}x_4^{11}x_5^{11}$, & $Y_{2}=x_1^{4}x_2^{11}x_3^{8}x_4^{7}x_5^{11}$, & $Y_{3}=x_1^{5}x_2^{3}x_3^{8}x_4^{11}x_5^{14}$, \\
$Y_{4}=x_1^{5}x_2^{6}x_3^{8}x_4^{11}x_5^{11}$, & $Y_{5}=x_1^{5}x_2^{7}x_3^{8}x_4^{11}x_5^{10}$, & $Y_{6}=x_1^{5}x_2^{10}x_3^{8}x_4^{7}x_5^{11}$, \\
 $Y_{7}=x_1^{5}x_2^{11}x_3^{8}x_4^{7}x_5^{10}$, & $Y_{8}=x_1^{7}x_2^{3}x_3^{8}x_4^{9}x_5^{14}$,&
$Y_{9}=x_1^{7}x_2^{5}x_3^{8}x_4^{10}x_5^{11}$,\\
  $Y_{10}=x_1^{7}x_2^{5}x_3^{8}x_4^{11}x_5^{10}$, & $Y_{11}=x_1^{7}x_2^{7}x_3^{8}x_4^{8}x_5^{11}$. &
  \end{tabular}%
\end{center}
Lastly, through a straightforward calculation, we obtain
$$ \begin{array}{ll}
Z & =Sq^{1}\big(
x_1^{7}x_2^{9}x_3^{4}x_4^{7}x_5^{13}+
x_1^{7}x_2^{11}x_3^{4}x_4^{7}x_5^{11}+
x_1^{7}x_2^{13}x_3x_4^{7}x_5^{12}+
x_1^{7}x_2^{13}x_3x_4^{9}x_5^{10}\big)\\
&\quad + Sq^{2}\big(x_1^{3}x_2^{9}x_3^{2}x_4^{11}x_5^{14}+
x_1^{3}x_2^{10}x_3^{4}x_4^{11}x_5^{11}+
x_1^{3}x_2^{10}x_3^{8}x_4^{7}x_5^{11}+
x_1^{3}x_2^{11}x_3^{2}x_4^{11}x_5^{12}\\
& \quad\quad\quad\quad\quad+ x_1^{3}x_2^{14}x_3x_4^{11}x_5^{10}+
x_1^{7}x_2^{9}x_3^{2}x_4^{7}x_5^{14}+
x_1^{7}x_2^{10}x_3^{4}x_4^{7}x_5^{11}+
x_1^{7}x_2^{11}x_3x_4^{10}x_5^{10}\\
& \quad\quad\quad\quad\quad+ x_1^{7}x_2^{11}x_3^{2}x_4^{7}x_5^{12}+
x_1^{7}x_2^{11}x_3^{2}x_4^{9}x_5^{10}+
x_1^{7}x_2^{14}x_3x_4^{7}x_5^{10}\big)\\
&\quad + Sq^{4}\big(x_1^{4}x_2^{11}x_3^{4}x_4^{7}x_5^{11}+
x_1^{5}x_2^{9}x_3^{2}x_4^{7}x_5^{14}+
x_1^{5}x_2^{10}x_3^{4}x_4^{7}x_5^{11}+
x_1^{5}x_2^{11}x_3^{2}x_4^{7}x_5^{12}+
x_1^{5}x_2^{14}x_3x_4^{7}x_5^{10}\big)\\
&\quad + Z_{1}:=x_1^{3}x_2^{9}x_3^{2}x_4^{13}x_5^{14}+
Z_{2}:=x_1^{3}x_2^{9}x_3^{4}x_4^{11}x_5^{14}+
Z_{3}:=x_1^{3}x_2^{10}x_3^{4}x_4^{11}x_5^{13}\\
&\quad + Z_{4}:=x_1^{3}x_2^{10}x_3^{4}x_4^{13}x_5^{11}+
 Z_{5}:=x_1^{3}x_2^{10}x_3^{8}x_4^{7}x_5^{13}+
Z_{6}:=x_1^{3}x_2^{11}x_3^{2}x_4^{13}x_5^{12}\\
&\quad + Z_{7}:=x_1^{3}x_2^{11}x_3^{4}x_4^{11}x_5^{12}+
Z_{8}:=x_1^{3}x_2^{12}x_3^{4}x_4^{11}x_5^{11}+
Z_{9}:=x_1^{3}x_2^{12}x_3^{8}x_4^{7}x_5^{11}\\
&\quad + Z_{10}:=x_1^{3}x_2^{13}x_3^{2}x_4^{11}x_5^{12}+
Z_{11}:=x_1^{3}x_2^{14}x_3x_4^{11}x_5^{12}+
Z_{12}:=x_1^{3}x_2^{14}x_3x_4^{13}x_5^{10}\\
&\quad + Z_{13}:=x_1^{4}x_2^{11}x_3^{4}x_4^{11}x_5^{11}+
Z_{14}:=x_1^{4}x_2^{11}x_3^{8}x_4^{7}x_5^{11}+
Z_{15}:=x_1^{7}x_2^{9}x_3^{2}x_4^{9}x_5^{14}\\
&\quad + Z_{16}:=x_1^{7}x_2^{10}x_3^{4}x_4^{9}x_5^{11}+
Z_{17}:=x_1^{7}x_2^{11}x_3x_4^{10}x_5^{12}+
Z_{18}:=x_1^{7}x_2^{11}x_3x_4^{12}x_5^{10}\\
&\quad + Z_{19}:=x_1^{7}x_2^{11}x_3^{4}x_4^{8}x_5^{11}\mod(\mathcal P_5^{< \mathsf{Param}^{*}}),\ \mbox{where $Z_i < Z$ for every $i.$ }
\end{array}$$
It can be seen from the above equalities that $$X \sim_{(3,\, \mathsf{Param}^{*})}\sum_{1\leq i\leq 4}X_i,\ \ Y \sim_{(3,\, \mathsf{Param}^{*})}\sum_{1\leq i\leq 11}Y_i, \ \mbox{ and }\ Z \sim_{(3,\, \mathsf{Param}^{*})}\sum_{1\leq i\leq 19}Z_i.$$ Consequently, the nonallowability of $X,\, Y,\, Z$ can be established. Hence, the lemma is proven to be true.
\end{proof}

We can observe from a direct computation that there exists a monomial $W$ as given in Lemmas \ref{bd1} and \ref{bd2}, such that $X$ can be expressed as $X = \mathscr X_{(\emptyset,\,5)}Z_1^{2} = WS^{2^{l}}$, where $S$ is a suitable monomial in $\mathcal P_5$ and $l = {\rm max}\{j:\, \mathsf{Param}_j(W) >0\}.$ As per Theorem \ref{dlS}, it can be deduced that $X$ is strictly nonadmissible.

\medskip

\underline{\textbf{Case} \mbox{\boldmath $\mathsf{Param}(Z_1) = (3,4,2).$}} Our first observation is that the statement below follows directly from a result in \cite[Theorem 1.1]{N.S2}.

\begin{lem}\label{bd3}
The following monomials are strictly nonadmissible :

\begin{center}
\begin{tabular}{llll}
$x_1^{2}x_2^{4}x_3^{5}x_4^{7}x_5^{7}$, & $x_1^{2}x_2^{4}x_3^{7}x_4^{5}x_5^{7}$, & $x_1^{2}x_2^{5}x_3^{5}x_4^{6}x_5^{7}$, & $x_1^{2}x_2^{5}x_3^{6}x_4^{5}x_5^{7}$, \\
$x_1^{2}x_2^{5}x_3^{7}x_4^{4}x_5^{7}$, & $x_1^{2}x_2^{5}x_3^{7}x_4^{5}x_5^{6}$, & $x_1^{2}x_2^{6}x_3^{7}x_4^{5}x_5^{5}$, & $x_1^{2}x_2^{7}x_3^{7}x_4^{4}x_5^{5}$, \\
$x_1^{3}x_2^{4}x_3^{4}x_4^{7}x_5^{7}$, & $x_1^{3}x_2^{4}x_3^{5}x_4^{6}x_5^{7}$, & $x_1^{3}x_2^{4}x_3^{6}x_4^{5}x_5^{7}$, & $x_1^{3}x_2^{4}x_3^{7}x_4^{4}x_5^{7}$, \\
$x_1^{3}x_2^{4}x_3^{7}x_4^{5}x_5^{6}$, & $x_1^{3}x_2^{5}x_3^{5}x_4^{6}x_5^{6}$, & $x_1^{3}x_2^{5}x_3^{6}x_4^{4}x_5^{7}$, & $x_1^{3}x_2^{5}x_3^{6}x_4^{5}x_5^{6}$, \\
$x_1^{3}x_2^{5}x_3^{7}x_4^{4}x_5^{6}$, & $x_1^{3}x_2^{6}x_3^{6}x_4^{5}x_5^{5}$, & $x_1^{3}x_2^{6}x_3^{7}x_4^{4}x_5^{5}$, & $x_1^{3}x_2^{7}x_3^{7}x_4^{4}x_5^{4}$, \\
$x_1^{4}x_2^{4}x_3^{7}x_4^{3}x_5^{7}$, & $x_1^{4}x_2^{5}x_3^{6}x_4^{3}x_5^{7}$, & $x_1^{4}x_2^{5}x_3^{7}x_4^{2}x_5^{7}$, & $x_1^{4}x_2^{5}x_3^{7}x_4^{3}x_5^{6}$, \\
$x_1^{4}x_2^{6}x_3^{7}x_4^{3}x_5^{5}$, & $x_1^{4}x_2^{7}x_3^{7}x_4^{2}x_5^{5}$, & $x_1^{4}x_2^{7}x_3^{7}x_4^{3}x_5^{4}$, & $x_1^{5}x_2^{5}x_3^{6}x_4^{2}x_5^{7}$, \\
$x_1^{5}x_2^{5}x_3^{7}x_4^{2}x_5^{6}$, & $x_1^{5}x_2^{6}x_3^{7}x_4^{2}x_5^{5}$, & $x_1^{5}x_2^{6}x_3^{7}x_4^{3}x_5^{4}$, & $x_1^{5}x_2^{7}x_3^{7}x_4^{2}x_5^{4}$.
\end{tabular}%
\end{center}
\end{lem}

After a straightforward computation, we observe that the monomials of the form $X = \mathscr X_{(\emptyset,\,5)}Z_1^{2}$ can be expressed as $X = Y\mathscr X_{({1, 2, 3}, 5)}^{8},$ where $Y$ is a suitable monomial in Lemma \ref{bd3}. Observe that $\mathsf{Param}_3(Y) = 4\neq 0$ and $\mathsf{Param}_{l}(Y) = 0$ for any $l > 3.$ Thus, we can apply Theorem \ref{dlS} to conclude that $X$ is strictly nonadmissible.

Summing up, from the above cases, we must have that the intersection of $\mbox{\rm Ker}(\widetilde {Sq_*^0})_{(5, n_1)}$ with $(Q_{n_1}^{\otimes 5})^{>0}$ must be equal to $(Q_{n_1}^{\otimes 5})^{\widetilde{\mathsf{Param}}^{>0}}$, where $\widetilde{\mathsf{Param}} := (3,3,2,1,1)$. As a result, we obtain an isomorphism: $\mbox{\rm Ker}(\widetilde {Sq_*^0})_{(5, n_1)} \cong  (Q_{n_1}^{\otimes 5})^0  \bigoplus (Q_{n_1}^{\otimes 5})^{\widetilde{\mathsf{Param}}^{>0}}.$ We proceed to perform explicit computations to determine the dimensions of the subspaces $(Q_{n_1}^{\otimes 5})^0$ and $(Q_{n_1}^{\otimes 5})^{\widetilde{\mathsf{Param}}^{>0}}.$

\medskip

{\bf Calculation of \mbox{\boldmath $(Q_{n_1}^{\otimes 5})^0.$}}\ Using a result in Walker and Wood \cite[Proposition 6.2.9]{W.W}, one has 
$$ \begin{array}{ll}
 \dim (Q_{n_1}^{\otimes 5})^0 = \sum_{3\leq r\leq 4}\binom{5}{r}\dim (Q^{\otimes r}_{n_1})^{>0}.
\end{array}$$
On the other side, in accordance with the works \cite{M.K} and \cite{N.S1}, the dimensions of $(Q_{n_1}^{\otimes 3})^{>0}$ and $(Q_{n_1}^{\otimes 4})^{>0}$ are determined to be 15 and 165, respectively. Applying these results and the aforementioned formula yields the conclusion that the dimension of $(Q^{\otimes 5}_{n_1})^{0}$ is 975.

\medskip

{\bf Calculation of \mbox{\boldmath $(Q_{n_1}^{\otimes 5})^{\widetilde{\mathsf{Param}}^{>0}}.$}}\ For any natural numbers $s$ and $l$ satisfying $1\leq l\leq 5$, let us define the set $\mathscr C_{(l, n_1)}$ as the collection of elements of the form $x_l^{2^{s}-1}\mathsf{q}_{(l,5)}(Y)$ in $(\mathcal P_5)_{n_1}$, where $Y\in \mathscr C^{\otimes 4}_{39-2^{s}}$ and $\alpha(46-2^{s})\leq 4$. We also introduce the notation $\mathscr C_{(l, \widetilde{\mathsf{Param}})}$ to represent the intersection of $\mathscr C_{(l, n_1)}$ with $\mathcal P_5^{\leq \widetilde{\mathsf{Param}}}$, and $\mathscr C^{>0}_{(l, \widetilde{\mathsf{Param}})}$ to denote the intersection of $\mathscr C_{(l, \widetilde{\mathsf{Param}})}$ with $(\mathcal P^{>0}_5)_{n_1}$. According to Theorem \ref{dlMM}, it follows that $\mathscr C_{(l, \widetilde{\mathsf{Param}})}\subset (\mathscr C_{n_1}^{\otimes 5})^{\widetilde{\mathsf{Param}}}.$ We hereby conclude the demonstration of the theorem by showing that $
(\mathscr{C}^{\otimes 5}_{n_1})^{\widetilde{\mathsf{Param}}^{>0}} = \mathscr D \bigcup \widetilde{\Phi^{>0}}((\mathscr{C}^{\otimes 4}_{n_1})^{\widetilde{\mathsf{Param}}})\bigcup \big(\bigcup_{1\leq l\leq 5}\mathscr C^{>0}_{(l, \widetilde{\mathsf{Param}})}\big),$ where $\mathscr D = \big\{X_t|\ 1\leq t\leq 199\big\},\ \ \widetilde{\Phi^{>0}}((\mathscr{C}^{\otimes 4}_{n_1})^{\widetilde{\mathsf{Param}}}) = \big\{X_t|\ 200\leq t\leq 620\big\}$ and $\bigcup_{1\leq l\leq 5}\mathscr C^{>0}_{(l, \widetilde{\mathsf{Param}})} = \big\{X_t|\ 621\leq t\leq 925\big\}.$ Here, $X_t,\ 1\leq t\leq 925$ are admissible monomials of degree $n_1$ (see \textbf{Appendix \ref{ap5.4}}). Indeed, we begin by recalling a known result in \cite{S.E}.

\begin{prop}[see \cite{S.E}]\label{mdWW}

For any $X\in (\mathcal P_m)_1,$ the Steenrod squares $Sq^i$ act on $X^d$ as $\binom{d}{i}X^{d+i}$. In particular, if $d$ is odd, then $Sq^{1}(X^{d}) = X^{d+1}$, while if $d$ is even, then $Sq^{1}(X^{d}) = 0.$
\end{prop}

The following technical claim serves as a crucial step in the proof of the theorem. The approach taken to establish this result involves the Cartan formula and Proposition \ref{mdWW}. 

\begin{prop}\label{bd4}
We have a set of nonadmissible monomials $\big\{T_j \in (\mathcal{P}_5^{>0})_{n_1}|\, \mathsf{Param}(T_j) = \widetilde{\mathsf{Param}},\, 1 \leq j \leq 1685\big\}$, which is described in \textbf{Appendix \ref{ap5.2}}.
\end{prop}

\begin{proof}
We establish the proposition by demonstrating its validity for the monomials $T_{250} = x_1x_2^{7}x_3^{24}x_4^{3}x_5^{6},$ $T_{14} = x_1x_2^{14}x_3^{19}x_4^{5}x_5^{2},$ $T_{54} = x_1x_2^{15}x_3^{18}x_4^{5}x_5^{2},$ while other monomials are proven analogously. A noteworthy observation is that $\mathsf{Param}(T_{i}) = \widetilde{\mathsf{Param}}$ for every $i.$  Computing these monomials is rather intricate. Indeed, the application of the Cartan formula and Proposition \ref{mdWW} leads to the following equality:
$$T_{250}= \sum_{1\leq j\leq 63}X_j + Sq^{1}(W) + Sq^{2}(X) + Sq^{4}(Y) + Sq^{8}(Z)\mod(\mathcal P_5^{< \widetilde{\mathsf{Param}}}),$$
where
 \begin{align*}
W&= x_1x_2^{5}x_3^{7}x_4^{15}x_5^{12}+
x_1^{2}x_2^{7}x_3^{11}x_4^{7}x_5^{13}+
x_1^{2}x_2^{7}x_3^{11}x_4^{9}x_5^{11}+
x_1^{2}x_2^{7}x_3^{13}x_4^{7}x_5^{11}+
\medskip
x_1^{2}x_2^{9}x_3^{11}x_4^{7}x_5^{11},\\
X&=x_1x_2^{2}x_3^{7}x_4^{7}x_5^{22}+
x_1x_2^{2}x_3^{22}x_4^{7}x_5^{7}+
x_1x_2^{3}x_3^{6}x_4^{7}x_5^{22}+
x_1x_2^{3}x_3^{22}x_4^{7}x_5^{6}\\
&\quad + x_1x_2^{6}x_3^{7}x_4^{7}x_5^{18}+
x_1x_2^{6}x_3^{18}x_4^{7}x_5^{7}+
x_1x_2^{7}x_3^{6}x_4^{3}x_5^{22}+
x_1x_2^{7}x_3^{7}x_4^{2}x_5^{22}\\
&\quad + x_1x_2^{7}x_3^{7}x_4^{6}x_5^{18}+
x_1x_2^{7}x_3^{10}x_4^{7}x_5^{14}+
x_1x_2^{7}x_3^{11}x_4^{7}x_5^{13}+
x_1x_2^{7}x_3^{11}x_4^{9}x_5^{11}\\
&\quad + x_1x_2^{7}x_3^{13}x_4^{7}x_5^{11}+
x_1x_2^{7}x_3^{14}x_4^{7}x_5^{10}+
x_1x_2^{7}x_3^{18}x_4^{6}x_5^{7}+
x_1x_2^{7}x_3^{22}x_4^{2}x_5^{7}\\
&\quad + x_1x_2^{7}x_3^{22}x_4^{3}x_5^{6}+
\medskip
x_1x_2^{9}x_3^{11}x_4^{7}x_5^{11},\\
Y&=x_1x_2^{4}x_3^{7}x_4^{7}x_5^{18}+
x_1x_2^{4}x_3^{7}x_4^{7}x_5^{18}+
x_1x_2^{4}x_3^{7}x_4^{11}x_5^{14}+
x_1x_2^{4}x_3^{14}x_4^{11}x_5^{7}\\
&\quad + x_1x_2^{4}x_3^{18}x_4^{7}x_5^{7}+
x_1x_2^{4}x_3^{18}x_4^{7}x_5^{7}+
x_1x_2^{5}x_3^{6}x_4^{3}x_5^{22}+
x_1x_2^{5}x_3^{6}x_4^{7}x_5^{18}\\
&\quad + x_1x_2^{5}x_3^{6}x_4^{11}x_5^{14}+
x_1x_2^{5}x_3^{7}x_4^{2}x_5^{22}+
x_1x_2^{5}x_3^{7}x_4^{6}x_5^{18}+
x_1x_2^{5}x_3^{14}x_4^{11}x_5^{6}\\
&\quad +x_1x_2^{5}x_3^{18}x_4^{6}x_5^{7}+
x_1x_2^{5}x_3^{18}x_4^{7}x_5^{6}+
x_1x_2^{5}x_3^{22}x_4^{2}x_5^{7}+
x_1x_2^{5}x_3^{22}x_4^{3}x_5^{6}\\
&\quad + x_1x_2^{6}x_3^{7}x_4^{7}x_5^{16}+
x_1x_2^{6}x_3^{7}x_4^{11}x_5^{12}+
x_1x_2^{6}x_3^{12}x_4^{11}x_5^{7}+
x_1x_2^{6}x_3^{16}x_4^{7}x_5^{7}\\
&\quad +x_1x_2^{11}x_3^{6}x_4^{5}x_5^{14}+
x_1x_2^{11}x_3^{7}x_4^{4}x_5^{14}+
x_1x_2^{11}x_3^{7}x_4^{6}x_5^{12}+
x_1x_2^{11}x_3^{12}x_4^{6}x_5^{7}\\
&\quad + x_1x_2^{11}x_3^{14}x_4^{4}x_5^{7}+
\medskip
x_1x_2^{11}x_3^{14}x_4^{5}x_5^{6},\\
Z&=x_1x_2^{4}x_3^{7}x_4^{7}x_5^{14}+
x_1x_2^{4}x_3^{14}x_4^{7}x_5^{7}+
x_1x_2^{5}x_3^{6}x_4^{7}x_5^{14}+
x_1x_2^{5}x_3^{14}x_4^{7}x_5^{6}\\
&\quad+ x_1x_2^{6}x_3^{7}x_4^{7}x_5^{12}+
x_1x_2^{6}x_3^{12}x_4^{7}x_5^{7}+
x_1x_2^{7}x_3^{6}x_4^{5}x_5^{14}+
x_1x_2^{7}x_3^{7}x_4^{4}x_5^{14}\\
&\quad+  x_1x_2^{7}x_3^{7}x_4^{6}x_5^{12}+
x_1x_2^{7}x_3^{12}x_4^{6}x_5^{7}+
x_1x_2^{7}x_3^{14}x_4^{4}x_5^{7}+
\medskip
x_1x_2^{7}x_3^{14}x_4^{5}x_5^{6},\\
 \sum_{1\leq j\leq 63}X_j &= x_1x_2^{2}x_3^{7}x_4^{7}x_5^{24}+
x_1x_2^{2}x_3^{7}x_4^{9}x_5^{22}+
x_1x_2^{2}x_3^{9}x_4^{7}x_5^{22}+
x_1x_2^{2}x_3^{22}x_4^{7}x_5^{9}\\
&\quad + x_1x_2^{2}x_3^{22}x_4^{9}x_5^{7}+
x_1x_2^{2}x_3^{24}x_4^{7}x_5^{7}+
x_1x_2^{3}x_3^{6}x_4^{7}x_5^{24}+
x_1x_2^{3}x_3^{6}x_4^{9}x_5^{22}\\
&\quad + x_1x_2^{3}x_3^{8}x_4^{7}x_5^{22}+
x_1x_2^{3}x_3^{22}x_4^{7}x_5^{8}+
x_1x_2^{3}x_3^{22}x_4^{9}x_5^{6}+
x_1x_2^{3}x_3^{24}x_4^{7}x_5^{6}\\
&\quad +  x_1x_2^{4}x_3^{7}x_4^{11}x_5^{18}+
x_1x_2^{4}x_3^{7}x_4^{11}x_5^{18}+
x_1x_2^{4}x_3^{11}x_4^{7}x_5^{18}+
x_1x_2^{4}x_3^{18}x_4^{7}x_5^{11}\\
&\quad +  x_1x_2^{4}x_3^{18}x_4^{11}x_5^{7}+
x_1x_2^{4}x_3^{18}x_4^{11}x_5^{7}+
x_1x_2^{5}x_3^{6}x_4^{3}x_5^{26}+
x_1x_2^{5}x_3^{6}x_4^{11}x_5^{18}\\
&\quad +  x_1x_2^{5}x_3^{7}x_4^{2}x_5^{26}+
x_1x_2^{5}x_3^{7}x_4^{10}x_5^{18}+
x_1x_2^{5}x_3^{10}x_4^{3}x_5^{22}+
x_1x_2^{5}x_3^{11}x_4^{2}x_5^{22}\\
&\quad +  x_1x_2^{5}x_3^{11}x_4^{6}x_5^{18}+
x_1x_2^{5}x_3^{18}x_4^{6}x_5^{11}+
x_1x_2^{5}x_3^{18}x_4^{10}x_5^{7}+
x_1x_2^{5}x_3^{18}x_4^{11}x_5^{6}\\
&\quad +  x_1x_2^{5}x_3^{22}x_4^{2}x_5^{11}+
x_1x_2^{5}x_3^{22}x_4^{3}x_5^{10}+
x_1x_2^{5}x_3^{26}x_4^{2}x_5^{7}+
x_1x_2^{5}x_3^{26}x_4^{3}x_5^{6}\\
&\quad +  x_1x_2^{6}x_3^{7}x_4^{9}x_5^{18}+
x_1x_2^{6}x_3^{7}x_4^{11}x_5^{16}+
x_1x_2^{6}x_3^{9}x_4^{7}x_5^{18}+
x_1x_2^{6}x_3^{16}x_4^{11}x_5^{7}\\
&\quad +  x_1x_2^{6}x_3^{18}x_4^{7}x_5^{9}+
x_1x_2^{6}x_3^{18}x_4^{9}x_5^{7}+
x_1x_2^{7}x_3^{6}x_4^{3}x_5^{24}+
x_1x_2^{7}x_3^{6}x_4^{9}x_5^{18}\\
&\quad +  x_1x_2^{7}x_3^{7}x_4^{2}x_5^{24}+
x_1x_2^{7}x_3^{7}x_4^{8}x_5^{18}+
x_1x_2^{7}x_3^{7}x_4^{8}x_5^{18}+
x_1x_2^{7}x_3^{7}x_4^{10}x_5^{16}\\
&\quad +  x_1x_2^{7}x_3^{8}x_4^{3}x_5^{22}+
x_1x_2^{7}x_3^{9}x_4^{2}x_5^{22}+
x_1x_2^{7}x_3^{9}x_4^{6}x_5^{18}+
x_1x_2^{7}x_3^{10}x_4^{5}x_5^{18}\\
&\quad +  x_1x_2^{7}x_3^{10}x_4^{7}x_5^{16}+
x_1x_2^{7}x_3^{11}x_4^{4}x_5^{18}+
x_1x_2^{7}x_3^{11}x_4^{6}x_5^{16}+
x_1x_2^{7}x_3^{16}x_4^{6}x_5^{11}\\
&\quad +  x_1x_2^{7}x_3^{16}x_4^{7}x_5^{10}+
x_1x_2^{7}x_3^{16}x_4^{10}x_5^{7}+
x_1x_2^{7}x_3^{18}x_4^{4}x_5^{11}+
x_1x_2^{7}x_3^{18}x_4^{5}x_5^{10}\\
&\quad +  x_1x_2^{7}x_3^{18}x_4^{6}x_5^{9}+
x_1x_2^{7}x_3^{18}x_4^{8}x_5^{7}+
x_1x_2^{7}x_3^{18}x_4^{8}x_5^{7}+
x_1x_2^{7}x_3^{18}x_4^{9}x_5^{6}\\
&\quad + x_1x_2^{7}x_3^{22}x_4^{2}x_5^{9}+
x_1x_2^{7}x_3^{22}x_4^{3}x_5^{8}+
x_1x_2^{7}x_3^{24}x_4^{2}x_5^{7}.
\end{align*}
It can be observed that $X_j < T_{250}$ for all $j,\, 1\leq j\leq 63.$ Next, with the monomial $T_{14},$ we have
\begin{align*}
T_{14}&= \big[Sq^{1}\big(x_1^{2}x_2^{11}x_3^{15}x_4^{7}x_5^{5}+
x_1^{2}x_2^{11}x_3^{15}x_4^{9}x_5^{3}+
x_1^{2}x_2^{11}x_3^{17}x_4^{7}x_5^{3}+
x_1^{2}x_2^{13}x_3^{15}x_4^{7}x_5^{3}\big)\\
&\quad + Sq^{2}\big(x_1x_2^{6}x_3^{23}x_4^{3}x_5^{6}+
x_1x_2^{6}x_3^{23}x_4^{6}x_5^{3}+
x_1x_2^{6}x_3^{23}x_4^{7}x_5^{2}+
x_1x_2^{7}x_3^{23}x_4^{2}x_5^{6}\\
&\quad\quad\quad\quad + x_1x_2^{7}x_3^{23}x_4^{6}x_5^{2}+
x_1x_2^{10}x_3^{15}x_4^{7}x_5^{6}+
x_1x_2^{10}x_3^{15}x_4^{10}x_5^{3}+
x_1x_2^{11}x_3^{15}x_4^{7}x_5^{5}\\
&\quad\quad\quad\quad + x_1x_2^{11}x_3^{15}x_4^{9}x_5^{3}+
x_1x_2^{11}x_3^{17}x_4^{7}x_5^{3}+
x_1x_2^{13}x_3^{15}x_4^{7}x_5^{3}+
x_1x_2^{14}x_3^{15}x_4^{7}x_5^{2}\big)\\
&\quad + Sq^{4}\big(x_1x_2^{6}x_3^{15}x_4^{12}x_5^{3}+
x_1x_2^{7}x_3^{15}x_4^{8}x_5^{6}+
x_1x_2^{7}x_3^{15}x_4^{10}x_5^{4}+
x_1x_2^{10}x_3^{15}x_4^{5}x_5^{6}\\
&\quad\quad\quad\quad + x_1x_2^{10}x_3^{15}x_4^{6}x_5^{5}+
x_1x_2^{10}x_3^{15}x_4^{7}x_5^{4}+
x_1x_2^{11}x_3^{15}x_4^{4}x_5^{6}+
x_1x_2^{11}x_3^{15}x_4^{6}x_5^{4}\\
&\quad\quad\quad\quad + x_1x_2^{14}x_3^{15}x_4^{4}x_5^{3}+
x_1x_2^{14}x_3^{15}x_4^{5}x_5^{2}\big)\\
&\quad+ Sq^{8}\big(x_1x_2^{6}x_3^{15}x_4^{5}x_5^{6}+
x_1x_2^{6}x_3^{15}x_4^{6}x_5^{5}+
x_1x_2^{6}x_3^{15}x_4^{7}x_5^{4}+
x_1x_2^{7}x_3^{15}x_4^{4}x_5^{6}\\
&\quad\quad\quad\quad + x_1x_2^{7}x_3^{15}x_4^{6}x_5^{4}+
x_1x_2^{8}x_3^{15}x_4^{7}x_5^{2}+
x_1x_2^{10}x_3^{15}x_4^{4}x_5^{3}+
x_1x_2^{10}x_3^{15}x_4^{5}x_5^{2}\big)\\
&\quad + \sum_{1\leq j\leq 33}Y_j \big]\mod(\mathcal P_5^{<\widetilde{\mathsf{Param}}}),
\end{align*}
where 
\begin{align*}
\sum_{1\leq j\leq 33}Y_j &= x_1x_2^{6}x_3^{15}x_4^{16}x_5^{3}+
x_1x_2^{6}x_3^{19}x_4^{5}x_5^{10}+
x_1x_2^{6}x_3^{19}x_4^{6}x_5^{9}+
x_1x_2^{6}x_3^{19}x_4^{7}x_5^{8}\\
&\quad + x_1x_2^{6}x_3^{19}x_4^{9}x_5^{6}+
x_1x_2^{6}x_3^{19}x_4^{10}x_5^{5}+
x_1x_2^{6}x_3^{19}x_4^{11}x_5^{4}+
x_1x_2^{6}x_3^{19}x_4^{12}x_5^{3}\\
&\quad + x_1x_2^{6}x_3^{23}x_4^{3}x_5^{8}+
x_1x_2^{6}x_3^{23}x_4^{8}x_5^{3}+
x_1x_2^{6}x_3^{23}x_4^{9}x_5^{2}+
x_1x_2^{6}x_3^{25}x_4^{3}x_5^{6}\\
&\quad + x_1x_2^{6}x_3^{25}x_4^{6}x_5^{3}+
x_1x_2^{6}x_3^{25}x_4^{7}x_5^{2}+
x_1x_2^{7}x_3^{19}x_4^{4}x_5^{10}+
x_1x_2^{7}x_3^{19}x_4^{6}x_5^{8}\\
&\quad + x_1x_2^{7}x_3^{23}x_4^{2}x_5^{8}+
x_1x_2^{7}x_3^{23}x_4^{8}x_5^{2}+
x_1x_2^{7}x_3^{25}x_4^{2}x_5^{6}+
x_1x_2^{7}x_3^{25}x_4^{6}x_5^{2}\\
&\quad + x_1x_2^{8}x_3^{23}x_4^{3}x_5^{6}+
x_1x_2^{8}x_3^{23}x_4^{6}x_5^{3}+
x_1x_2^{9}x_3^{23}x_4^{2}x_5^{6}+
x_1x_2^{9}x_3^{23}x_4^{6}x_5^{2}\\
&\quad + x_1x_2^{10}x_3^{17}x_4^{7}x_5^{6}+
x_1x_2^{10}x_3^{23}x_4^{4}x_5^{3}+
x_1x_2^{10}x_3^{23}x_4^{5}x_5^{2}+
x_1x_2^{11}x_3^{16}x_4^{7}x_5^{6}\\
&\quad + x_1x_2^{11}x_3^{18}x_4^{7}x_5^{4}+
x_1x_2^{12}x_3^{18}x_4^{7}x_5^{3}+
x_1x_2^{14}x_3^{16}x_4^{7}x_5^{3}+
x_1x_2^{14}x_3^{17}x_4^{7}x_5^{2}\\
&\quad + x_1x_2^{14}x_3^{19}x_4^{4}x_5^{3}.
\end{align*}
It is evident that $Y_j < T_{14}$ for each $j.$ Finally, with the monomial $T_{54},$ we have
\begin{align*}
T_{3}&= \big[Sq^1\big(x_1^{2}x_2^{15}x_3^{11}x_4^{7}x_5^{5}+
x_1^{2}x_2^{15}x_3^{11}x_4^{9}x_5^{3}+
x_1^{2}x_2^{15}x_3^{13}x_4^{7}x_5^{3}+
x_1^{2}x_2^{17}x_3^{11}x_4^{7}x_5^{3}\big)\\
&\quad+ Sq^2\big(x_1x_2^{15}x_3^{10}x_4^{7}x_5^{6}+
x_1x_2^{15}x_3^{10}x_4^{10}x_5^{3}+
x_1x_2^{15}x_3^{11}x_4^{7}x_5^{5}+
x_1x_2^{15}x_3^{11}x_4^{9}x_5^{3}\\
&\quad\quad\quad\quad + x_1x_2^{15}x_3^{13}x_4^{7}x_5^{3}+
x_1x_2^{15}x_3^{14}x_4^{7}x_5^{2}+
x_1x_2^{17}x_3^{11}x_4^{7}x_5^{3}+
x_1x_2^{23}x_3^{6}x_4^{3}x_5^{6}\\
&\quad\quad\quad\quad + x_1x_2^{23}x_3^{6}x_4^{6}x_5^{3}+
x_1x_2^{23}x_3^{6}x_4^{7}x_5^{2}+
x_1x_2^{23}x_3^{7}x_4^{2}x_5^{6}+
x_1x_2^{23}x_3^{7}x_4^{6}x_5^{2}\big)\\
&\quad +Sq^4\big(x_1x_2^{15}x_3^{6}x_4^{5}x_5^{10}+
x_1x_2^{15}x_3^{6}x_4^{6}x_5^{9}+
x_1x_2^{15}x_3^{6}x_4^{11}x_5^{4}+
x_1x_2^{15}x_3^{6}x_4^{12}x_5^{3}\\
&\quad\quad\quad\quad +  x_1x_2^{15}x_3^{7}x_4^{8}x_5^{6}+
x_1x_2^{15}x_3^{7}x_4^{10}x_5^{4}+
x_1x_2^{15}x_3^{11}x_4^{4}x_5^{6}+
x_1x_2^{15}x_3^{11}x_4^{6}x_5^{4}\\
&\quad\quad\quad\quad +  x_1x_2^{15}x_3^{14}x_4^{4}x_5^{3}+
x_1x_2^{15}x_3^{14}x_4^{5}x_5^{2}+
x_1x_2^{21}x_3^{6}x_4^{3}x_5^{6}+
x_1x_2^{21}x_3^{6}x_4^{6}x_5^{3}\\
&\quad\quad\quad\quad + x_1x_2^{21}x_3^{6}x_4^{7}x_5^{2}+
x_1x_2^{21}x_3^{7}x_4^{2}x_5^{6}+
x_1x_2^{21}x_3^{7}x_4^{6}x_5^{2}\big)\\
&\quad + Sq^8\big(x_1x_2^{8}x_3^{11}x_4^{7}x_5^{6}+
x_1x_2^{8}x_3^{14}x_4^{7}x_5^{3}+
x_1x_2^{9}x_3^{10}x_4^{7}x_5^{6}+
x_1x_2^{9}x_3^{14}x_4^{7}x_5^{2}\\
&\quad\quad\quad\quad +  x_1x_2^{10}x_3^{11}x_4^{7}x_5^{4}+
x_1x_2^{10}x_3^{12}x_4^{7}x_5^{3}+
x_1x_2^{11}x_3^{6}x_4^{7}x_5^{8}+
x_1x_2^{11}x_3^{6}x_4^{9}x_5^{6}\\
&\quad\quad\quad\quad +  x_1x_2^{11}x_3^{6}x_4^{10}x_5^{5}+
x_1x_2^{11}x_3^{6}x_4^{12}x_5^{3}+
x_1x_2^{11}x_3^{7}x_4^{4}x_5^{10}+
x_1x_2^{11}x_3^{7}x_4^{6}x_5^{8}\\
&\quad\quad\quad\quad +  x_1x_2^{11}x_3^{10}x_4^{5}x_5^{6}+
x_1x_2^{11}x_3^{10}x_4^{6}x_5^{5}+
x_1x_2^{11}x_3^{10}x_4^{7}x_5^{4}+
x_1x_2^{11}x_3^{14}x_4^{4}x_5^{3}\\
&\quad\quad\quad\quad + x_1x_2^{11}x_3^{14}x_4^{5}x_5^{2}+
x_1x_2^{13}x_3^{6}x_4^{3}x_5^{10}+
x_1x_2^{13}x_3^{6}x_4^{10}x_5^{3}+
x_1x_2^{13}x_3^{6}x_4^{11}x_5^{2}\\
&\quad\quad\quad\quad + x_1x_2^{13}x_3^{7}x_4^{2}x_5^{10}+
x_1x_2^{13}x_3^{7}x_4^{10}x_5^{2}+
x_1x_2^{13}x_3^{10}x_4^{3}x_5^{6}+
x_1x_2^{13}x_3^{10}x_4^{6}x_5^{3}\\
&\quad\quad\quad\quad +  x_1x_2^{13}x_3^{10}x_4^{7}x_5^{2}+
x_1x_2^{13}x_3^{11}x_4^{2}x_5^{6}+
x_1x_2^{13}x_3^{11}x_4^{6}x_5^{2}+
x_1x_2^{15}x_3^{6}x_4^{3}x_5^{8}\\
&\quad\quad\quad\quad +  x_1x_2^{15}x_3^{6}x_4^{5}x_5^{6}+
x_1x_2^{15}x_3^{6}x_4^{6}x_5^{5}+
x_1x_2^{15}x_3^{6}x_4^{7}x_5^{4}+
x_1x_2^{15}x_3^{6}x_4^{8}x_5^{3}\\
&\quad\quad\quad\quad + x_1x_2^{15}x_3^{6}x_4^{9}x_5^{2}+
x_1x_2^{15}x_3^{7}x_4^{2}x_5^{8}+
x_1x_2^{15}x_3^{7}x_4^{4}x_5^{6}+
x_1x_2^{15}x_3^{7}x_4^{6}x_5^{4}\\
&\quad\quad\quad\quad +  x_1x_2^{15}x_3^{7}x_4^{8}x_5^{2}+
x_1x_2^{15}x_3^{8}x_4^{3}x_5^{6}+
x_1x_2^{15}x_3^{8}x_4^{6}x_5^{3}+
x_1x_2^{15}x_3^{8}x_4^{7}x_5^{2}\\
&\quad\quad\quad\quad +  x_1x_2^{15}x_3^{9}x_4^{2}x_5^{6}+
x_1x_2^{15}x_3^{9}x_4^{6}x_5^{2}\big) + \sum_{1\leq j\leq 36}Z_j \big]\mod(\mathcal P_5^{<\widetilde{\mathsf{Param}}}),
\end{align*}
where 
\begin{align*}
\sum_{1\leq j\leq 36}Z_j=&x_1x_2^{8}x_3^{19}x_4^{7}x_5^{6}+
x_1x_2^{8}x_3^{22}x_4^{7}x_5^{3}+
x_1x_2^{9}x_3^{18}x_4^{7}x_5^{6}+
x_1x_2^{9}x_3^{22}x_4^{7}x_5^{2}\\
&\quad + x_1x_2^{10}x_3^{19}x_4^{7}x_5^{4}+
x_1x_2^{10}x_3^{20}x_4^{7}x_5^{3}+
x_1x_2^{11}x_3^{6}x_4^{7}x_5^{16}+
x_1x_2^{11}x_3^{6}x_4^{17}x_5^{6}\\
&\quad + x_1x_2^{11}x_3^{6}x_4^{18}x_5^{5}+
x_1x_2^{11}x_3^{6}x_4^{20}x_5^{3}+
x_1x_2^{11}x_3^{7}x_4^{4}x_5^{18}+
x_1x_2^{11}x_3^{7}x_4^{6}x_5^{16}\\
&\quad + x_1x_2^{11}x_3^{18}x_4^{5}x_5^{6}+
x_1x_2^{11}x_3^{18}x_4^{6}x_5^{5}+
x_1x_2^{11}x_3^{18}x_4^{7}x_5^{4}+
x_1x_2^{11}x_3^{22}x_4^{4}x_5^{3}\\
&\quad + x_1x_2^{11}x_3^{22}x_4^{5}x_5^{2}+
x_1x_2^{13}x_3^{6}x_4^{3}x_5^{18}+
x_1x_2^{13}x_3^{6}x_4^{18}x_5^{3}+
x_1x_2^{13}x_3^{6}x_4^{19}x_5^{2}\\
&\quad + x_1x_2^{13}x_3^{7}x_4^{2}x_5^{18}+
x_1x_2^{13}x_3^{7}x_4^{18}x_5^{2}+
x_1x_2^{13}x_3^{18}x_4^{3}x_5^{6}+
x_1x_2^{13}x_3^{18}x_4^{6}x_5^{3}\\
&\quad + x_1x_2^{13}x_3^{18}x_4^{7}x_5^{2}+
x_1x_2^{13}x_3^{19}x_4^{2}x_5^{6}+
x_1x_2^{13}x_3^{19}x_4^{6}x_5^{2}+
x_1x_2^{15}x_3^{6}x_4^{3}x_5^{16}\\
&\quad + x_1x_2^{15}x_3^{6}x_4^{17}x_5^{2}+
x_1x_2^{15}x_3^{7}x_4^{2}x_5^{16}+
x_1x_2^{15}x_3^{7}x_4^{16}x_5^{2}+
x_1x_2^{15}x_3^{16}x_4^{3}x_5^{6}\\
&\quad + x_1x_2^{15}x_3^{16}x_4^{6}x_5^{3}+
x_1x_2^{15}x_3^{17}x_4^{2}x_5^{6}+
x_1x_2^{15}x_3^{17}x_4^{6}x_5^{2}+
x_1x_2^{15}x_3^{18}x_4^{4}x_5^{3}.
\end{align*}
In particular, $Z_j$ is less than $T_{54}$ for any $j.$ 

Thus, we have successfully established that $T_{250}\sim_{(4,\, \widetilde{\mathsf{Param}})}\sum_{1\leq j\leq 63}X_j,\ \ T_{14}\sim_{(4,\, \widetilde{\mathsf{Param}})}\sum_{1\leq j\leq 33}Y_j,$ and $T_{54}\sim_{(4,\, \widetilde{\mathsf{Param}})}\sum_{1\leq j\leq 36}Z_j.$ This, in turn, confirms the validity of the proposition.
\end{proof}

\medskip

Let $X\in (\mathcal P_5^{>0})_{n_1}$ denote an admissible monomial such that $\mathsf{Param}(X) = \widetilde{\mathsf{Param}}.$ Let $\mathsf{Param}^{*}$ be defined as the parameter vector $(3,2,1,1).$ We have $\mathsf{Param}_1(X) = 3$ and can write $X = \mathscr X_{(\{i,\, j\},\, 5)}Y^2$, where $1\leq i < j \leq 5$ and $Y$ is a monomial of degree $19$ in $\mathcal P_5$. Since $X$ belongs to $(\mathscr C_{n_1}^{\otimes 5})^{\widetilde{\mathsf{Param}}^{>0}}$, we can apply Theorem \ref{dlK} to conclude that $Y$ must be an element of $(\mathscr C_{19}^{\otimes 5})^{\mathsf{Param}^{*}}.$ Upon performing a direct computation, we observe that for any $Z\in (\mathscr C_{19}^{\otimes 5})^{\mathsf{Param}^{*}}$ and $1\leq i < j \leq 5$, if $\mathscr X_{(\{i,\, j\},\, 5)}Z^2\neq X_t$ for $1\leq t\leq 925$, then either $\mathscr X_{(\{i,\, j\},\, 5)}Z^2$  belongs to the set of monomials delineated in Proposition \ref{bd4}, or it has the form $FG^{16}$, where $G$ is a suitable monomial in $\mathcal P_5$, and $F$ is a nonadmissible monomial of degree $25$ in the $\mathscr A$-module $\mathcal P_5.$ Thus, in light of Theorem \ref{dlS}, we infer that $[X_{(\{i,\, j\},\, 5)}Z^2]_{\widetilde{\mathsf{Param}}} = [X_{(\{i,\, j\},\, 5)}Z^2] = [0].$ Since $X = \mathscr X_{(\{i,\, j\},\, 5)}Y^2$ is admissible, $X = X_t$ for some $t,\, 1\leq t\leq 925.$ This implies $(\mathscr{C}^{\otimes 5}_{n_1})^{\widetilde{\mathsf{Param}}^{>0}} \subseteq \{X_t|\ 1\leq t\leq 925\}$. Further, the set $[\{X_t|\ 1\leq t\leq 925\}]$ is linearly independent in the space $(Q^{\otimes 5}_{n_1})^{\widetilde{\mathsf{Param}}^{>0}}.$ To verify this, we exploit a result in \cite{N.S1} and Theorems \ref{dlsig} and \ref{dlbs}, which are considered as useful support tools. Indeed, assume there is a linear relation $\mathcal S = \sum_{1\leq t\leq 925}\gamma_tX_t\sim 0,$ on which $\gamma_t$ belongs to $\mathbb F_2$ for every $t.$ Consider the homomorphisms 
\xymatrix{
\mathsf{p}_{(l, \mathscr L)}: \mathcal P_5\ar[r] &\mathcal P_4,} which are defined in Subsect.\ref{sub3-3} for $m = 5.$ According to Theorem \ref{dlbs}, $\mathsf{p}_{(l, \mathscr L)}$ passes to a homomorphism from $(Q^{\otimes 5}_{n_1})^{\widetilde{\mathsf{Param}}}$ to $(Q_{n_1}^{\otimes 4})^{\widetilde{\mathsf{Param}}}.$ Moreover, $|(\mathscr{C}^{\otimes 4}_{n_1})^{\widetilde{\mathsf{Param}}^{>0}}| = |(\mathscr{C}^{\otimes 4}_{n_1})^{>0}| = 165$ (see \cite{N.S1}). By the usage of Theorems \ref{dlsig} and \ref{dlbs}, we explicitly compute $\mathsf{p}_{(l, \mathscr L)}(\mathcal S)$ in terms of admissible monomials in $(\mathscr{C}^{\otimes 4}_{n_1})^{\widetilde{\mathsf{Param}}^{>0}}$ (modulo $\overline{\mathscr A}\mathcal P_4$). We perform explicit computations using the relations $\mathsf{p}_{(l, \mathscr L)}(\mathcal S) \sim 0$ where $\ell(\mathscr L) > 0$. Combined with a series of advanced calculations, this shows that $\gamma_t = 0$ for all $t.$ The proof of the theorem is complete.

\medskip\noindent\textbf{Final remark.} In our prior research \cite{D.P3}, we have demonstrated:

\begin{prop}\label{md18}
The following claims are each true:
\begin{enumerate}
\item[{\rm i)}] If $Y\in \mathscr C_{n_0}^{\otimes 5},$ then $\overline{\mathsf{Param}}:=\mathsf{Param}(Y)$ is one of the following sequences:
$$ \begin{array}{ll}
\medskip
 \overline{\mathsf{Param}}_{[1]} := (2,2,1,1), \ \   \overline{\mathsf{Param}}_{[2]} :=(2,2,3), \ \   \overline{\mathsf{Param}}_{[3]} :=(2,4,2),\\
\overline{\mathsf{Param}}_{[4]} :=(4,1,1,1), \ \   \overline{\mathsf{Param}}_{[5]} :=(4,1,3), \ \   \overline{\mathsf{Param}}_{[6]} :=(4,3,2).
\end{array}$$

\item[{\rm ii)}] 
$|(\mathscr{C}^{\otimes 5}_{n_0})^{\overline{\mathsf{Param}}_{[k]}}|  = \left\{\begin{array}{ll}
300 &\mbox{{\rm if}}\ k= 1,\\
15 &\mbox{{\rm if}}\ k= 2, 5,\\
10 &\mbox{{\rm if}}\ k= 3,\\
110 &\mbox{{\rm if}}\ k= 4,\\
280 &\mbox{{\rm if}}\ k= 6.
\end{array}\right.$
\end{enumerate}
\end{prop}
It is relevant to note that $|(\mathscr{C}^{\otimes 5}_{n_0})^{\overline{\mathsf{Param}}_{[k]}}| = |(\mathscr{C}^{\otimes 5}_{n_0})^{\overline{\mathsf{Param}}_{[k]}^{>0}}|$ for $k = 2, 3$, and $|(\mathscr{C}^{\otimes 5}_{n_0})^{\overline{\mathsf{Param}}_{[2]}^{0}}| = 0 = |(\mathscr{C}^{\otimes 5}_{n_0})^{\overline{\mathsf{Param}}_{[3]}^{0}}|.$ Moreover, $\dim(Q^{\otimes 5}_{n_0}) = \sum_{1\leq k\leq 6}|(\mathscr{C}^{\otimes 5}_{n_0})^{\overline{\mathsf{Param}}_{[k]}}| = 730.$ Using Corollary \ref{hq0} (refer to Sect.\ref{s2}), one gets $Q^{\otimes 5}_{n_0}\cong \varphi(Q^{\otimes 5}_{n_0})\cong \bigoplus_{1\leq k\leq 6}(Q^{\otimes 5}_{n_1})^{\widehat{\mathsf{Param}}_{[k]}},$ where 
\xymatrix{
\varphi: Q^{\otimes 5}_{n_0}\ar[r]& Q^{\otimes 5}_{n_1},\, [u]\longmapsto [x_1\ldots x_5u^{2}]}, 
is a monomorphism and $$ \begin{array}{ll}
\medskip
 \widehat{\mathsf{Param}}_{[1]} := (5,2,2,1,1), \ \   \widehat{\mathsf{Param}}_{[2]} :=(5,2,2,3), \ \   \widehat{\mathsf{Param}}_{[3]} :=(5,2,4,2),\\
\widehat{\mathsf{Param}}_{[4]} :=(5,4,1,1,1), \ \   \widehat{\mathsf{Param}}_{[5]} :=(5,4,1,3), \ \   \widehat{\mathsf{Param}}_{[6]} :=(5,4,3,2)
\end{array}$$
are parameter vectors of degree $n_1.$ So, as demonstrated in the proof of Theorem \ref{dlc1}, we obtain
$$ \begin{array}{ll}
Q^{\otimes 5}_{n_1}&\cong {\rm Ker}((\widetilde {Sq^0_*})_{(5, n_1)})\bigoplus \varphi(Q^{\otimes 5}_{n_0}) \cong (Q^{\otimes 5}_{n_1})^{\widetilde{\mathsf{Param}}^{0}}\bigoplus (Q^{\otimes 5}_{n_1})^{\widetilde{\mathsf{Param}}^{>0}}\bigoplus (\bigoplus_{1\leq k\leq 6}(Q^{\otimes 5}_{n_1})^{\widehat{\mathsf{Param}}_{[k]}})\\
& \cong (Q^{\otimes 5}_{n_1})^{\widetilde{\mathsf{Param}}}\bigoplus (\bigoplus_{1\leq k\leq 6}(Q^{\otimes 5}_{n_1})^{\widehat{\mathsf{Param}}_{[k]}}),\\
\end{array}$$ 
where $\dim (Q^{\otimes 5}_{n_1})^{\widetilde{\mathsf{Param}}} = \dim {\rm Ker}((\widetilde {Sq^0_*})_{(5, n_1)}) = 1900$ and $\dim (Q^{\otimes 5}_{n_1})^{\widehat{\mathsf{Param}}_{[k]}} = \dim (Q^{\otimes 5}_{n_0})^{\overline{\mathsf{Param}}_{[k]}}$ for every $k.$ %Combining these data with the fact that $Q^{\otimes 5}_{n_s}\cong Q^{\otimes 5}_{n_1}$ for any $s\geq 1,$ we 
Drawing on the information presented and the fact that $Q^{\otimes 5}_{n_s}$ is isomorphic to $Q^{\otimes 5}_{n_1}$ for all $s$ greater than or equal to 1, we have immediately

\begin{corl}
Conjecture \ref{gtK} is valid for $m = 5$ in the generic degree $n_s = 5(2^{s}-1) + 18.2^{s}$ for any $s\geq 0.$  
\end{corl}

\subsection{Proof of Theorem \ref{dlc2}}\label{sub3}

In order to carry out the proof of Theorem \ref{dlc2}, we first need the important proposition below (Proposition \ref{kq3}). The invariant calculation is carried out in two stages. We first determine the $\Sigma_5$-invariant subspaces inside each parameter-vector component, and then impose the additional transvection relation $\rho_5(f)+f\sim 0$ to pass from $\Sigma_5$-invariants to $G(5)$-invariants. Note that the results in this proposition have also been verified by computational algorithms implemented in \texttt{SageMath} \cite{Phuc25d} and \texttt{OSCAR} \cite{Phuc25e}. Detailed output of the algorithms is provided in \textbf{Appendix \ref{ap5.3} and Appendix \ref{ap5.4}}.

\begin{prop}\label{kq3} 
Let $\widetilde{\mathsf{Param}}$ and $\overline{\mathsf{Param}}_{[k]}$ be the parameter vectors as in the proof of Theorem \ref{dlc1} and Proposition \ref{md18}, respectively. Then, the following assertions each hold:
\begin{enumerate}

\item[{\rm i)}] $((Q_{n_0=18}^{\otimes 5})^{\overline{\mathsf{Param}}_{[4]}})^{G(5)} = \langle [\mathcal R'_4]_{\overline{\mathsf{Param}}_{[4]}} \rangle,$ where
$$ \begin{array}{ll}
\medskip
\mathcal R'_4 &= x_1x_2x_3x_4x_5^{14} + x_1x_2x_3x_4^{14}x_5 + x_1x_2x_3^{14}x_4x_5 + x_1x_2^{3}x_3x_4x_5^{12}\\
\medskip
&\quad + x_1x_2^{3}x_3x_4^{12}x_5 + x_1x_2^{3}x_3^{12}x_4x_5 + x_1^{3}x_2x_3x_4x_5^{12} + x_1^{3}x_2x_3x_4^{12}x_5\\
\medskip
&\quad + x_1^{3}x_2x_3^{12}x_4x_5 + x_1^{3}x_2^{5}x_3x_4x_5^{8} + x_1^{3}x_2^{5}x_3x_4^{8}x_5 + x_1^{3}x_2^{5}x_3^{8}x_4x_5; 
\end{array}$$

\item[{\rm ii)}] $((Q_{n_0=18}^{\otimes 5})^{\overline{\mathsf{Param}}_{[k]}})^{G(5)} = 0$ with  $k \neq 4.$

\item[{\rm iii)}]  $((Q_{n_1=41}^{\otimes 5})^{\widetilde{\mathsf{Param}}})^{G(5)} = 0.$
\end{enumerate}
\end{prop}

\begin{proof}
%We shall demonstrate the validity of item $i)$ when $k=1,$ and provide a detailed proof for item $ii)$. The remaining items can be established utilizing similar methodologies. 
Let $\mathsf{Param}$ be a parameter vector of degree $5$ and let $T_1, T_2, \ldots, T_s$ be the monomials in $\mathcal P_5^{\leq \overline{\mathsf{Param}}_{[k]}}$ for $s\geq 1.$ We set
$$ \Sigma_5(T_1, T_2, \ldots, T_s)= \{\sigma(T_j):\, \sigma \in \Sigma_5,\, 1\leq j\leq s\}\subset \mathcal P_5^{\leq \overline{\mathsf{Param}}_{[k]}}.$$
%Proposition \ref{md18} establishes that the dimension of the cohit module $(Q^{\otimes 5}_{n_0})^{\overline{\mathsf{Param}}_{[1]}}$ is $300.$ Further, a basis for this space can be found in the set $ \{[X_j]:\, 1\leq j\leq 300 \},$ where the monomials $X_j,\ 1\leq j\leq 300,$ are described as in the Appendix \ref{ap5.2} of the online version \cite{Phucf}. Perhaps it shoud be noted again that $\overline{\mathsf{Param}}_{[1]}$ is a parameter vector of the mimimal spike $x_1^{15}x_2^{3}=x_1^{15}x_2^{3}\prod_{3\leq i\leq 5}x_i^{0}\in (\mathcal P_5)_{n_0},$ and so, $[X]_{\overline{\mathsf{Param}}_{[1]}} = [X]$ for all $X\in (\mathcal P_5)_{n_0}.$

$\bullet$ We first prove item i). By a simple computation, we obtain a direct summand decomposition of the $\Sigma_5$-module: $ (Q^{\otimes 5}_{n_0})^{\overline{\mathsf{Param}}_{[4]}} = \mathbb V_1 \bigoplus \mathbb V_2,$
where
$$ \begin{array}{ll}
\medskip
\mathbb V_1:= \langle [\Sigma_5(Y_1)\bigcup \Sigma_5(Y_{21})\bigcup \Sigma_5(Y_{51})]_{\overline{\mathsf{Param}}_{[4]}} \rangle &= \langle \{[Y_j]_{\overline{\mathsf{Param}}_{[4]}}:\, 1\leq j\leq 70 \}\rangle,\\
\mathbb V_2:= \langle [\Sigma_5(Y_{71}, Y_{74})\bigcup \Sigma_5(Y_{77}, Y_{78}, Y_{81}, Y_{82})]_{\overline{\mathsf{Param}}_{[4]}} \rangle  &=  \langle\{ [Y_j]_{\overline{\mathsf{Param}}_{[4]}}:\, 71\leq j\leq 110 \}\rangle.
\end{array}$$

\begin{lem}\label{bd6}
We have
$$\mathbb V_1^{\Sigma_5} = \langle\{ [\mathcal R'_1]_{\overline{\mathsf{Param}}_{[4]}},\, [\mathcal R'_2]_{\overline{\mathsf{Param}}_{[4]}},\, [\mathcal R'_3]_{\overline{\mathsf{Param}}_{[4]}}\} \rangle.$$
and 
$$\mathbb V_2^{\Sigma_5} = \langle [\mathcal R'_4]_{\overline{\mathsf{Param}}_{[4]}}\rangle,$$
where $$\mathcal R'_1:=\sum_{1\leq j\leq 20}Y_j,\, \mathcal R'_2:= \sum_{21\leq j\leq 50}Y_j,\, \mathcal R'_3:=\sum_{51\leq j\leq 70}Y_j,\ \mathcal R'_4:=\sum_{96\leq j\leq 107}Y_j.$$ Here, the monomials $Y_j,\ 1\leq j\leq 110,$ are given in \textbf{Appendix \ref{ap5.1}}.
\end{lem}

\begin{proof}[{\it Outline of the proof}]
We see that $\mathbb V_2$ is an $\mathbb F_2$-vector space with dimension $40$ and basis $\{[Y_j]_{\overline{\mathsf{Param}}_{[4]}}:\, 71\leq j\leq 110\}.$ By using the relations $\rho_t(e)\sim_{\overline{\mathsf{Param}}_{[4]}}e,$ where $1\leq t\leq 4,\ e = \sum_{71\leq j\leq 110}\gamma_jY_j$ with $\gamma_j\in \mathbb F_2,\, j = 71, \ldots, 110,$ and $[e]_{\overline{\mathsf{Param}}_{[4]}}\in \mathbb V_2^{\Sigma_5},$ we get $\gamma_{96} = \gamma_{97} = \cdots = \gamma_{107}$ and $\gamma_{j} = 0$ for $j\not\in\{96, 97, \ldots, 107\}.$ The computation is performed in an entirely analogous manner for $\mathbb V_1^{\Sigma_5},$ and we also obtain $\mathbb V_1^{\Sigma_5} = \langle\{ [\mathcal R'_1]_{\overline{\mathsf{Param}}_{[4]}},\, [\mathcal R'_2]_{\overline{\mathsf{Param}}_{[4]}},\, [\mathcal R'_3]_{\overline{\mathsf{Param}}_{[4]}}\} \rangle.$  
\end{proof}

Now, for any $[H]_{\overline{\mathsf{Param}}_{[4]}}\in ((Q^{\otimes 5}_{n_0})^{\overline{\mathsf{Param}}_{[4]}})^{G(5)},$ due to Lemma \ref{bd6}, one has
$$ H\sim _{\overline{\mathsf{Param}}_{[4]}} \gamma_1\mathcal R'_1 + \gamma_2\mathcal R'_2 + \gamma_3\mathcal R'_3 + \gamma_4\mathcal R'_4,\  \gamma_j\in\mathbb F_2,\ 1\leq j\leq 4.$$ From Theorem \ref{dlsig} and the $\mathscr A$-homomorphism 
\xymatrix{
\rho_5: \mathcal P_5\ar[r]& \mathcal P_5} mentioned in Notation \ref{kh41}(iv), we explicitly compute $\rho_5(H) + H$ in admissible terms $Y_j$ (modulo $((\overline{\mathscr A}\mathcal P_5\cap \mathcal P_5^{\leq \overline{\mathsf{Param}}_{[4]}}) + \mathcal P_5^{<\overline{\mathsf{Param}}_{[4]}})),$ and obtain
$$\begin{array}{ll} \rho_5(H) + H& \sim_{\overline{\mathsf{Param}}_{[4]}} \big(\gamma_1(x_1^{3}x_2x_3^{12}x_4x_5) + (\gamma_2+\gamma_3)(x_1x_2x_3^{3}x_4^{12}x_5) + \gamma_3(x_2x_3^{3}x_4^{5}x_5^{9}) +  \, \mbox{ other terms}\big)\\
&\sim_{\overline{\mathsf{Param}}_{[4]}} 0.
\end{array}$$
The equality implies that $\gamma_1 = \gamma_2 = \gamma_3 = 0,$ and therefore, we get 
$$((Q_{n_0}^{\otimes 5})^{\overline{\mathsf{Param}}_{[4]}})^{G(5)} = \langle [\mathcal R'_4]_{\overline{\mathsf{Param}}_{[4]}} \rangle.$$

$\bullet$ For item ii), by similar calculations using the basis of $(Q_{n_0}^{\otimes 5})^{\overline{\mathsf{Param}}_{[k]}}$ (given in \textbf{Appendix \ref{ap5.3}}) and the homomorphisms \xymatrix{
\rho_u: \mathcal P_5\ar[r]& \mathcal P_5},\, $1\leq u\leq 4,$ we obtain the following results:
$$ \dim ((Q_{n_0}^{\otimes 5})^{\overline{\mathsf{Param}}_{[k]}})^{\Sigma_5} = \left\{\begin{array}{ll}
11&\mbox{if $k = 1$},\\
2&\mbox{if $k = 2,\, 5,\, 6$},\\
1&\mbox{if $k = 3$},\\
\end{array}\right.$$
Explicit bases for the $\Sigma_5$-invariants $((Q_{n_0}^{\otimes 5})^{\overline{\mathsf{Param}}_{[k]}})^{\Sigma_5}$ are given explicitly in \textbf{Appendix \ref{ap5.3}}. Then, using these results and the homomorphism \xymatrix{\rho_5: \mathcal P_5\ar[r]& \mathcal P_5}, we get $((Q_{n_0}^{\otimes 5})^{\overline{\mathsf{Param}}_{[k]}})^{G(5)} = 0$ for all $k\neq 4.$

$\bullet$ For item iii), we first compute the $\Sigma_5$-invariant space $((Q_{n_1}^{\otimes 5})^{\widetilde{\mathsf{Param}}})^{\Sigma_5}$ by using Theorem \ref{dlc1} (specifically, an explicit basis for $(Q_{n_1}^{\otimes 5})^{\widetilde{\mathsf{Param}}}$, which is given in \textbf{Appendix \ref{ap5.4}}) and the homomorphisms $\rho_j: \mathcal P_5\to \mathcal P_5$, $1\leq j\leq 4$. We obtain:
$$\dim ((Q_{n_1}^{\otimes 5})^{\widetilde{\mathsf{Param}}})^{\Sigma_5} = 31,$$
and an explicit generating set for $((Q_{n_1}^{\otimes 5})^{\widetilde{\mathsf{Param}}})^{\Sigma_5}$ is described fully in \textbf{Appendix \ref{ap5.4}}. Consequently, by using this result and the homomorphism \xymatrix{
\rho_5: \mathcal P_5\ar[r]& \mathcal P_5}, we obtain $((Q_{n_1}^{\otimes 5})^{\widetilde{\mathsf{Param}}})^{G(5)} = 0.$ This completes the proof of the proposition.
\end{proof}

We now turn to the proof of Theorem \ref{dlc2}.

\begin{proof}

We perform the computation of the invariant spaces $(Q_{n_s}^{\otimes 5})^{G(5)},$ where $s = 0,\, 1.$

$\bullet$ {\bf Computation of $(Q_{n_0 = 18}^{\otimes 5})^{G(5)}.$} Assume that $g\in \mathcal P_5$ such that $[g]\in (Q_{n_0 = 18}^{\otimes 5})^{G(5)}.$ Then, by Proposition \ref{kq3}(i) and (ii), we have
$$ g\sim \beta_0\cdot \mathcal R'_4 + \sum_{x\in \bigcup_{1\leq i\leq 3}(\mathscr C^{\otimes 5}_{n_0})^{\overline{\mathsf{Param}}_{[i]}}}\beta_x\cdot x, \ \ \beta_0,\, \beta_x\in \mathbb F_2.$$
By a direct calculation using the relations $\rho_j(g)\sim g,\, 1\leq j\leq 4,$ we obtain
$$g\sim \beta_0\cdot \mathcal R'_4 + \sum_{1\leq k\leq 31}\beta_kg_k,$$
where 
\begin{align*}
g_1 &= x_{1} x_{2}^{2} x_{3} x_{4}^{2} x_{5}^{12}, \\
g_2 &= x_{1}^{3} x_{2} x_{3}^{2} x_{4}^{12} + x_{1} x_{2}^{3} x_{3}^{2} x_{4}^{12} + x_{1}^{3} x_{2} x_{3}^{2} x_{5}^{12} + x_{1} x_{2}^{3} x_{3}^{2} x_{5}^{12} \\
    & + x_{1}^{3} x_{2} x_{4}^{2} x_{5}^{12} + x_{1} x_{2}^{3} x_{4}^{2} x_{5}^{12} + x_{1}^{3} x_{3} x_{4}^{2} x_{5}^{12} + x_{2}^{3} x_{3} x_{4}^{2} x_{5}^{12} \\
    & + x_{1} x_{3}^{3} x_{4}^{2} x_{5}^{12} + x_{2} x_{3}^{3} x_{4}^{2} x_{5}^{12}, \\
g_3 &= x_{1} x_{2}^{2} x_{3}^{4} x_{4}^{5} x_{5}^{6}, \\
g_4 &= x_{1}^{15} x_{2} x_{3}^{2} + x_{1} x_{2}^{15} x_{3}^{2} + x_{1} x_{2}^{2} x_{3}^{15} + x_{1}^{15} x_{2} x_{4}^{2} \\
    & + x_{1} x_{2}^{15} x_{4}^{2} + x_{1}^{15} x_{3} x_{4}^{2} + x_{2}^{15} x_{3} x_{4}^{2} + x_{1} x_{3}^{15} x_{4}^{2} \\
    & + x_{2} x_{3}^{15} x_{4}^{2} + x_{1} x_{2}^{2} x_{4}^{15} + x_{1} x_{3}^{2} x_{4}^{15} + x_{2} x_{3}^{2} x_{4}^{15} \\
    & + x_{1}^{15} x_{2} x_{5}^{2} + x_{1} x_{2}^{15} x_{5}^{2} + x_{1}^{15} x_{3} x_{5}^{2} + x_{2}^{15} x_{3} x_{5}^{2} \\
    & + x_{1} x_{3}^{15} x_{5}^{2} + x_{2} x_{3}^{15} x_{5}^{2} + x_{1}^{15} x_{4} x_{5}^{2} + x_{2}^{15} x_{4} x_{5}^{2} \\
    & + x_{3}^{15} x_{4} x_{5}^{2} + x_{1} x_{4}^{15} x_{5}^{2} + x_{2} x_{4}^{15} x_{5}^{2} + x_{3} x_{4}^{15} x_{5}^{2} \\
    & + x_{1} x_{2}^{2} x_{5}^{15} + x_{1} x_{3}^{2} x_{5}^{15} + x_{2} x_{3}^{2} x_{5}^{15} + x_{1} x_{4}^{2} x_{5}^{15} \\
    & + x_{2} x_{4}^{2} x_{5}^{15} + x_{3} x_{4}^{2} x_{5}^{15}, \\
g_5 &= x_{1}^{3} x_{2}^{3} x_{3}^{4} x_{4}^{8} + x_{1}^{3} x_{2}^{3} x_{3}^{4} x_{5}^{8} + x_{1}^{3} x_{2}^{3} x_{4}^{4} x_{5}^{8} + x_{1}^{3} x_{3}^{3} x_{4}^{4} x_{5}^{8} \\
    & + x_{2}^{3} x_{3}^{3} x_{4}^{4} x_{5}^{8}, \\
g_6 &= x_{1}^{7} x_{2} x_{3}^{2} x_{4}^{8} + x_{1} x_{2}^{7} x_{3}^{2} x_{4}^{8} + x_{1} x_{2}^{2} x_{3}^{7} x_{4}^{8} + x_{1} x_{2}^{2} x_{3}^{4} x_{4}^{11} \\
    & + x_{1}^{7} x_{2} x_{3}^{2} x_{5}^{8} + x_{1} x_{2}^{7} x_{3}^{2} x_{5}^{8} + x_{1} x_{2}^{2} x_{3}^{7} x_{5}^{8} + x_{1}^{7} x_{2} x_{4}^{2} x_{5}^{8} \\
    & + x_{1} x_{2}^{7} x_{4}^{2} x_{5}^{8} + x_{1}^{7} x_{3} x_{4}^{2} x_{5}^{8} + x_{2}^{7} x_{3} x_{4}^{2} x_{5}^{8} + x_{1} x_{3}^{7} x_{4}^{2} x_{5}^{8} \\
    & + x_{2} x_{3}^{7} x_{4}^{2} x_{5}^{8} + x_{1} x_{2}^{2} x_{4}^{7} x_{5}^{8} + x_{1} x_{3}^{2} x_{4}^{7} x_{5}^{8} + x_{2} x_{3}^{2} x_{4}^{7} x_{5}^{8} \\
    & + x_{1} x_{2}^{2} x_{3}^{4} x_{5}^{11} + x_{1} x_{2}^{2} x_{4}^{4} x_{5}^{11} + x_{1} x_{3}^{2} x_{4}^{4} x_{5}^{11} + x_{2} x_{3}^{2} x_{4}^{4} x_{5}^{11}, \\
g_7 &= x_{1}^{3} x_{2}^{5} x_{3}^{6} x_{4}^{2} x_{5}^{2} + x_{1}^{3} x_{2}^{5} x_{3}^{2} x_{4}^{6} x_{5}^{2} + x_{1}^{3} x_{2} x_{3}^{6} x_{4}^{6} x_{5}^{2} + x_{1} x_{2}^{3} x_{3}^{6} x_{4}^{6} x_{5}^{2} \\
    & + x_{1}^{3} x_{2}^{5} x_{3}^{2} x_{4}^{2} x_{5}^{6} + x_{1}^{3} x_{2} x_{3}^{6} x_{4}^{2} x_{5}^{6} + x_{1} x_{2}^{3} x_{3}^{6} x_{4}^{2} x_{5}^{6} + x_{1}^{3} x_{2} x_{3}^{2} x_{4}^{6} x_{5}^{6} \\
    & + x_{1} x_{2}^{3} x_{3}^{2} x_{4}^{6} x_{5}^{6} + x_{1} x_{2}^{2} x_{3}^{3} x_{4}^{6} x_{5}^{6} + x_{1} x_{2} x_{3}^{2} x_{4}^{6} x_{5}^{8} + x_{1} x_{2} x_{3}^{2} x_{4}^{2} x_{5}^{12}, \\
g_8 &= x_{1} x_{2} x_{3}^{6} x_{4}^{2} x_{5}^{8} + x_{1} x_{2} x_{3}^{2} x_{4}^{4} x_{5}^{10}, \\
g_9 &= x_{1} x_{2}^{3} x_{3}^{2} x_{4}^{4} x_{5}^{8}, \\
g_{10} &= x_{1} x_{2}^{2} x_{3}^{12} x_{4}^{3} + x_{1} x_{2}^{2} x_{3}^{3} x_{4}^{12} + x_{1} x_{2}^{2} x_{3}^{12} x_{5}^{3} + x_{1} x_{2}^{2} x_{4}^{12} x_{5}^{3} \\
    & + x_{1} x_{3}^{2} x_{4}^{12} x_{5}^{3} + x_{2} x_{3}^{2} x_{4}^{12} x_{5}^{3} + x_{1} x_{2}^{2} x_{3}^{3} x_{5}^{12} + x_{1} x_{2}^{2} x_{4}^{3} x_{5}^{12} \\
    & + x_{1} x_{3}^{2} x_{4}^{3} x_{5}^{12} + x_{2} x_{3}^{2} x_{4}^{3} x_{5}^{12}, \\
g_{11} &= x_{1}^{7} x_{2} x_{3}^{10} + x_{1} x_{2}^{7} x_{3}^{10} + x_{1} x_{2}^{6} x_{3}^{11} + x_{1}^{7} x_{2} x_{4}^{10} \\
    & + x_{1} x_{2}^{7} x_{4}^{10} + x_{1}^{7} x_{3} x_{4}^{10} + x_{2}^{7} x_{3} x_{4}^{10} + x_{1} x_{3}^{7} x_{4}^{10} \\
    & + x_{2} x_{3}^{7} x_{4}^{10} + x_{1} x_{2}^{6} x_{4}^{11} + x_{1} x_{3}^{6} x_{4}^{11} + x_{2} x_{3}^{6} x_{4}^{11} \\
    & + x_{1}^{7} x_{2} x_{5}^{10} + x_{1} x_{2}^{7} x_{5}^{10} + x_{1}^{7} x_{3} x_{5}^{10} + x_{2}^{7} x_{3} x_{5}^{10} \\
    & + x_{1} x_{3}^{7} x_{5}^{10} + x_{2} x_{3}^{7} x_{5}^{10} + x_{1}^{7} x_{4} x_{5}^{10} + x_{2}^{7} x_{4} x_{5}^{10} \\
    & + x_{3}^{7} x_{4} x_{5}^{10} + x_{1} x_{4}^{7} x_{5}^{10} + x_{2} x_{4}^{7} x_{5}^{10} + x_{3} x_{4}^{7} x_{5}^{10} \\
    & + x_{1} x_{2}^{6} x_{5}^{11} + x_{1} x_{3}^{6} x_{5}^{11} + x_{2} x_{3}^{6} x_{5}^{11} + x_{1} x_{4}^{6} x_{5}^{11} \\
    & + x_{2} x_{4}^{6} x_{5}^{11} + x_{3} x_{4}^{6} x_{5}^{11}, \\
g_{12} &= x_{1} x_{2}^{14} x_{3}^{3} + x_{1} x_{2}^{14} x_{4}^{3} + x_{1} x_{3}^{14} x_{4}^{3} + x_{2} x_{3}^{14} x_{4}^{3} \\
    & + x_{1} x_{2}^{14} x_{5}^{3} + x_{1} x_{3}^{14} x_{5}^{3} + x_{2} x_{3}^{14} x_{5}^{3} + x_{1} x_{4}^{14} x_{5}^{3} \\
    & + x_{2} x_{4}^{14} x_{5}^{3} + x_{3} x_{4}^{14} x_{5}^{3}, \\
g_{13} &= x_{1} x_{2}^{2} x_{3}^{5} x_{4}^{8} x_{5}^{2} + x_{1} x_{2}^{2} x_{3}^{4} x_{4}^{9} x_{5}^{2} + x_{1} x_{2}^{2} x_{3}^{5} x_{4}^{2} x_{5}^{8} + x_{1} x_{2}^{2} x_{3} x_{4}^{6} x_{5}^{8} \\
    & + x_{1} x_{2}^{2} x_{3}^{4} x_{4} x_{5}^{10} + x_{1} x_{2}^{2} x_{3} x_{4}^{4} x_{5}^{10}, \\
g_{14} &= x_{1}^{3} x_{2}^{5} x_{3}^{10} + x_{1}^{3} x_{2}^{5} x_{4}^{10} + x_{1}^{3} x_{3}^{5} x_{4}^{10} + x_{2}^{3} x_{3}^{5} x_{4}^{10} \\
    & + x_{1}^{3} x_{2}^{5} x_{5}^{10} + x_{1}^{3} x_{3}^{5} x_{5}^{10} + x_{2}^{3} x_{3}^{5} x_{5}^{10} + x_{1}^{3} x_{4}^{5} x_{5}^{10} \\
    & + x_{2}^{3} x_{4}^{5} x_{5}^{10} + x_{3}^{3} x_{4}^{5} x_{5}^{10}, \\
g_{15} &= x_{1}^{3} x_{2}^{5} x_{3}^{2} x_{4}^{4} x_{5}^{4}, \\
g_{16} &= x_{1}^{3} x_{2} x_{3}^{4} x_{4}^{8} x_{5}^{2} + x_{1} x_{2}^{3} x_{3}^{4} x_{4}^{8} x_{5}^{2} + x_{1} x_{2}^{3} x_{3}^{4} x_{4}^{2} x_{5}^{8}, \\
g_{17} &= x_{1}^{3} x_{2}^{13} x_{3}^{2} + x_{1}^{3} x_{2}^{3} x_{3}^{12} + x_{1}^{3} x_{2}^{13} x_{4}^{2} + x_{1}^{3} x_{3}^{13} x_{4}^{2} \\
    & + x_{2}^{3} x_{3}^{13} x_{4}^{2} + x_{1}^{3} x_{2}^{3} x_{4}^{12} + x_{1}^{3} x_{3}^{3} x_{4}^{12} + x_{2}^{3} x_{3}^{3} x_{4}^{12} \\
    & + x_{1}^{3} x_{2}^{13} x_{5}^{2} + x_{1}^{3} x_{3}^{13} x_{5}^{2} + x_{2}^{3} x_{3}^{13} x_{5}^{2} + x_{1}^{3} x_{4}^{13} x_{5}^{2} \\
    & + x_{2}^{3} x_{4}^{13} x_{5}^{2} + x_{3}^{3} x_{4}^{13} x_{5}^{2} + x_{1}^{3} x_{2}^{3} x_{5}^{12} + x_{1}^{3} x_{3}^{3} x_{5}^{12} \\
    & + x_{2}^{3} x_{3}^{3} x_{5}^{12} + x_{1}^{3} x_{4}^{3} x_{5}^{12} + x_{2}^{3} x_{4}^{3} x_{5}^{12} + x_{3}^{3} x_{4}^{3} x_{5}^{12}, \\
g_{18} &= x_{1} x_{2}^{2} x_{3}^{4} x_{4}^{8} x_{5}^{3} + x_{1}^{7} x_{2} x_{3}^{2} x_{4}^{4} x_{5}^{4} + x_{1} x_{2}^{7} x_{3}^{2} x_{4}^{4} x_{5}^{4} + x_{1} x_{2}^{2} x_{3}^{7} x_{4}^{4} x_{5}^{4} \\
    & + x_{1} x_{2}^{2} x_{3}^{4} x_{4}^{7} x_{5}^{4} + x_{1} x_{2}^{2} x_{3}^{4} x_{4}^{4} x_{5}^{7} + x_{1} x_{2}^{2} x_{3}^{4} x_{4}^{3} x_{5}^{8}, \\
g_{19} &= x_{1}^{15} x_{2}^{3} + x_{1}^{3} x_{2}^{15} + x_{1}^{15} x_{3}^{3} + x_{2}^{15} x_{3}^{3} \\
    & + x_{1}^{3} x_{3}^{15} + x_{2}^{3} x_{3}^{15} + x_{1}^{15} x_{4}^{3} + x_{2}^{15} x_{4}^{3} \\
    & + x_{3}^{15} x_{4}^{3} + x_{1}^{3} x_{4}^{15} + x_{2}^{3} x_{4}^{15} + x_{3}^{3} x_{4}^{15} \\
    & + x_{1}^{15} x_{5}^{3} + x_{2}^{15} x_{5}^{3} + x_{3}^{15} x_{5}^{3} + x_{4}^{15} x_{5}^{3} \\
    & + x_{1}^{3} x_{5}^{15} + x_{2}^{3} x_{5}^{15} + x_{3}^{3} x_{5}^{15} + x_{4}^{3} x_{5}^{15}, \\
g_{20} &= x_{1}^{3} x_{2}^{4} x_{3} x_{4}^{2} x_{5}^{8}, \\
g_{21} &= x_{1} x_{2} x_{3}^{2} x_{4}^{12} x_{5}^{2}, \\
g_{22} &= x_{1} x_{2}^{6} x_{3} x_{4}^{2} x_{5}^{8}, \\
g_{23} &= x_{1} x_{2}^{3} x_{3}^{14} + x_{1} x_{2}^{3} x_{4}^{14} + x_{1} x_{3}^{3} x_{4}^{14} + x_{2} x_{3}^{3} x_{4}^{14} \\
    & + x_{1} x_{2}^{3} x_{5}^{14} + x_{1} x_{3}^{3} x_{5}^{14} + x_{2} x_{3}^{3} x_{5}^{14} + x_{1} x_{4}^{3} x_{5}^{14} \\
    & + x_{2} x_{4}^{3} x_{5}^{14} + x_{3} x_{4}^{3} x_{5}^{14}, \\
g_{24} &= x_{1} x_{2}^{2} x_{3}^{3} x_{4}^{4} x_{5}^{8}, \\
g_{25} &= x_{1}^{3} x_{2}^{5} x_{3}^{8} x_{4}^{2} + x_{1}^{3} x_{2} x_{3}^{12} x_{4}^{2} + x_{1} x_{2}^{3} x_{3}^{12} x_{4}^{2} + x_{1} x_{2} x_{3}^{14} x_{4}^{2} \\
    & + x_{1}^{3} x_{2}^{5} x_{3}^{2} x_{4}^{8} + x_{1} x_{2} x_{3}^{2} x_{4}^{14} + x_{1}^{3} x_{2}^{5} x_{3}^{8} x_{5}^{2} + x_{1}^{3} x_{2} x_{3}^{12} x_{5}^{2} \\
    & + x_{1} x_{2}^{3} x_{3}^{12} x_{5}^{2} + x_{1} x_{2} x_{3}^{14} x_{5}^{2} + x_{1}^{3} x_{2}^{5} x_{4}^{8} x_{5}^{2} + x_{1}^{3} x_{3}^{5} x_{4}^{8} x_{5}^{2} \\
    & + x_{2}^{3} x_{3}^{5} x_{4}^{8} x_{5}^{2} + x_{1}^{3} x_{2} x_{4}^{12} x_{5}^{2} + x_{1} x_{2}^{3} x_{4}^{12} x_{5}^{2} + x_{1}^{3} x_{3} x_{4}^{12} x_{5}^{2} \\
    & + x_{2}^{3} x_{3} x_{4}^{12} x_{5}^{2} + x_{1} x_{3}^{3} x_{4}^{12} x_{5}^{2} + x_{2} x_{3}^{3} x_{4}^{12} x_{5}^{2} + x_{1} x_{2} x_{4}^{14} x_{5}^{2} \\
    & + x_{1} x_{3} x_{4}^{14} x_{5}^{2} + x_{2} x_{3} x_{4}^{14} x_{5}^{2} + x_{1}^{3} x_{2}^{5} x_{3}^{2} x_{5}^{8} + x_{1}^{3} x_{2}^{5} x_{4}^{2} x_{5}^{8} \\
    & + x_{1}^{3} x_{3}^{5} x_{4}^{2} x_{5}^{8} + x_{2}^{3} x_{3}^{5} x_{4}^{2} x_{5}^{8} + x_{1} x_{2} x_{3}^{2} x_{5}^{14} + x_{1} x_{2} x_{4}^{2} x_{5}^{14} \\
    & + x_{1} x_{3} x_{4}^{2} x_{5}^{14} + x_{2} x_{3} x_{4}^{2} x_{5}^{14}, \\
g_{26} &= x_{1}^{3} x_{2} x_{3}^{6} x_{4}^{4} x_{5}^{4} + x_{1} x_{2}^{3} x_{3}^{6} x_{4}^{4} x_{5}^{4} + x_{1}^{3} x_{2} x_{3}^{4} x_{4}^{6} x_{5}^{4} + x_{1} x_{2}^{3} x_{3}^{4} x_{4}^{6} x_{5}^{4} \\
    & + x_{1} x_{2}^{2} x_{3}^{5} x_{4}^{6} x_{5}^{4} + x_{1}^{3} x_{2} x_{3}^{4} x_{4}^{4} x_{5}^{6} + x_{1} x_{2}^{3} x_{3}^{4} x_{4}^{4} x_{5}^{6} + x_{1} x_{2}^{2} x_{3}^{5} x_{4}^{4} x_{5}^{6} \\
    & + x_{1}^{3} x_{2} x_{3}^{4} x_{4}^{2} x_{5}^{8}, \\
g_{27} &= x_{1}^{3} x_{2} x_{3}^{2} x_{4}^{4} x_{5}^{8}, \\
g_{28} &= x_{1}^{3} x_{2} x_{3}^{14} + x_{1}^{3} x_{2} x_{4}^{14} + x_{1}^{3} x_{3} x_{4}^{14} + x_{2}^{3} x_{3} x_{4}^{14} \\
    & + x_{1}^{3} x_{2} x_{5}^{14} + x_{1}^{3} x_{3} x_{5}^{14} + x_{2}^{3} x_{3} x_{5}^{14} + x_{1}^{3} x_{4} x_{5}^{14} \\
    & + x_{2}^{3} x_{4} x_{5}^{14} + x_{3}^{3} x_{4} x_{5}^{14}, \\
g_{29} &= x_{1} x_{2} x_{3}^{6} x_{4}^{10} + x_{1} x_{2} x_{3}^{6} x_{5}^{10} + x_{1} x_{2} x_{4}^{6} x_{5}^{10} + x_{1} x_{3} x_{4}^{6} x_{5}^{10} \\
    & + x_{2} x_{3} x_{4}^{6} x_{5}^{10}, \\
g_{30} &= x_{1}^{7} x_{2}^{11} + x_{1}^{7} x_{3}^{11} + x_{2}^{7} x_{3}^{11} + x_{1}^{7} x_{4}^{11} \\
    & + x_{2}^{7} x_{4}^{11} + x_{3}^{7} x_{4}^{11} + x_{1}^{7} x_{5}^{11} + x_{2}^{7} x_{5}^{11} \\
    & + x_{3}^{7} x_{5}^{11} + x_{4}^{7} x_{5}^{11}, \\
g_{31} &= x_{1} x_{2}^{2} x_{3}^{12} x_{4} x_{5}^{2} + x_{1} x_{2}^{2} x_{3} x_{4}^{12} x_{5}^{2}.
\end{align*}
By the relation $\rho_5(g)\sim g,$ we deduce that $\beta_0 = \beta_2 = \beta_8 = \beta_9 = \beta_{13} = \beta_{15} = \beta_{25} = \beta_{27} = \beta_{29}$ and $\beta_k = 0$ otherwise. This implies that
$$ g\sim \beta_0(\mathcal R'_4 + g_2 + g_8 + g_9 + g_{13} + g_{15} + g_{25} + g_{27} + g_{29}) = \beta_0 \widetilde{\xi}_{n_0},$$
where $\mathcal R'_4 + g_2 + g_8 + g_9 + g_{13} + g_{15} + g_{25} + g_{27} + g_{29} = \widetilde{\xi}_{n_0}.$ Thus, 
$$((Q_{n_0= 18}^{\otimes 5})^{G(5)} = \langle [\widetilde{\xi}_{n_0}] \rangle.$$

$\bullet$ {\bf Computation of $(Q_{n_1 = 41}^{\otimes 5})^{G(5)}.$} Take any $[h]\in ((Q_{n_1= 41}^{\otimes 5})^{G(5)}$ with representative $h\in (\mathcal P_{5})_{n_1}$. Appealing to Proposition \ref{kq3}(iii) and using that $(\widetilde {Sq^0_*})_{(5, n_1)}$ is a $G(5)$-epimorphism, we have
\[
((\widetilde {Sq^0_*})_{(5, n_1)}([h])=\gamma\,[\varphi(\widetilde{\xi}_{n_0})],\quad \gamma\in\mathbb F_2.
\]
where $\varphi$ is the up Kameko map 
\xymatrix{
\mathcal P_5\ar[r]& \mathcal P_5,\, u\longmapsto x_1x_2\ldots x_5u^{2}.
}
Consequently, $$ h\sim \gamma \varphi(\widetilde{\xi}_{n_0}) + {\boldmath{p}},$$
where ${\boldmath{p}}\in (\mathcal P_5)_{n_1}$ such that $[{\boldmath{p}}]\in {\rm Ker}((\widetilde {Sq^0_*})_{(5, n_1)}).$
By a direct calculation using the relations $\rho_j(h)\sim h,\, 1\leq j\leq 5,$ we obtain
$$h\sim \gamma (\varphi(\widetilde{\xi}_{n_0}) + \widetilde{\xi}_{n_1}),$$
which implies that 
$$(Q^{\otimes 5}_{n_1})^{G(5)} = \langle [\varphi(\widetilde{\xi}_{n_0}) + \widetilde{\xi}_{n_1}] \rangle.$$
The proof of Theorem \ref{dlc2} is complete.
\end{proof}

\subsection{Proof of Theorem \ref{dlc3}}

It is well to note that in this subsection, the Kameko homomorphism 
\xymatrix{
(\widetilde {Sq^0_*})_{(m, m)}: Q^{\otimes m}_m\ar[r]& Q^{\otimes m}_0 = \mathbb F_2} is an epimorphism for every $m,$ and so,  $$Q^{\otimes m}_m\cong {\rm Ker}(\widetilde {Sq^0_*})_{(m, m)}\bigoplus \langle [x_1x_2\ldots x_m]_{(m, 0)} \rangle.$$ 

\medskip

For $n = 1,$ and $x\in \mathscr C^{\otimes m}_1,$ one has that $\mathsf{Param}(x) := \mathsf{Param} = (1,0)$ and $ \dim (Q^{\otimes m}_1)^{\mathsf{Param}} = \dim Q^{\otimes m}_1=m$ for any $m\geq 1.$ Further, $(Q^{\otimes m}_1)^{\mathsf{Param}}  = \langle \{[x_i]_{\mathsf{Param}} = [x_i]:\, 1\leq i\leq m\} \rangle.$

\medskip

 For $n = 2,$ and $x\in \mathscr C^{\otimes m}_2,$ then $\mathsf{Param}(x) := \mathsf{Param} = (2,0).$ So, from a result in \cite{MKR}, we deduce that $\dim (Q^{\otimes m}_2)^{\mathsf{Param}} = \dim Q^{\otimes m}_2 = \binom{m}{2}$ for any $m\geq 2.$ (Note that since $\beta(2)  = 2 > 1,$ by Theorem \ref{dlWK}, $Q^{\otimes 1}_2 = 0.$) Further, $(Q^{\otimes m}_2)^{\mathsf{Param}} = \langle \{[x_ix_j]_{\mathsf{Param}} = [x_ix_j]:\, 1\leq i < j\leq m\} \rangle.$

\medskip

For $n = 3,$ and $x\in \mathscr C^{\otimes m}_3,$ as we will see, either $\mathsf{Param}(x):= \mathsf{Param}_{(1)} = (1,1)$ or $\mathsf{Param}(x):= \mathsf{Param}_{(2)} = (3,0)$ for any $m\geq 1.$ By Peterson \cite{F.P}, one gets the following: For $m = 1,$ we have $\dim (Q^{\otimes 1}_3)^{\mathsf{Param}_{(1)}} = 1$ and $\dim (Q^{\otimes 1}_3)^{\mathsf{Param}_{(2)}} = 0.$ For $m = 2,$ we have $\dim (Q^{\otimes 2}_3)^{\mathsf{Param}_{(1)}} = 3$ and $\dim (Q^{\otimes 2}_3)^{\mathsf{Param}_{(2)}} = 0.$ For $m = 3,$ we have an isomorphism $Q^{\otimes 3}_3\cong {\rm Ker}(\widetilde {Sq^0_*})_{(3, 3)}\bigoplus \langle [x_1x_2x_3]_{\mathsf{Param}_{(2)}} \rangle,$ where ${\rm Ker}(\widetilde {Sq^0_*})_{(3, 3)} = (Q^{\otimes 3}_3)^{\mathsf{Param}_{(1)}}$ and $\langle [x_1x_2x_3]_{\mathsf{Param}_{(2)}} \rangle = (Q^{\otimes 3}_3)^{\mathsf{Param}_{(2)}}.$ So, by Kameko \cite{M.K}, $\dim  (Q^{\otimes 3}_3)^{\mathsf{Param}_{(1)}}  = 6.$ For $m \geq 4,$ as $Q^{\otimes m}_3 = (Q^{\otimes m}_3)^{0},$ we get $$\dim (Q^{\otimes m}_3)^{\mathsf{Param}_{(1)}} = \binom{m}{1} + \binom{m}{2}\ \mbox{and}\ \dim (Q^{\otimes m}_3)^{\mathsf{Param}_{(2)}} = \binom{m}{3}.$$ Further, 
$$ \begin{array}{ll}
\medskip
& (Q^{\otimes m}_3)^{\mathsf{Param}_{(1)}} = \langle \{[x_ix_j^{2}]_{\mathsf{Param}_{(1)}}:\, 1\leq i\leq j\leq m,\, 1\leq k\leq m\} \rangle, \ \mbox{and} \\ 
& (Q^{\otimes m}_3)^{\mathsf{Param}_{(2)}} = \langle \{[x_ix_jx_t]_{\mathsf{Param}_{(2)}}:\, 1\leq i<j<t\leq m\}\rangle.
\end{array}$$

\medskip

For $n = 4,$ and $x\in \mathscr C^{\otimes m}_4,$ by Theorem \ref{dlWK}(I), $Q^{\otimes 1}_4 = 0.$ For each $m\geq 2,$ it is straightforward to see that either  $\mathsf{Param}(x):= \mathsf{Param}_{(1)} = (2,1)$ or $\mathsf{Param}(x):= \mathsf{Param}_{(2)} = (4,0).$ For $m=2,$ by Theorem \ref{dlWK}(II), $Q^{\otimes 2}_4\cong Q^{\otimes 2}_1,$ and so, $\dim (Q^{\otimes 2}_4)^{\mathsf{Param}_{(1)}} = 2$ and $\dim (Q^{\otimes 2}_4)^{\mathsf{Param}_{(2)}} = 0.$ For $m=3,$ by Kameko \cite{M.K}, $\dim (Q^{\otimes 3}_4)^{\mathsf{Param}_{(1)}} = 2$ and $\dim (Q^{\otimes 3}_4)^{\mathsf{Param}_{(2)}} = 6.$ For $m=4,$ one has an isomorphism $Q^{\otimes 4}_4\cong {\rm Ker}(\widetilde {Sq^0_*})_{(4, 4)}\bigoplus \langle [x_1x_2x_3x_4]_{\mathsf{Param}_{(2)}} \rangle,$ where ${\rm Ker}(\widetilde {Sq^0_*})_{(4, 4)} = (Q^{\otimes 4}_4)^{\mathsf{Param}_{(1)}}$ and $\langle [x_1x_2x_3x_4]_{\mathsf{Param}_{(2)}} \rangle = (Q^{\otimes 4}_4)^{\mathsf{Param}_{(2)}}.$ So, by Sum \cite{N.S1}, $\dim  (Q^{\otimes 4}_4)^{\mathsf{Param}_{(1)}}  = 16.$ For $m \geq 5,$ as $Q^{\otimes m}_4 = (Q^{\otimes m}_4)^{0},$ we get $$\dim (Q^{\otimes m}_4)^{\mathsf{Param}_{(1)}} = 2.\binom{m}{2} + 2.\binom{m}{3}\ \mbox{and}\ \dim (Q^{\otimes m}_4)^{\mathsf{Param}_{(2)}} = \binom{m}{4}.$$ Further, 
$$ \begin{array}{ll}
\medskip
&(Q^{\otimes m}_4)^{\mathsf{Param}_{(1)}} = \langle \{[x_ix_j^{2}x_k]_{\mathsf{Param}_{(1)}}:\, i<j,\, i\neq k,\, 1\leq i,\,j,\, k\leq m\} \rangle,\ \mbox{and}\\
&(Q^{\otimes m}_4)^{\mathsf{Param}_{(2)}} = \langle \{[x_ix_jx_kx_{\ell}]_{\mathsf{Param}_{(2)}}:\, 1\leq i<j<k<\ell\leq m\}\rangle.
\end{array}$$

\medskip

For $n = 5,$ and $x\in \mathscr C^{\otimes m}_5,$ by Theorem \ref{dlWK}(I), $Q^{\otimes 1}_5 = 0 = Q^{\otimes 2}_5.$ For each $m\geq 3,$ direct calculations point out that either $\mathsf{Param}(x):= \mathsf{Param}_{(1)} = (3,1)$ or $\mathsf{Param}(x):= \mathsf{Param}_{(2)} = (5,0).$  For $m=3,$ by Theorem \ref{dlWK}(II), $Q^{\otimes 3}_5\cong Q^{\otimes 3}_1,$ and so, by Kameko \cite{M.K}, $\dim (Q^{\otimes 3}_5)^{\mathsf{Param}_{(1)}} = 3$ and $\dim (Q^{\otimes 3}_5)^{\mathsf{Param}_{(2)}} = 0.$ For $m = 4,$ by Sum \cite{N.S1}, $\dim (Q^{\otimes 4}_5)^{\mathsf{Param}_{(1)}} = 15$ and $\dim (Q^{\otimes 4}_5)^{\mathsf{Param}_{(2)}} = 0.$ For $m = 5,$ one has an isomorphism $Q^{\otimes 5}_5\cong {\rm Ker}(\widetilde {Sq^0_*})_{(5, 5)}\bigoplus \langle [x_1x_2x_3x_4x_5]_{\mathsf{Param}_{(2)} } \rangle,$ where ${\rm Ker}(\widetilde {Sq^0_*})_{(5, 5)} = (Q^{\otimes 5}_5)^{\mathsf{Param}_{(1)}}$ and $\langle [x_1x_2x_3x_4x_5]_{\mathsf{Param}_{(2)}} \rangle = (Q^{\otimes 5}_5)^{\mathsf{Param}_{(2)}}.$ So, by Sum \cite{N.S0}, $\dim  (Q^{\otimes 5}_5)^{\mathsf{Param}_{(1)}}  = 45.$ For $m \geq 6,$ as $Q^{\otimes m}_5 = (Q^{\otimes m}_5)^{0},$ we get $$\dim (Q^{\otimes m}_5)^{\mathsf{Param}_{(1)}} = 3.\binom{m}{3} + 3.\binom{m}{4}\ \mbox{and}\ \dim (Q^{\otimes m}_5)^{\mathsf{Param}_{(2)}} = \binom{m}{5}.$$ Further, 
$$ \begin{array}{ll}
&(Q^{\otimes m}_5)^{\mathsf{Param}_{(1)}} = \langle \{[x_ix_j^{2}x_kx_{\ell}]_{\mathsf{Param}_{(1)}},\, [x_ix_jx_k^{2}x_{\ell}]_{\mathsf{Param}_{(1)}},\, [x_ix_jx_kx_{\ell}^{2}]_{\mathsf{Param}_{(1)}},\, [x_ux_vx_w^{3}]_{\mathsf{Param}_{(1)}}:\\
\medskip
&\hspace{3cm} 1\leq i<j<k<\ell\leq m,\, 1\leq u<v<w\leq m\} \rangle, \ \mbox{and}\\
&(Q^{\otimes m}_5)^{\mathsf{Param}_{(2)}} = \langle \{[x_ix_jx_kx_{\ell}x_t]_{\mathsf{Param}_{(2)}}:\, 1\leq i<j<k<\ell<t\leq m\}\rangle.
\end{array}$$

\medskip

For $n = 6,$ and $x\in \mathscr C^{\otimes m}_6,$ by Theorem \ref{dlWK}(I), $Q^{\otimes 1}_6 = 0.$ For each $m\geq 2,$ A direct computation shows that either $\mathsf{Param}(x):= \mathsf{Param}_{(1)} = (2,2)$ or $\mathsf{Param}(x):= \mathsf{Param}_{(2)} = (4,1)$ or $\mathsf{Param}(x):= \mathsf{Param}_{(2)} = (6,0).$ For $m = 2,$  according to Theorem \ref{dlWK}(II), we have an isomorphism $Q^{\otimes 2}_6\cong Q^{\otimes 2}_2,$ and so, due to Peterson \cite{F.P}, $\dim (Q^{\otimes 2}_6)^{\mathsf{Param}_{(1)}} = 1$ and $\dim (Q^{\otimes 2}_6)^{\mathsf{Param}_{(j)}} = 0$ for $j = 2,\, 3.$ For $m = 3,$ by Kameko \cite{M.K}, $\dim (Q^{\otimes 3}_6)^{\mathsf{Param}_{(1)}} = 6$ and $\dim (Q^{\otimes 3}_6)^{\mathsf{Param}_{(j)}} = 0$ for $j = 2,\, 3.$ For $m  =4,$ due to Sum \cite{N.S1}, $\dim (Q^{\otimes 4}_6)^{\mathsf{Param}_{(1)}} = 20$, $\dim (Q^{\otimes 4}_6)^{\mathsf{Param}_{(2)}} = 4$ and $\dim (Q^{\otimes 4}_6)^{\mathsf{Param}_{(3)}} =0.$ For $m = 5,$ by our previous work \cite{D.P1}, $\dim (Q^{\otimes 5}_6)^{\mathsf{Param}_{(1)}} = 50$, $\dim (Q^{\otimes 5}_6)^{\mathsf{Param}_{(2)}} = 24$ and $\dim (Q^{\otimes 5}_6)^{\mathsf{Param}_{(3)}} =0.$ For $m = 6,$ one has an isomorphism $Q^{\otimes 6}_6\cong {\rm Ker}(\widetilde {Sq^0_*})_{(6, 6)}\bigoplus \langle [x_1x_2x_3x_4x_5x_6]_{\mathsf{Param}_{(3)}} \rangle,$ where ${\rm Ker}(\widetilde {Sq^0_*})_{(6, 6)} = (Q^{\otimes 6}_6)^{\mathsf{Param}_{(1)}}\bigoplus (Q^{\otimes 6}_6)^{\mathsf{Param}_{(2)}}$ and $\langle [x_1x_2x_3x_4x_5x_6]_{\mathsf{Param}_{(3)}} \rangle = (Q^{\otimes 6}_6)^{\mathsf{Param}_{(3)}}.$ So, a simple computation shows: $\dim (Q^{\otimes 6}_6)^{\mathsf{Param}_{(1)}} = \binom{6}{2} + 3.\binom{6}{3} + 2.\binom{6}{4}=105$ and $\dim (Q^{\otimes 6}_6)^{\mathsf{Param}_{(2)}} =  4.\binom{6}{4} + 4.\binom{6}{5} = 84.$ For $m \geq 7,$ as $Q^{\otimes m}_6 = (Q^{\otimes m}_6)^{0},$ we get 
$$ \begin{array}{ll}
\medskip
& \dim (Q^{\otimes m}_6)^{\mathsf{Param}_{(1)}} = \binom{m}{2} +  3.\binom{m}{3} + 2.\binom{m}{4},\ \ \dim (Q^{\otimes m}_6)^{\mathsf{Param}_{(2)}} = 4.\binom{m}{4} + 4.\binom{m}{5}, \ \mbox{and}\\
 &\dim (Q^{\otimes m}_6)^{\mathsf{Param}_{(3)}} = \binom{m}{6}.
\end{array}$$ Further, 
$$ \begin{array}{ll}
&(Q^{\otimes m}_6)^{\mathsf{Param}_{(1)}} = \langle \{[x_ix_jx_k^{2}x_{\ell}^{2}]_{\mathsf{Param}_{(1)}},\, [x_ix_j^{2}x_kx_{\ell}^{2}]_{\mathsf{Param}_{(1)}} ,\, [x_px_q^{2}x_r^{3}]_{\mathsf{Param}_{(1)}},\, [x_s^{3}x_t^{3}]_{\mathsf{Param}_{(1)}}:\\
\medskip
&\hspace{1cm} 1\leq i<j<k<\ell\leq m,\, p<q,\, p\neq r,\, q\neq r,\, 1\leq p,\, q,\, r\leq m,\,  1\leq s<t\leq m\} \rangle,\\
&(Q^{\otimes m}_6)^{\mathsf{Param}_{(2)}} = \langle \{[x_ix_jx_kx_{\ell}x_p^{2}]_{\mathsf{Param}_{(2)}},\, [x_ix_jx_kx_{\ell}^{2}x_p]_{\mathsf{Param}_{(2)}},\, [x_ix_jx_k^{2}x_{\ell}x_p]_{\mathsf{Param}_{(2)}},\\
\medskip
&\hspace{1cm}[x_ix_j^{2}x_kx_{\ell}x_p]_{\mathsf{Param}_{(2)}},\, [x_qx_rx_sx_t^{3}]_{\mathsf{Param}_{(2)}}: 1\leq i<j<k<\ell<p\leq m,\, 1\leq q<r<s<t\leq m\}\rangle,\\
&(Q^{\otimes m}_6)^{\mathsf{Param}_{(3)}} = \langle \{[x_ix_jx_kx_{\ell}x_px_q]_{\mathsf{Param}_{(3)}}:\, 1\leq i<j<k<\ell<p<q\leq m\}\rangle.
\end{array}$$

\medskip

For $n = 7,$ and $x\in \mathscr C^{\otimes m}_7,$ by routine calculations, one obtains that the parameter vector $\mathsf{Param}(x)$ is one of the following sequences:
$$ \mathsf{Param}_{(1)}:= (1,1,1), \ \ \mathsf{Param}_{(2)}:= (1,3), \ \ \mathsf{Param}_{(3)}:= (3,2), \ \ \mathsf{Param}_{(4)}:= (5,1), \ \ \mathsf{Param}_{(5)}:= (7,0).$$
In particular,
$$ \mathsf{Param}(x)\in  \left\{ \begin{array}{ll}
\{\mathsf{Param}_{(1)}\}&\mbox{if $m = 1,\, 2$\ (see also \cite{F.P}),} \\[1mm]
\{\mathsf{Param}_{(1)},\, \mathsf{Param}_{(3)} \}&\mbox{if $m=3$\ (see also \cite{M.K}),} \\[1mm]
\{\mathsf{Param}_{(j)}:\, 1\leq j\leq 3\}&\mbox{if $m=4$\ (see also \cite{N.S1}),} \\[1mm]
\{\mathsf{Param}_{(j)}:\, 1\leq j\leq 4\}&\mbox{if $m=5$\ (see also \cite{N.T}),} \\[1mm]
\{\mathsf{Param}_{(j)}:\, 1\leq j\leq 4\}&\mbox{if $m=6$},\\[1mm]
\{\mathsf{Param}_{(j)}:\, 1\leq j\leq 5\}&\mbox{if $m\geq 7.$}
\end{array}\right.$$
For $m = 1,\, 2,$ by Peterson \cite{F.P}, $(Q^{\otimes m}_7)^{\mathsf{Param}_{(1)}}$ is 1-dimensional if $m = 1$ and is 3-dimensional if $m = 2.$ For $m = 3,$ by Kameko \cite{M.K}, $\dim (Q^{\otimes 3}_7)^{\mathsf{Param}_{(1)}} = 7$ and $\dim (Q^{\otimes 3}_7)^{\mathsf{Param}_{(3)}} = 3.$ For $m = 4,$ by Sum \cite{N.S1}, $\dim (Q^{\otimes 4}_7)^{\mathsf{Param}_{(1)}} = 14,$ $\dim (Q^{\otimes 4}_7)^{\mathsf{Param}_{(2)}} = 1$ and $\dim (Q^{\otimes 4}_7)^{\mathsf{Param}_{(3)}} = 20.$ For $m = 5,$ by Tin \cite{N.T}, $\dim (Q^{\otimes 5}_7)^{\mathsf{Param}_{(1)}} = 25,$ $\dim (Q^{\otimes 5}_7)^{\mathsf{Param}_{(2)}} = 5,$ $\dim (Q^{\otimes 5}_7)^{\mathsf{Param}_{(3)}} = 75$ and $\dim (Q^{\otimes 5}_7)^{\mathsf{Param}_{(4)}} = 5.$ For $m = 6,$ by direct calculations, we get $\dim (Q^{\otimes 6}_7)^{\mathsf{Param}_{(1)}} = \binom{6}{1} + \binom{6}{2} + \binom{6}{3} =  41,$ $\dim (Q^{\otimes 6}_7)^{\mathsf{Param}_{(2)}} = \binom{6}{4} = 15,$ $\dim (Q^{\otimes 6}_7)^{\mathsf{Param}_{(3)}} = 3.\binom{6}{3} + 8.\binom{6}{4} + 5.\binom{6}{5} = 210$ and $\dim (Q^{\otimes 6}_7)^{\mathsf{Param}_{(4)}} = 5.\binom{6}{5} + 5.\binom{6}{6} = 35.$ For $m = 7,$ one has an isomorphism $Q^{\otimes 7}_7\cong {\rm Ker}(\widetilde {Sq^0_*})_{(7, 7)}\bigoplus \langle [x_1x_2x_3x_4x_5x_6x_7]_{\mathsf{Param}_{(5)}} \rangle,$ where ${\rm Ker}(\widetilde {Sq^0_*})_{(7, 7)} = \bigoplus_{1\leq j\leq 4}(Q^{\otimes 7}_7)^{\mathsf{Param}_{(j)}}$ and $\langle [x_1x_2x_3x_4x_5x_6x_7]_{\mathsf{Param}_{(5)}} \rangle = (Q^{\otimes 7}_7)^{\mathsf{Param}_{(5)}}.$ So, a simple computation shows: $\dim (Q^{\otimes 7}_7)^{\mathsf{Param}_{(1)}} = \binom{7}{1} + \binom{7}{2} + \binom{7}{3} =  63,$ $\dim (Q^{\otimes 7}_7)^{\mathsf{Param}_{(2)}} = \binom{7}{4} = 35,$ $\dim (Q^{\otimes 7}_7)^{\mathsf{Param}_{(3)}} = 3.\binom{7}{3} + 8.\binom{7}{4} + 5.\binom{7}{5} = 490$ and $\dim (Q^{\otimes 7}_7)^{\mathsf{Param}_{(4)}} = 5.\binom{7}{5} + 5.\binom{7}{6} = 140.$

\medskip

 For $m\geq 8,$ as $Q^{\otimes m}_7 = (Q^{\otimes m}_7)^{0},$ we get 
$$ \begin{array}{ll}
\medskip
& \dim (Q^{\otimes m}_7)^{\mathsf{Param}_{(1)}} = \binom{m}{1} +  \binom{m}{2} + \binom{m}{3},\ \ \dim (Q^{\otimes m}_7)^{\mathsf{Param}_{(2)}} = \binom{m}{4},\\
\medskip
 &\dim (Q^{\otimes m}_7)^{\mathsf{Param}_{(3)}} = 3.\binom{m}{3} + 8.\binom{m}{4} + 5.\binom{m}{5},\ \ \dim (Q^{\otimes m}_7)^{\mathsf{Param}_{(4)}} = 5.\binom{m}{5} + 5.\binom{m}{6},\\
 &\dim (Q^{\otimes m}_7)^{\mathsf{Param}_{(5)}} = \binom{m}{7}.
\end{array}$$ 

\medskip

For $n = 8,$ and $x\in \mathscr C^{\otimes m}_8,$ by Theorem \ref{dlWK}(I), $Q^{\otimes 1}_8 = 0.$ For each $m\geq 2,$ elementary computations in these cases show that the parameter vector $\mathsf{Param}(x)$ is one of the following sequences:
$$ \mathsf{Param}_{(1)}:= (2,1,1), \ \ \mathsf{Param}_{(2)}:= (2,3), \ \ \mathsf{Param}_{(3)}:= (4,2), \ \ \mathsf{Param}_{(4)}:= (6,1), \ \ \mathsf{Param}_{(5)}:= (8,0).$$
In particular,
$$ \mathsf{Param}(x)\in  \left\{ \begin{array}{ll}
\{\mathsf{Param}_{(1)}\}&\mbox{if $m = 2,\, 3$\ (see also \cite{F.P, M.K}),} \\[1mm]
\{\mathsf{Param}_{(j)}:\, 1\leq j\leq 3\}&\mbox{if $m=4,\, 5$\ (see also \cite{N.S1, N.T}),} \\[1mm]
\{\mathsf{Param}_{(j)}:\, 1\leq j\leq 4\}&\mbox{if $m = 6,\, 7$},\\[1mm]
\{\mathsf{Param}_{(j)}:\, 1\leq j\leq 5\}&\mbox{if $m\geq 8.$}
\end{array}\right.$$
For $m = 2,$ due to Peterson \cite{F.P}, $\dim (Q^{\otimes 2}_8)^{\mathsf{Param}_{(1)}} = 3.$ For $m = 3,$ according to Kameko \cite{M.K}, $\dim (Q^{\otimes 3}_8)^{\mathsf{Param}_{(1)}} = 15.$ For $m = 4,$ by Sum \cite{N.S1}, $\dim (Q^{\otimes 4}_8)^{\mathsf{Param}_{(1)}} =45,$ $\dim (Q^{\otimes 4}_8)^{\mathsf{Param}_{(2)}} =4$ and $\dim (Q^{\otimes 4}_8)^{\mathsf{Param}_{(3)}} =6.$ For $m = 5,$ due to Tin \cite{N.T}, $\dim (Q^{\otimes 5}_8)^{\mathsf{Param}_{(1)}} =105,$ $\dim (Q^{\otimes 5}_8)^{\mathsf{Param}_{(2)}} =24$ and $\dim (Q^{\otimes 5}_8)^{\mathsf{Param}_{(3)}} =45.$ For $m = 6,$ from our previous work \cite{D.P10-2}, we  $\dim (Q^{\otimes 6}_8)^{\mathsf{Param}_{(1)}} = 210,$ $\dim (Q^{\otimes 6}_8)^{\mathsf{Param}_{(2)}} = 84$ and $\dim (Q^{\otimes 6}_8)^{\mathsf{Param}_{(3)}} = 189.$ For the other inclusion, we observe that $Q^{\otimes 6}_8 \cong {\rm Ker}((\widetilde {Sq^0_*})_{(6, 8)})\bigoplus Q^{\otimes 6}_1,$ where ${\rm Ker}((\widetilde {Sq^0_*})_{(6, 8)})\cong \bigoplus_{1\leq i\leq 3}(Q^{\otimes 6}_8)^{\mathsf{Param}_{(i)}}.$ So, $(Q^{\otimes 6}_8)^{\mathsf{Param}_{(4)}}\cong \varphi(Q^{\otimes 6}_1) = \langle \{[x_1\ldots x_6u^{2}]_{\mathsf{Param}_{(4)}}:\, u\in \mathscr C^{\otimes 6}_1\} \rangle$ where 
\xymatrix{
\varphi: Q^{\otimes 6}_1\ar[r]& Q^{\otimes 6}_8,\, [u]\longmapsto [x_1x_2\ldots x_6u^{2}]} and $\mathscr C^{\otimes 6}_1 = \{x_i:\, 1\leq i\leq 6\}.$ We see therefore that $\dim (Q^{\otimes 6}_8)^{\mathsf{Param}_{(4)}} = 6.$ For $m = 7,$ by direct calculations, we get:
$$ \begin{array}{ll}
\medskip
& \dim (Q^{\otimes 7}_8)^{\mathsf{Param}_{(1)}} =3.\binom{7}{2} + 6.\binom{7}{3} + 3.\binom{7}{4} = 378 ,\ \ \dim (Q^{\otimes 7}_8)^{\mathsf{Param}_{(2)}} = 4.\binom{7}{4} + 4.\binom{7}{5} = 224,\\
&\dim (Q^{\otimes 7}_8)^{\mathsf{Param}_{(3)}} = 6.\binom{7}{4} + 15.\binom{7}{5} + 9.\binom{7}{6} = 588, \ \ \dim (Q^{\otimes 7}_8)^{\mathsf{Param}_{(4)}} = 6.\binom{7}{6} +6.\binom{7}{7}= 48.
\end{array}$$
For $m = 8,$ one has that $Q^{\otimes 8}_8\cong {\rm Ker}(\widetilde {Sq^0_*})_{(8, 8)}\bigoplus \langle [x_1x_2x_3x_4x_5x_6x_7x_8]_{\mathsf{Param}_{(5)}} \rangle,$ where $${\rm Ker}(\widetilde {Sq^0_*})_{(8, 8)} = \bigoplus_{1\leq j\leq 4}(Q^{\otimes 8}_8)^{\mathsf{Param}_{(j)}}\ \mbox{and}\ \langle [x_1x_2x_3x_4x_5x_6x_7x_8]_{\mathsf{Param}_{(5)}} \rangle = (Q^{\otimes 8}_8)^{\mathsf{Param}_{(5)}}.$$ So, by direct calculations, we obtain: 
$$ \begin{array}{ll}
\medskip
& \dim (Q^{\otimes 8}_8)^{\mathsf{Param}_{(1)}} =3.\binom{8}{2} + 6.\binom{8}{3} + 3.\binom{8}{4} = 630 ,\ \ \dim (Q^{\otimes 8}_8)^{\mathsf{Param}_{(2)}} = 4.\binom{8}{4} + 4.\binom{8}{5} = 504,\\
&\dim (Q^{\otimes 8}_8)^{\mathsf{Param}_{(3)}} = 6.\binom{8}{4} + 15.\binom{8}{5} + 9.\binom{8}{6} = 1512, \ \ \dim (Q^{\otimes 8}_8)^{\mathsf{Param}_{(4)}} = 6.\binom{8}{6} +6.\binom{8}{7}= 216.
\end{array}$$
For $m\geq 9,$ as $Q^{\otimes m}_8 = (Q^{\otimes m}_8)^{0},$ we get 
$$ \begin{array}{ll}
\medskip
& \dim (Q^{\otimes m}_8)^{\mathsf{Param}_{(1)}} = 3.\binom{m}{2} + 6.\binom{m}{3} + 3.\binom{m}{4},\\
\medskip
 & \dim (Q^{\otimes m}_8)^{\mathsf{Param}_{(2)}} = 4.\binom{m}{4} + 4.\binom{m}{5} ,\\
\medskip
 &\dim (Q^{\otimes m}_8)^{\mathsf{Param}_{(3)}} = 6.\binom{m}{4} + 15.\binom{m}{5} + 9.\binom{m}{6},\\ 
\medskip
&\dim (Q^{\otimes m}_8)^{\mathsf{Param}_{(4)}} = 6.\binom{m}{6} +6.\binom{m}{7},\\
 &\dim (Q^{\otimes m}_8)^{\mathsf{Param}_{(5)}} = \binom{m}{8}.
\end{array}$$ 

We observe that as $\beta(12) = 4,$ by Theorem \ref{dlWK}(I), $Q^{\otimes m}_{12} = 0$ for $m\leq 3.$ As it is known, $Q^{\otimes m}_{12}\cong (Q^{\otimes m}_{12})^{0}\bigoplus (Q^{\otimes m}_{12})^{>0},$ and so, we need to determine the dimensions of $(Q^{\otimes m}_{12})^{0}$ and $(Q^{\otimes m}_{12})^{>0}.$ Thanks to the results by \cite{MKR} and \cite{D.P10-2}, we obtain

\begin{corl}\label{hqP}
For each integer $m\geq 7,$ suppose $X$ is an admissible monomial in $(\mathcal P_m)_{12}.$ Then the parameter vector of $X$ is one of the following sequences:
$$ \begin{array}{ll}
\medskip
\mathsf{Param}_{(1)}:= (4,2,1),\ \ \mathsf{Param}_{(2)}:= (4,4),\ \ \mathsf{Param}_{(3)}:= (6,1,1),\ \ \mathsf{Param}_{(4)}:= (6,3),\\
\mathsf{Param}_{(5)}:= (8,2),\ \ \mathsf{Param}_{(6)}:= (10,1),\ \ \mathsf{Param}_{(7)}:= (12,0). 
\end{array}$$
\end{corl}

Indeed, as is well-known, $(Q^{\otimes m}_{12})^{0} =  \langle[\widetilde {\Phi^0}(\mathscr C^{\otimes (m-1)}_{12})]\rangle,$ and for every $m > 12,$ we have $Q^{\otimes m}_{12} = (Q^{\otimes m}_{12})^{0}.$ So, we only need to determine parameter vectors of $X\in (\mathcal P_m)_{12}$ with $4\leq m\leq 12.$ A direct calculation shows:
$$\mathsf{Param}(X)\in  \left\{ \begin{array}{ll}
\{\mathsf{Param}_{(j)}:\, 1\leq j\leq 4\}&\mbox{if $m=7,$}\\[1mm]
\{\mathsf{Param}_{(j)}:\, 1\leq j\leq 5\}&\mbox{if $8\leq m\leq 9,$}\\[1mm]
\{\mathsf{Param}_{(j)}:\, 1\leq j\leq 6\}&\mbox{if $10\leq m\leq 11,$}\\[1mm]
\{\mathsf{Param}_{(j)}:\, 1\leq j\leq 7\}&\mbox{if $m = 12.$}
\end{array}\right.$$
For instance, consider $m = 7,$ since $\beta(12) = 4$ and $\deg(X) = 12$ is an even number, either $\mathsf{Param}_1(X) = 4$ or $\mathsf{Param}_1(X) = 6.$ If $\mathsf{Param}_1(X) = 4,$ then $X$ is of the form $x_ax_bx_cx_{d}u^{2}$ where $1\leq a<b<c<d\leq 7$ and $u\in (\mathcal P_7)_{4} = (\mathcal P_7)^{0}_{4}.$ As $X$ is admissible, by Theorem \ref{dlK}, $u$ is also an admissible monomial. According to \cite{MKR}, we have $\dim Q^{\otimes 7}_4 = 2.\binom{7}{2} + 2.\binom{7}{3} + \binom{7}{4} = 147.$ 

A monomial basis for $Q^{\otimes 7}_4$ is a set consisting of all classes represented by the admissible monomials $u$ of the following forms:
$$ \begin{array}{ll}
\medskip
&x_ix_j^{3},\ \ \mbox{for $1\leq i,\, j\leq 7,\, i\neq j,$}\\
\medskip
&x_ix_jx_kx_{\ell},\ \ \mbox{for $1\leq i < j<k<\ell\leq 7,$}\\
&x_ix_j^{2}x_k,\ \ \mbox{for $i<j,\, 1\leq i,\, j,\, k\leq 7,\, i\neq k,\, j\neq k.$}
\end{array}$$  
Obviously $\mathsf{Param}(x_ix_j^{3}) = \mathsf{Param}(x_ix_j^{2}x_k) = (2,1)$ and $\mathsf{Param}(x_ix_jx_kx_{\ell}) = (4),$ which imply that either $\mathsf{Param}(u) = (2,1)$ or $\mathsf{Param}(u) = (4).$ So, either $\mathsf{Param}(X) = \mathsf{Param}_{(1)}$ or $\mathsf{Param}(X) = \mathsf{Param}_{(2)}.$ Similarly, if $\mathsf{Param}_1(X) = 6,$ then $X$ can be represented as $x_ax_bx_cx_dx_ex_fv^{2}$ where $v$ is an admissible monomial of degree $3$ in $\mathcal P_7,$ and $1\leq a<b<c<d<e<f\leq 7.$ According to the calculations carried out in \cite{MKR}, the dimension of $Q^{\otimes 7}_3$ is determined to be equal to $\sum_{1\leq j\leq 3}\binom{7}{j} = 63.$ Then the parameter vector of $v$ is either $(1,1)$ or $(3)$. This leads to either $\mathsf{Param}(X) = \mathsf{Param}_{(3)}$ or $\mathsf{Param}(X) = \mathsf{Param}_{(4)}.$ In summary, if $X\in \mathscr C^{\otimes 7}_{12},$ then $\mathsf{Param}(X)\in \{\mathsf{Param}_{(j)}:\, 1\leq j\leq 4\}.$  As with previous results (see Theorem \ref{dlc1}, Corollary \ref{hqP} and the papers \cite{MKR, D.P6, D.P10-2, D.P9}), by direct calculations, one obtains the descriptions of the indecomposables $ (Q^{\otimes m}_{12})^{0}$ and $ (Q^{\otimes m}_{12})^{>0}$ below.
\begin{corl}\label{hqP2}
For any $m\geq 9,$ we have the following isomorphisms:
$$ (Q^{\otimes m}_{12})^{0}\cong \left\{\begin{array}{ll}
\bigoplus_{1\leq j\leq 5}(Q^{\otimes m}_{12})^{\mathsf{Param}_{(j)}^{0}}&\mbox{if $9\leq m\leq 10$},\\[1mm]
\bigoplus_{1\leq j\leq 6}(Q^{\otimes m}_{12})^{\mathsf{Param}_{(j)}^{0}}&\mbox{if $11\leq m\leq 12$},\\[1mm]
\bigoplus_{1\leq j\leq 7}(Q^{\otimes m}_{12})^{\mathsf{Param}_{(j)}^{0}}&\mbox{if $m\geq 13$},
\end{array}\right.$$
and
$$ (Q^{\otimes m}_{12})^{>0}\cong \left\{\begin{array}{ll}
\bigoplus_{4\leq j\leq 5}(Q^{\otimes m}_{12})^{\mathsf{Param}_{(j)}^{>0}}&\mbox{if $m = 9$},\\[1mm]
\bigoplus_{5\leq j\leq 6}(Q^{\otimes m}_{12})^{\mathsf{Param}_{(j)}}&\mbox{if $m = 10$},\\[1mm]
(Q^{\otimes m}_{12})^{\mathsf{Param}_{(6)}^{>0}}&\mbox{if $m = 11$},\\[1mm]
(Q^{\otimes m}_{12})^{\mathsf{Param}_{(7)}^{>0}}&\mbox{if $m = 12$},\\[1mm]
0&\mbox{if $m \geq 13$},
\end{array}\right.$$
Consequently, the dimensions of the cohit modules $(Q^{\otimes m}_{12})^{\mathsf{Param}_{(j)}^{0}}$ and $(Q^{\otimes m}_{12})^{\mathsf{Param}_{(j)}^{>0}}$ are determined as follows:

\medskip

\nopagebreak
\begingroup
\setlength{\LTpre}{0pt}
\setlength{\LTpost}{0pt}

\begin{longtable}{|c|c|p{6cm}|c|}
\hline
$m$ & $j$ & $\dim (Q^{\otimes m}_{12})^{\mathsf{Param}_{(j)}^{0}}$ & $\dim (Q^{\otimes m}_{12})^{\mathsf{Param}_{(j)}^{>0}}$ \\
\hline
\endfirsthead

\multicolumn{4}{c}%
{} \\
\hline
$m$ & $j$ & $\dim (Q^{\otimes m}_{12})^{\mathsf{Param}_{(j)}^{0}}$ & $\dim (Q^{\otimes m}_{12})^{\mathsf{Param}_{(j)}^{>0}}$ \\
\hline
\endhead

\hline
\multicolumn{4}{r}{} \\
\endfoot

\hline
\endlastfoot

9 & 1 & $20\binom{9}{4} + 75\binom{9}{5} + 90\binom{9}{6} + 35\binom{9}{7}$ & $0$ \\
\hline
9 & 2 & $\binom{9}{4} + 10\binom{9}{5} + 45\binom{9}{6} + 70\binom{9}{7} + 34\binom{9}{8}$ & $0$ \\
\hline
9 & 3 & $21\binom{9}{6} + 42\binom{9}{7} + 21\binom{9}{8}$ & $0$ \\
\hline
9 & 4 & $20\binom{9}{6} + 84\binom{9}{7} + 120\binom{9}{8}$ & $56$ \\
\hline
9 & 5 & $28\binom{9}{8}$ & $63$ \\
\hline
10 & 1 & $20\binom{10}{4} + 75\binom{10}{5} + 90\binom{10}{6} + 35\binom{10}{7}$ & $0$ \\
\hline
10 & 2 & $\binom{10}{4} + 10\binom{10}{5} + 45\binom{10}{6} + 70\binom{10}{7} + 34\binom{10}{8}$ & $0$ \\
\hline
10 & 3 & $21\binom{10}{6} + 42\binom{10}{7} + 21\binom{10}{8}$ & $0$ \\
\hline
10 & 4 & $20\binom{10}{6} + 84\binom{10}{7} + 120\binom{10}{8} + 56\binom{10}{9}$ & $0$ \\
\hline
10 & 5 & $28\binom{10}{8} + 63\binom{10}{9}$ & $35$ \\
\hline
10 & 6 & $0$ & $10$ \\
\hline
11 & 1 & $20\binom{11}{4} + 75\binom{11}{5} + 90\binom{11}{6} + 35\binom{11}{7}$ & $0$ \\
\hline
11 & 2 & $\binom{11}{4} + 10\binom{11}{5} + 45\binom{11}{6} + 70\binom{11}{7} + 34\binom{11}{8}$ & $0$ \\
\hline
11 & 3 & $21\binom{11}{6} + 42\binom{11}{7} + 21\binom{11}{8}$ & $0$ \\
\hline
11 & 4 & $20\binom{11}{6} + 84\binom{11}{7} + 120\binom{11}{8} + 56\binom{11}{9}$ & $0$ \\
\hline
11 & 5 & $28\binom{11}{8} + 63\binom{11}{9} + 35\binom{11}{10}$ & $0$ \\
\hline
11 & 6 & $10\binom{11}{10}$ & $10$ \\
\hline
12 & 1 & $20\binom{12}{4} + 75\binom{12}{5} + 90\binom{12}{6} + 35\binom{12}{7}$ & $0$ \\
\hline
12 & 2 & $\binom{12}{4} + 10\binom{12}{5} + 45\binom{12}{6} + 70\binom{12}{7} + 34\binom{12}{8}$ & $0$ \\
\hline
12 & 3 & $21\binom{12}{6} + 42\binom{12}{7} + 21\binom{12}{8}$ & $0$ \\
\hline
12 & 4 & $20\binom{12}{6} + 84\binom{12}{7} + 120\binom{12}{8} + 56\binom{12}{9}$ & $0$ \\
\hline
12 & 5 & $28\binom{12}{8} + 63\binom{12}{9} + 35\binom{12}{10}$ & $0$ \\
\hline
12 & 6 & $10\binom{12}{10} + 10\binom{12}{11}$ & $0$ \\
\hline
12 & 7 & $0$ & $1$ \\
\hline
$m \geq 13$ & 1 & $20\binom{m}{4} + 75\binom{m}{5} + 90\binom{m}{6} + 35\binom{m}{7}$ & $0$ \\
\hline
$m \geq 13$ & 2 & $\binom{m}{4} + 10\binom{m}{5} + 45\binom{m}{6} + 70\binom{m}{7} + 34\binom{m}{8}$ & $0$ \\
\hline
$m \geq 13$ & 3 & $21\binom{m}{6} + 42\binom{m}{7} + 21\binom{m}{8}$ & $0$ \\
\hline
$m \geq 13$ & 4 & $20\binom{m}{6} + 84\binom{m}{7} + 120\binom{m}{8} + 56\binom{m}{9}$ & $0$ \\
\hline
$m \geq 13$ & 5 & $28\binom{m}{8} + 63\binom{m}{9} + 35\binom{m}{10}$ & $0$ \\
\hline
$m \geq 13$ & 6 & $10\binom{m}{10} + 10\binom{m}{11}$ & $0$ \\
\hline
$m \geq 13$ & 7 & $\binom{m}{12}$ & $0$ \\
\hline
\multicolumn{2}{|c|}{otherwise} & $0$ & $0$ \\
\hline
\end{longtable}

\endgroup
\end{corl}

Incorporating the results from our previous works \cite{D.P6, D.P9} and Corollary \ref{hqP2}, we conclude that Conjecture \ref{gtK} holds true for all $m\geq 1$ and parameter vectors of degrees not exceeding $12.$ This completes the proof of Theorem \ref{dlc3}.

\medskip\noindent\textbf{Final remark.} The previous findings in \cite{MKR} and \cite{D.P9} allow us to easily verify the following assertion.

\begin{corl}
Let us consider the parameter vectors:
$$ \begin{array}{ll}
\medskip
\mathsf{Param}_{(1,\, m)}:=(m,0),\ \ \mathsf{Param}_{(2,\, m)}:=(m-1, 1),\ \ \mathsf{Param}_{(3,\, m)}:=(m-2, 2),\\
\mathsf{Param}_{(4,\, m)}:=(m,1),\ \ \mathsf{Param}_{(5,\, m)}:=(m-1, 2).
\end{array}$$
Then the dimensions of the cohit modules $ (Q^{\otimes m}_{m+j})^{\mathsf{Param}_{(k,\, m)}},\, 0\leq j\leq 3,\, 1\leq k\leq 5,$ are determined as follows:
\begin{itemize}
\item[$\bullet$] $\dim (Q^{\otimes m}_{m})^{\mathsf{Param}_{(1,\, m)}} = \dim  (Q^{\otimes m}_{m})^{\mathsf{Param}_{(1,\, m)}^{>0}} = 1,$ for any $m > 0;$

\item[$\bullet$] $\dim  (Q^{\otimes m}_{m+1})^{\mathsf{Param}_{(2,\, m)}} = \dim  (Q^{\otimes m}_{m+1})^{\mathsf{Param}_{(2,\, m)}^{0}} + \dim  (Q^{\otimes m}_{m+1})^{\mathsf{Param}_{(2,\, m)}^{>0}}  =   m^{2}-1,$ for arbitrary $m  >1,$ where $\dim  (Q^{\otimes m}_{m+1})^{\mathsf{Param}_{(2,\, m)}^{0}} = m(m-1), \ \mbox{and} \ \dim  (Q^{\otimes m}_{m+1})^{\mathsf{Param}_{(2,\, m)}^{>0}} = m-1;$

\item[$\bullet$] $\begin{array}{ll}
 \dim  (Q^{\otimes m}_{m+2})^{\mathsf{Param}_{(3,\, m)}} &=  \dim  (Q^{\otimes m}_{m+2})^{\mathsf{Param}_{(3,\, m)}^{0}}  + \dim  (Q^{\otimes m}_{m+2})^{\mathsf{Param}_{(3,\, m)}^{>0}}\\
& =  \bigg(\binom{m}{2} + m\bigg)\bigg(\binom{m-1}{2}-1\bigg)\ \mbox{for all $m > 3,$ where} \end{array}$
$$\dim  (Q^{\otimes m}_{m+2})^{\mathsf{Param}_{(3,\, m)}^{0}} = \bigg[\binom{m}{2} + m-1\bigg]\bigg[\binom{m-1}{2}-1\bigg], \ \mbox{and} \  \dim  (Q^{\otimes m}_{m+2})^{\mathsf{Param}_{(3,\, m)}^{>0}} = \binom{m-1}{2}-1;$$

\item[$\bullet$]  $\dim  (Q^{\otimes m}_{m+2})^{\mathsf{Param}_{(4,\, m)}} = \dim  (Q^{\otimes m}_{m+2})^{\mathsf{Param}_{(4,\, m)}^{>0}} = m,$ for every $m >3;$

\item[$\bullet$]  $\dim  (Q^{\otimes m}_{m+3})^{\mathsf{Param}_{(5,\, m)}} = \dim  (Q^{\otimes m}_{m+3})^{\mathsf{Param}_{(5,\, m)}^{0}} + \dim  (Q^{\otimes m}_{m+3})^{\mathsf{Param}_{(5,\, m)}^{>0}} =  3\binom{m}{m-3} + m(m-2)$ for all $m > 3,$ where 
$\dim  (Q^{\otimes m}_{m+3})^{\mathsf{Param}_{(5,\, m)}^{0}} = 3\binom{m}{m-3}, \ \mbox{and} \ \dim  (Q^{\otimes m}_{m+3})^{\mathsf{Param}_{(5,\, m)}^{>0}} = m(m-2).$
\end{itemize}
\end{corl}

It should be noted that $\mathsf{Param}_{(1,\, 12)} = \mathsf{Param}_{(7)},$\ $\mathsf{Param}_{(2,\, 11)} = \mathsf{Param}_{(6)},$\ $\mathsf{Param}_{(3,\, 10)} = \mathsf{Param}_{(5)},$\ $\mathsf{Param}_{(4,\, 10)} = \mathsf{Param}_{(6)},$\ and $\mathsf{Param}_{(5,\, 9)} = \mathsf{Param}_{(5)},$ where the parameter vectors $\mathsf{Param}_{(j)},\, 5\leq j\leq 7,$ are as in Corollary \ref{hqP2}. 

As a result, Conjecture \ref{gtK} is confirmed for parameter vectors $\mathsf{Param}_{(j,\, m)}$ with $1\leq j\leq 5.$

\section{Conclusion and further directions}\label{s5}

We have investigated the Peterson hit problem and its applications to the Singer algebraic transfer for the polynomial algebra $\mathcal P_5=\mathbb F_2[x_1,\ldots,x_5]$ in the generic family of degrees
\[
n_s=5(2^s-1)+18\cdot 2^s,\qquad s\ge 0.
\]
Our first main result determines the kernel of the Kameko morphism $(\widetilde{Sq^0_*})_{(5,n_1)}$ and yields the uniform dimension formula
$\dim Q^{\otimes 5}_{n_s}=2630$ for all $s\ge 1$.
We then identify the $G(5)$-invariant line in the relevant cohit spaces and obtain an explicit generator, which implies that the fifth algebraic transfer is an isomorphism in bidegree $(5,5+n_s)$ for every $s\ge 0$.
As a complementary outcome, we verify a localized variation of Kameko's conjecture for all $m\ge 1$ in degrees $\le 12$.

\smallskip
\noindent\textbf{Further directions.}
The present work suggests several natural extensions.
\begin{itemize}
\item[(1)] \emph{Other generic families with $d=5$.}
The degrees $n_s=d(2^s-1)+k\cdot 2^s$ with fixed $d=5$ but different values of $k$ constitute a broad class where iterated Kameko morphisms can be effective. It would be interesting to determine which choices of $k$ lead to stable patterns for $\dim Q^{\otimes 5}_{n_s}$ and to explicit descriptions of $(Q^{\otimes 5}_{n_s})^{G(5)}$.

\item[(2)] \emph{Higher ranks $m\ge 6$.}
Our results emphasize the interaction between admissibility, the Kameko morphism, and $G(m)$-invariant theory. Extending comparable results to $m\ge 6$ is highly desirable, especially in view of recent evidence that Singer's conjecture fails at $m=6$ in certain degrees. 
\end{itemize}

\smallskip
\noindent\textbf{Limitations.}
Even for fixed $m$, explicit bases and admissibility analyses grow rapidly in complexity as the degree increases, and computations become expensive both combinatorially and algorithmically. The results here rely on a mixture of structural arguments and explicit calculations.

\medskip

\section{Appendix}\label{s6}

In this appendix, we collect explicit data that support the statements and computations in the main text. 
In particular, we record explicit polynomial representatives for the $G(5)$-invariant classes appearing in Theorem~\ref{dlc2}, and we also include (or point to) computational outputs and lists of monomials used in the verification of the results in Sect.~\ref{s4}. 
These data are provided for the reader's convenience and for reproducibility of the invariant checks and dimension computations.

\subsection{Explicit forms of the invariant representatives $\widetilde{\xi}_{n_0}$ and $\widetilde{\xi}_{n_1}$}\label{ap5.0}

In Theorem~\ref{dlc2}, we identify one-dimensional $G(5)$-invariant subspaces in
$(Q^{\otimes 5}_{n_0})^{G(5)}$ and $(Q^{\otimes 5}_{n_1})^{G(5)}$.
For the reader's convenience and for reproducibility of the invariant checks, we record below explicit polynomial representatives
$\widetilde{\xi}_{n_0}\in (\mathcal P_5)_{n_0}$ and $\widetilde{\xi}_{n_1}\in (\mathcal P_5)_{n_1}$ whose classes generate these invariants.

\begin{align*}
\widetilde{\xi}_{n_0}&= x_{1}^{3} x_{2}^{5} x_{3}^{8} x_{4}^{2} + x_{1}^{3} x_{2} x_{3}^{12} x_{4}^{2} + x_{1} x_{2}^{3} x_{3}^{12} x_{4}^{2} + x_{1} x_{2} x_{3}^{14} x_{4}^{2} \\
    & + x_{1}^{3} x_{2}^{5} x_{3}^{2} x_{4}^{8} + x_{1} x_{2} x_{3}^{6} x_{4}^{10} + x_{1}^{3} x_{2} x_{3}^{2} x_{4}^{12} + x_{1} x_{2}^{3} x_{3}^{2} x_{4}^{12} \\
    & + x_{1} x_{2} x_{3}^{2} x_{4}^{14} + x_{1}^{3} x_{2}^{5} x_{3}^{8} x_{4} x_{5} + x_{1}^{3} x_{2} x_{3}^{12} x_{4} x_{5} + x_{1} x_{2}^{3} x_{3}^{12} x_{4} x_{5} \\
    & + x_{1} x_{2} x_{3}^{14} x_{4} x_{5} + x_{1}^{3} x_{2}^{5} x_{3} x_{4}^{8} x_{5} + x_{1}^{3} x_{2} x_{3} x_{4}^{12} x_{5} + x_{1} x_{2}^{3} x_{3} x_{4}^{12} x_{5} \\
    & + x_{1} x_{2} x_{3} x_{4}^{14} x_{5} + x_{1}^{3} x_{2}^{5} x_{3}^{8} x_{5}^{2} + x_{1}^{3} x_{2} x_{3}^{12} x_{5}^{2} + x_{1} x_{2}^{3} x_{3}^{12} x_{5}^{2} \\
    & + x_{1} x_{2} x_{3}^{14} x_{5}^{2} + x_{1}^{3} x_{2}^{5} x_{4}^{8} x_{5}^{2} + x_{1}^{3} x_{3}^{5} x_{4}^{8} x_{5}^{2} + x_{1} x_{2}^{2} x_{3}^{5} x_{4}^{8} x_{5}^{2} \\
    & + x_{2}^{3} x_{3}^{5} x_{4}^{8} x_{5}^{2} + x_{1} x_{2}^{2} x_{3}^{4} x_{4}^{9} x_{5}^{2} + x_{1}^{3} x_{2} x_{4}^{12} x_{5}^{2} + x_{1} x_{2}^{3} x_{4}^{12} x_{5}^{2} \\
    & + x_{1}^{3} x_{3} x_{4}^{12} x_{5}^{2} + x_{2}^{3} x_{3} x_{4}^{12} x_{5}^{2} + x_{1} x_{3}^{3} x_{4}^{12} x_{5}^{2} + x_{2} x_{3}^{3} x_{4}^{12} x_{5}^{2} \\
    & + x_{1} x_{2} x_{4}^{14} x_{5}^{2} + x_{1} x_{3} x_{4}^{14} x_{5}^{2} + x_{2} x_{3} x_{4}^{14} x_{5}^{2} + x_{1}^{3} x_{2}^{5} x_{3}^{2} x_{4}^{4} x_{5}^{4} \\
    & + x_{1}^{3} x_{2}^{5} x_{3}^{2} x_{5}^{8} + x_{1}^{3} x_{2}^{5} x_{3} x_{4} x_{5}^{8} + x_{1}^{3} x_{2}^{5} x_{4}^{2} x_{5}^{8} + x_{1}^{3} x_{3}^{5} x_{4}^{2} x_{5}^{8} \\
    & + x_{1} x_{2}^{2} x_{3}^{5} x_{4}^{2} x_{5}^{8} + x_{2}^{3} x_{3}^{5} x_{4}^{2} x_{5}^{8} + x_{1} x_{2} x_{3}^{6} x_{4}^{2} x_{5}^{8} + x_{1}^{3} x_{2} x_{3}^{2} x_{4}^{4} x_{5}^{8} \\
    & + x_{1} x_{2}^{3} x_{3}^{2} x_{4}^{4} x_{5}^{8} + x_{1} x_{2}^{2} x_{3} x_{4}^{6} x_{5}^{8} + x_{1} x_{2} x_{3}^{6} x_{5}^{10} + x_{1} x_{2}^{2} x_{3}^{4} x_{4} x_{5}^{10} \\
    & + x_{1} x_{2}^{2} x_{3} x_{4}^{4} x_{5}^{10} + x_{1} x_{2} x_{3}^{2} x_{4}^{4} x_{5}^{10} + x_{1} x_{2} x_{4}^{6} x_{5}^{10} + x_{1} x_{3} x_{4}^{6} x_{5}^{10} \\
    & + x_{2} x_{3} x_{4}^{6} x_{5}^{10} + x_{1}^{3} x_{2} x_{3}^{2} x_{5}^{12} + x_{1} x_{2}^{3} x_{3}^{2} x_{5}^{12} + x_{1}^{3} x_{2} x_{3} x_{4} x_{5}^{12} \\
    & + x_{1} x_{2}^{3} x_{3} x_{4} x_{5}^{12} + x_{1}^{3} x_{2} x_{4}^{2} x_{5}^{12} + x_{1} x_{2}^{3} x_{4}^{2} x_{5}^{12} + x_{1}^{3} x_{3} x_{4}^{2} x_{5}^{12} \\
    & + x_{2}^{3} x_{3} x_{4}^{2} x_{5}^{12} + x_{1} x_{3}^{3} x_{4}^{2} x_{5}^{12} + x_{2} x_{3}^{3} x_{4}^{2} x_{5}^{12} + x_{1} x_{2} x_{3}^{2} x_{5}^{14} \\
    & + x_{1} x_{2} x_{3} x_{4} x_{5}^{14} + x_{1} x_{2} x_{4}^{2} x_{5}^{14} + x_{1} x_{3} x_{4}^{2} x_{5}^{14} + x_{2} x_{3} x_{4}^{2} x_{5}^{14},\\
\widetilde{\xi}_{n_1}&= x_{1}^{15} x_{2}^{19} x_{3}^{7} + x_{1}^{7} x_{2}^{27} x_{3}^{7} + x_{1}^{15} x_{2}^{7} x_{3}^{19} + x_{1}^{7} x_{2}^{7} x_{3}^{27} \\
    & + x_{1}^{15} x_{2}^{19} x_{3} x_{4}^{6} + x_{1}^{7} x_{2}^{27} x_{3} x_{4}^{6} + x_{1}^{15} x_{2}^{3} x_{3}^{17} x_{4}^{6} + x_{1}^{7} x_{2}^{11} x_{3}^{17} x_{4}^{6} \\
    & + x_{1}^{3} x_{2}^{15} x_{3}^{17} x_{4}^{6} + x_{1}^{15} x_{2} x_{3}^{19} x_{4}^{6} + x_{1}^{3} x_{2}^{13} x_{3}^{19} x_{4}^{6} + x_{1} x_{2}^{15} x_{3}^{19} x_{4}^{6} \\
    & + x_{1}^{7} x_{2} x_{3}^{27} x_{4}^{6} + x_{1}^{3} x_{2}^{5} x_{3}^{27} x_{4}^{6} + x_{1} x_{2}^{7} x_{3}^{27} x_{4}^{6} + x_{1}^{3} x_{2}^{3} x_{3}^{29} x_{4}^{6} \\
    & + x_{1}^{15} x_{2}^{19} x_{4}^{7} + x_{1}^{7} x_{2}^{27} x_{4}^{7} + x_{1}^{15} x_{2} x_{3}^{18} x_{4}^{7} + x_{1}^{3} x_{2}^{13} x_{3}^{18} x_{4}^{7} \\
    & + x_{1} x_{2}^{15} x_{3}^{18} x_{4}^{7} + x_{1}^{15} x_{3}^{19} x_{4}^{7} + x_{1} x_{2}^{14} x_{3}^{19} x_{4}^{7} + x_{2}^{15} x_{3}^{19} x_{4}^{7} \\
    & + x_{1}^{7} x_{2} x_{3}^{26} x_{4}^{7} + x_{1}^{3} x_{2}^{5} x_{3}^{26} x_{4}^{7} + x_{1} x_{2}^{7} x_{3}^{26} x_{4}^{7} + x_{1}^{7} x_{3}^{27} x_{4}^{7} \\
    & + x_{1} x_{2}^{6} x_{3}^{27} x_{4}^{7} + x_{2}^{7} x_{3}^{27} x_{4}^{7} + x_{1}^{7} x_{2}^{7} x_{3}^{11} x_{4}^{16} + x_{1}^{15} x_{2}^{7} x_{3} x_{4}^{18} \\
    & + x_{1}^{15} x_{2}^{3} x_{3}^{5} x_{4}^{18} + x_{1}^{7} x_{2}^{11} x_{3}^{5} x_{4}^{18} + x_{1}^{3} x_{2}^{15} x_{3}^{5} x_{4}^{18} + x_{1}^{15} x_{2} x_{3}^{7} x_{4}^{18} \\
    & + x_{1}^{3} x_{2}^{13} x_{3}^{7} x_{4}^{18} + x_{1} x_{2}^{15} x_{3}^{7} x_{4}^{18} + x_{1}^{15} x_{2}^{7} x_{4}^{19} + x_{1}^{15} x_{2} x_{3}^{6} x_{4}^{19} \\
    & + x_{1}^{3} x_{2}^{13} x_{3}^{6} x_{4}^{19} + x_{1} x_{2}^{15} x_{3}^{6} x_{4}^{19} + x_{1}^{15} x_{3}^{7} x_{4}^{19} + x_{1} x_{2}^{14} x_{3}^{7} x_{4}^{19} \\
    & + x_{2}^{15} x_{3}^{7} x_{4}^{19} + x_{1}^{7} x_{2}^{7} x_{3}^{8} x_{4}^{19} + x_{1}^{7} x_{2} x_{3}^{14} x_{4}^{19} + x_{1}^{3} x_{2}^{5} x_{3}^{14} x_{4}^{19} \\
    & + x_{1} x_{2}^{7} x_{3}^{14} x_{4}^{19} + x_{1}^{7} x_{2} x_{3}^{11} x_{4}^{22} + x_{1}^{3} x_{2}^{5} x_{3}^{11} x_{4}^{22} + x_{1} x_{2}^{7} x_{3}^{11} x_{4}^{22} \\
    & + x_{1}^{7} x_{2}^{7} x_{3} x_{4}^{26} + x_{1}^{7} x_{2} x_{3}^{7} x_{4}^{26} + x_{1}^{3} x_{2}^{5} x_{3}^{7} x_{4}^{26} + x_{1} x_{2}^{7} x_{3}^{7} x_{4}^{26} \\
    & + x_{1}^{7} x_{2}^{7} x_{4}^{27} + x_{1}^{7} x_{2} x_{3}^{6} x_{4}^{27} + x_{1}^{3} x_{2}^{5} x_{3}^{6} x_{4}^{27} + x_{1} x_{2}^{7} x_{3}^{6} x_{4}^{27} \\
    & + x_{1}^{7} x_{3}^{7} x_{4}^{27} + x_{1} x_{2}^{6} x_{3}^{7} x_{4}^{27} + x_{2}^{7} x_{3}^{7} x_{4}^{27} + x_{1}^{3} x_{2}^{3} x_{3}^{5} x_{4}^{30} \\
    & + x_{1}^{7} x_{2}^{11} x_{3}^{17} x_{4}^{4} x_{5}^{2} + x_{1}^{7} x_{2}^{3} x_{3}^{25} x_{4}^{4} x_{5}^{2} + x_{1}^{3} x_{2}^{7} x_{3}^{25} x_{4}^{4} x_{5}^{2} + x_{1}^{3} x_{2}^{3} x_{3}^{29} x_{4}^{4} x_{5}^{2} \\
    & + x_{1}^{7} x_{2}^{11} x_{3}^{16} x_{4}^{5} x_{5}^{2} + x_{1}^{7} x_{2}^{3} x_{3}^{24} x_{4}^{5} x_{5}^{2} + x_{1}^{3} x_{2}^{7} x_{3}^{24} x_{4}^{5} x_{5}^{2} + x_{1}^{3} x_{2}^{3} x_{3}^{28} x_{4}^{5} x_{5}^{2} \\
    & + x_{1}^{15} x_{2}^{3} x_{3}^{5} x_{4}^{16} x_{5}^{2} + x_{1}^{7} x_{2}^{11} x_{3}^{5} x_{4}^{16} x_{5}^{2} + x_{1}^{3} x_{2}^{15} x_{3}^{5} x_{4}^{16} x_{5}^{2} + x_{1}^{3} x_{2}^{5} x_{3}^{15} x_{4}^{16} x_{5}^{2} \\
    & + x_{1}^{15} x_{2}^{3} x_{3} x_{4}^{20} x_{5}^{2} + x_{1}^{7} x_{2}^{11} x_{3} x_{4}^{20} x_{5}^{2} + x_{1}^{3} x_{2}^{15} x_{3} x_{4}^{20} x_{5}^{2} + x_{1}^{15} x_{2} x_{3}^{3} x_{4}^{20} x_{5}^{2} \\
    & + x_{1} x_{2}^{15} x_{3}^{3} x_{4}^{20} x_{5}^{2} + x_{1}^{3} x_{2}^{3} x_{3}^{13} x_{4}^{20} x_{5}^{2} + x_{1}^{3} x_{2} x_{3}^{15} x_{4}^{20} x_{5}^{2} + x_{1} x_{2}^{3} x_{3}^{15} x_{4}^{20} x_{5}^{2} \\
    & + x_{1}^{3} x_{2}^{13} x_{3}^{2} x_{4}^{21} x_{5}^{2} + x_{1}^{3} x_{2}^{3} x_{3}^{12} x_{4}^{21} x_{5}^{2} + x_{1}^{15} x_{2} x_{3} x_{4}^{22} x_{5}^{2} + x_{1}^{3} x_{2}^{13} x_{3} x_{4}^{22} x_{5}^{2} \\
    & + x_{1} x_{2}^{15} x_{3} x_{4}^{22} x_{5}^{2} + x_{1}^{7} x_{2} x_{3}^{9} x_{4}^{22} x_{5}^{2} + x_{1}^{3} x_{2} x_{3}^{13} x_{4}^{22} x_{5}^{2} + x_{1} x_{2} x_{3}^{15} x_{4}^{22} x_{5}^{2} \\
    & + x_{1}^{3} x_{2}^{5} x_{3}^{8} x_{4}^{23} x_{5}^{2} + x_{1}^{3} x_{2} x_{3}^{12} x_{4}^{23} x_{5}^{2} + x_{1} x_{2}^{3} x_{3}^{12} x_{4}^{23} x_{5}^{2} + x_{1} x_{2} x_{3}^{14} x_{4}^{23} x_{5}^{2} \\
    & + x_{1}^{7} x_{2} x_{3}^{7} x_{4}^{24} x_{5}^{2} + x_{1}^{3} x_{2}^{5} x_{3}^{7} x_{4}^{24} x_{5}^{2} + x_{1}^{3} x_{2}^{7} x_{3}^{4} x_{4}^{25} x_{5}^{2} + x_{1}^{3} x_{2}^{5} x_{3}^{6} x_{4}^{25} x_{5}^{2} \\
    & + x_{1}^{3} x_{2}^{4} x_{3}^{7} x_{4}^{25} x_{5}^{2} + x_{1} x_{2}^{6} x_{3}^{7} x_{4}^{25} x_{5}^{2} + x_{1}^{3} x_{2}^{7} x_{3} x_{4}^{28} x_{5}^{2} + x_{1}^{7} x_{2} x_{3}^{3} x_{4}^{28} x_{5}^{2} \\
    & + x_{1} x_{2}^{7} x_{3}^{3} x_{4}^{28} x_{5}^{2} + x_{1}^{3} x_{2} x_{3}^{7} x_{4}^{28} x_{5}^{2} + x_{1} x_{2}^{3} x_{3}^{7} x_{4}^{28} x_{5}^{2} + x_{1}^{3} x_{2}^{5} x_{3}^{2} x_{4}^{29} x_{5}^{2} \\
    & + x_{1}^{3} x_{2}^{4} x_{3}^{3} x_{4}^{29} x_{5}^{2} + x_{1} x_{2}^{6} x_{3}^{3} x_{4}^{29} x_{5}^{2} + x_{1}^{3} x_{2}^{3} x_{3}^{4} x_{4}^{29} x_{5}^{2} + x_{1}^{3} x_{2} x_{3}^{6} x_{4}^{29} x_{5}^{2} \\
    & + x_{1}^{7} x_{2} x_{3} x_{4}^{30} x_{5}^{2} + x_{1}^{3} x_{2}^{5} x_{3} x_{4}^{30} x_{5}^{2} + x_{1} x_{2}^{7} x_{3} x_{4}^{30} x_{5}^{2} + x_{1}^{3} x_{2} x_{3}^{5} x_{4}^{30} x_{5}^{2} \\
    & + x_{1} x_{2}^{3} x_{3}^{5} x_{4}^{30} x_{5}^{2} + x_{1} x_{2}^{3} x_{3}^{28} x_{4}^{6} x_{5}^{3} + x_{1} x_{2}^{2} x_{3}^{29} x_{4}^{6} x_{5}^{3} + x_{1}^{7} x_{2}^{7} x_{3}^{8} x_{4}^{16} x_{5}^{3} \\
    & + x_{1}^{3} x_{2} x_{3}^{14} x_{4}^{20} x_{5}^{3} + x_{1}^{3} x_{2} x_{3}^{12} x_{4}^{22} x_{5}^{3} + x_{1}^{7} x_{2} x_{3}^{6} x_{4}^{24} x_{5}^{3} + x_{1}^{3} x_{2}^{5} x_{3}^{6} x_{4}^{24} x_{5}^{3} \\
    & + x_{1} x_{2}^{7} x_{3}^{2} x_{4}^{28} x_{5}^{3} + x_{1} x_{2}^{3} x_{3}^{6} x_{4}^{28} x_{5}^{3} + x_{1} x_{2}^{3} x_{3}^{4} x_{4}^{30} x_{5}^{3} + x_{1} x_{2}^{2} x_{3}^{5} x_{4}^{30} x_{5}^{3} \\
    & + x_{1} x_{2} x_{3}^{6} x_{4}^{30} x_{5}^{3} + x_{1}^{15} x_{2}^{19} x_{3} x_{4}^{2} x_{5}^{4} + x_{1}^{7} x_{2}^{27} x_{3} x_{4}^{2} x_{5}^{4} + x_{1}^{15} x_{2}^{3} x_{3}^{17} x_{4}^{2} x_{5}^{4} \\
    & + x_{1}^{7} x_{2}^{11} x_{3}^{17} x_{4}^{2} x_{5}^{4} + x_{1}^{3} x_{2}^{15} x_{3}^{17} x_{4}^{2} x_{5}^{4} + x_{1}^{15} x_{2} x_{3}^{19} x_{4}^{2} x_{5}^{4} + x_{1}^{3} x_{2}^{13} x_{3}^{19} x_{4}^{2} x_{5}^{4} \\
    & + x_{1} x_{2}^{15} x_{3}^{19} x_{4}^{2} x_{5}^{4} + x_{1}^{7} x_{2} x_{3}^{27} x_{4}^{2} x_{5}^{4} + x_{1}^{3} x_{2}^{5} x_{3}^{27} x_{4}^{2} x_{5}^{4} + x_{1} x_{2}^{7} x_{3}^{27} x_{4}^{2} x_{5}^{4} \\
    & + x_{1}^{3} x_{2}^{3} x_{3}^{29} x_{4}^{2} x_{5}^{4} + x_{1}^{15} x_{2}^{3} x_{3} x_{4}^{18} x_{5}^{4} + x_{1}^{3} x_{2}^{15} x_{3} x_{4}^{18} x_{5}^{4} + x_{1}^{15} x_{2} x_{3}^{3} x_{4}^{18} x_{5}^{4} \\
    & + x_{1}^{3} x_{2}^{13} x_{3}^{3} x_{4}^{18} x_{5}^{4} + x_{1} x_{2}^{15} x_{3}^{3} x_{4}^{18} x_{5}^{4} + x_{1}^{7} x_{2} x_{3}^{11} x_{4}^{18} x_{5}^{4} + x_{1}^{3} x_{2}^{5} x_{3}^{11} x_{4}^{18} x_{5}^{4} \\
    & + x_{1} x_{2}^{7} x_{3}^{11} x_{4}^{18} x_{5}^{4} + x_{1}^{3} x_{2}^{3} x_{3}^{13} x_{4}^{18} x_{5}^{4} + x_{1}^{3} x_{2} x_{3}^{15} x_{4}^{18} x_{5}^{4} + x_{1} x_{2}^{3} x_{3}^{15} x_{4}^{18} x_{5}^{4} \\
    & + x_{1}^{15} x_{2} x_{3}^{2} x_{4}^{19} x_{5}^{4} + x_{1}^{3} x_{2}^{13} x_{3}^{2} x_{4}^{19} x_{5}^{4} + x_{1} x_{2}^{15} x_{3}^{2} x_{4}^{19} x_{5}^{4} + x_{1}^{3} x_{2} x_{3}^{14} x_{4}^{19} x_{5}^{4} \\
    & + x_{1} x_{2}^{3} x_{3}^{14} x_{4}^{19} x_{5}^{4} + x_{1} x_{2}^{2} x_{3}^{15} x_{4}^{19} x_{5}^{4} + x_{1}^{7} x_{2}^{3} x_{3} x_{4}^{26} x_{5}^{4} + x_{1}^{3} x_{2}^{7} x_{3} x_{4}^{26} x_{5}^{4} \\
    & + x_{1}^{7} x_{2} x_{3}^{3} x_{4}^{26} x_{5}^{4} + x_{1}^{3} x_{2}^{5} x_{3}^{3} x_{4}^{26} x_{5}^{4} + x_{1} x_{2}^{3} x_{3}^{7} x_{4}^{26} x_{5}^{4} + x_{1}^{7} x_{2} x_{3}^{2} x_{4}^{27} x_{5}^{4} \\
    & + x_{1}^{3} x_{2}^{5} x_{3}^{2} x_{4}^{27} x_{5}^{4} + x_{1} x_{2}^{7} x_{3}^{2} x_{4}^{27} x_{5}^{4} + x_{1}^{3} x_{2}^{3} x_{3}^{4} x_{4}^{27} x_{5}^{4} + x_{1}^{3} x_{2} x_{3}^{6} x_{4}^{27} x_{5}^{4} \\
    & + x_{1} x_{2}^{3} x_{3}^{6} x_{4}^{27} x_{5}^{4} + x_{1} x_{2}^{2} x_{3}^{7} x_{4}^{27} x_{5}^{4} + x_{1}^{3} x_{2}^{3} x_{3}^{3} x_{4}^{28} x_{5}^{4} + x_{1}^{3} x_{2} x_{3}^{3} x_{4}^{30} x_{5}^{4} \\
    & + x_{1}^{15} x_{2}^{19} x_{3} x_{5}^{6} + x_{1}^{7} x_{2}^{27} x_{3} x_{5}^{6} + x_{1}^{15} x_{2}^{3} x_{3}^{17} x_{5}^{6} + x_{1}^{7} x_{2}^{11} x_{3}^{17} x_{5}^{6} \\
    & + x_{1}^{3} x_{2}^{15} x_{3}^{17} x_{5}^{6} + x_{1}^{15} x_{2} x_{3}^{19} x_{5}^{6} + x_{1}^{3} x_{2}^{13} x_{3}^{19} x_{5}^{6} + x_{1} x_{2}^{15} x_{3}^{19} x_{5}^{6} \\
    & + x_{1}^{7} x_{2} x_{3}^{27} x_{5}^{6} + x_{1}^{3} x_{2}^{5} x_{3}^{27} x_{5}^{6} + x_{1} x_{2}^{7} x_{3}^{27} x_{5}^{6} + x_{1}^{3} x_{2}^{3} x_{3}^{29} x_{5}^{6} \\
    & + x_{1}^{15} x_{2}^{19} x_{4} x_{5}^{6} + x_{1}^{7} x_{2}^{27} x_{4} x_{5}^{6} + x_{1}^{7} x_{2}^{11} x_{3}^{16} x_{4} x_{5}^{6} + x_{1}^{15} x_{2} x_{3}^{18} x_{4} x_{5}^{6} \\
    & + x_{1}^{3} x_{2}^{13} x_{3}^{18} x_{4} x_{5}^{6} + x_{1} x_{2}^{15} x_{3}^{18} x_{4} x_{5}^{6} + x_{1}^{15} x_{3}^{19} x_{4} x_{5}^{6} + x_{1} x_{2}^{14} x_{3}^{19} x_{4} x_{5}^{6} \\
    & + x_{2}^{15} x_{3}^{19} x_{4} x_{5}^{6} + x_{1}^{7} x_{2}^{3} x_{3}^{24} x_{4} x_{5}^{6} + x_{1}^{3} x_{2}^{7} x_{3}^{24} x_{4} x_{5}^{6} + x_{1}^{7} x_{2} x_{3}^{26} x_{4} x_{5}^{6} \\
    & + x_{1}^{3} x_{2}^{5} x_{3}^{26} x_{4} x_{5}^{6} + x_{1} x_{2}^{7} x_{3}^{26} x_{4} x_{5}^{6} + x_{1}^{7} x_{3}^{27} x_{4} x_{5}^{6} + x_{1} x_{2}^{6} x_{3}^{27} x_{4} x_{5}^{6} \\
    & + x_{2}^{7} x_{3}^{27} x_{4} x_{5}^{6} + x_{1}^{3} x_{2}^{3} x_{3}^{28} x_{4} x_{5}^{6} + x_{1} x_{2}^{3} x_{3}^{28} x_{4}^{3} x_{5}^{6} + x_{1} x_{2}^{2} x_{3}^{29} x_{4}^{3} x_{5}^{6} \\
    & + x_{1}^{15} x_{2}^{3} x_{3} x_{4}^{16} x_{5}^{6} + x_{1}^{7} x_{2}^{11} x_{3} x_{4}^{16} x_{5}^{6} + x_{1}^{3} x_{2}^{15} x_{3} x_{4}^{16} x_{5}^{6} + x_{1}^{15} x_{2} x_{3}^{3} x_{4}^{16} x_{5}^{6} \\
    & + x_{1}^{3} x_{2}^{13} x_{3}^{3} x_{4}^{16} x_{5}^{6} + x_{1} x_{2}^{15} x_{3}^{3} x_{4}^{16} x_{5}^{6} + x_{1}^{7} x_{2}^{3} x_{3}^{9} x_{4}^{16} x_{5}^{6} + x_{1}^{7} x_{2} x_{3}^{11} x_{4}^{16} x_{5}^{6} \\
    & + x_{1}^{3} x_{2}^{5} x_{3}^{11} x_{4}^{16} x_{5}^{6} + x_{1} x_{2}^{7} x_{3}^{11} x_{4}^{16} x_{5}^{6} + x_{1}^{3} x_{2} x_{3}^{15} x_{4}^{16} x_{5}^{6} + x_{1} x_{2}^{3} x_{3}^{15} x_{4}^{16} x_{5}^{6} \\
    & + x_{1}^{15} x_{2}^{3} x_{4}^{17} x_{5}^{6} + x_{1}^{7} x_{2}^{11} x_{4}^{17} x_{5}^{6} + x_{1}^{3} x_{2}^{15} x_{4}^{17} x_{5}^{6} + x_{1}^{15} x_{2} x_{3}^{2} x_{4}^{17} x_{5}^{6} \\
    & + x_{1} x_{2}^{15} x_{3}^{2} x_{4}^{17} x_{5}^{6} + x_{1}^{15} x_{3}^{3} x_{4}^{17} x_{5}^{6} + x_{1} x_{2}^{14} x_{3}^{3} x_{4}^{17} x_{5}^{6} + x_{2}^{15} x_{3}^{3} x_{4}^{17} x_{5}^{6} \\
    & + x_{1}^{7} x_{2}^{3} x_{3}^{8} x_{4}^{17} x_{5}^{6} + x_{1}^{7} x_{2} x_{3}^{10} x_{4}^{17} x_{5}^{6} + x_{1}^{3} x_{2}^{5} x_{3}^{10} x_{4}^{17} x_{5}^{6} + x_{1} x_{2}^{7} x_{3}^{10} x_{4}^{17} x_{5}^{6} \\
    & + x_{1}^{7} x_{3}^{11} x_{4}^{17} x_{5}^{6} + x_{1}^{3} x_{2}^{4} x_{3}^{11} x_{4}^{17} x_{5}^{6} + x_{2}^{7} x_{3}^{11} x_{4}^{17} x_{5}^{6} + x_{1}^{3} x_{2}^{3} x_{3}^{12} x_{4}^{17} x_{5}^{6} \\
    & + x_{1}^{3} x_{2} x_{3}^{14} x_{4}^{17} x_{5}^{6} + x_{1} x_{2}^{3} x_{3}^{14} x_{4}^{17} x_{5}^{6} + x_{1}^{3} x_{3}^{15} x_{4}^{17} x_{5}^{6} + x_{1} x_{2}^{2} x_{3}^{15} x_{4}^{17} x_{5}^{6} \\
    & + x_{2}^{3} x_{3}^{15} x_{4}^{17} x_{5}^{6} + x_{1} x_{2}^{3} x_{3}^{13} x_{4}^{18} x_{5}^{6} + x_{1}^{15} x_{2} x_{4}^{19} x_{5}^{6} + x_{1}^{3} x_{2}^{13} x_{4}^{19} x_{5}^{6} \\
    & + x_{1} x_{2}^{15} x_{4}^{19} x_{5}^{6} + x_{1}^{15} x_{3} x_{4}^{19} x_{5}^{6} + x_{1} x_{2}^{14} x_{3} x_{4}^{19} x_{5}^{6} + x_{2}^{15} x_{3} x_{4}^{19} x_{5}^{6} \\
    & + x_{1} x_{2}^{3} x_{3}^{12} x_{4}^{19} x_{5}^{6} + x_{1}^{3} x_{3}^{13} x_{4}^{19} x_{5}^{6} + x_{1} x_{2}^{2} x_{3}^{13} x_{4}^{19} x_{5}^{6} + x_{2}^{3} x_{3}^{13} x_{4}^{19} x_{5}^{6} \\
    & + x_{1} x_{3}^{15} x_{4}^{19} x_{5}^{6} + x_{2} x_{3}^{15} x_{4}^{19} x_{5}^{6} + x_{1}^{3} x_{2}^{5} x_{3}^{3} x_{4}^{24} x_{5}^{6} + x_{1} x_{2}^{7} x_{3}^{3} x_{4}^{24} x_{5}^{6} \\
    & + x_{1}^{3} x_{2} x_{3}^{7} x_{4}^{24} x_{5}^{6} + x_{1}^{7} x_{2} x_{3}^{2} x_{4}^{25} x_{5}^{6} + x_{1}^{3} x_{2}^{5} x_{3}^{2} x_{4}^{25} x_{5}^{6} + x_{1} x_{2}^{7} x_{3}^{2} x_{4}^{25} x_{5}^{6} \\
    & + x_{1}^{3} x_{2}^{3} x_{3}^{4} x_{4}^{25} x_{5}^{6} + x_{1} x_{2}^{7} x_{3} x_{4}^{26} x_{5}^{6} + x_{1} x_{2} x_{3}^{7} x_{4}^{26} x_{5}^{6} + x_{1}^{7} x_{2} x_{4}^{27} x_{5}^{6} \\
    & + x_{1}^{3} x_{2}^{5} x_{4}^{27} x_{5}^{6} + x_{1} x_{2}^{7} x_{4}^{27} x_{5}^{6} + x_{1}^{7} x_{3} x_{4}^{27} x_{5}^{6} + x_{1}^{3} x_{2}^{4} x_{3} x_{4}^{27} x_{5}^{6} \\
    & + x_{2}^{7} x_{3} x_{4}^{27} x_{5}^{6} + x_{1} x_{2}^{3} x_{3}^{4} x_{4}^{27} x_{5}^{6} + x_{1}^{3} x_{3}^{5} x_{4}^{27} x_{5}^{6} + x_{1} x_{2}^{2} x_{3}^{5} x_{4}^{27} x_{5}^{6} \\
    & + x_{2}^{3} x_{3}^{5} x_{4}^{27} x_{5}^{6} + x_{1} x_{3}^{7} x_{4}^{27} x_{5}^{6} + x_{2} x_{3}^{7} x_{4}^{27} x_{5}^{6} + x_{1}^{3} x_{2} x_{3}^{3} x_{4}^{28} x_{5}^{6} \\
    & + x_{1}^{3} x_{2}^{3} x_{4}^{29} x_{5}^{6} + x_{1} x_{2}^{3} x_{3}^{2} x_{4}^{29} x_{5}^{6} + x_{1}^{3} x_{3}^{3} x_{4}^{29} x_{5}^{6} + x_{2}^{3} x_{3}^{3} x_{4}^{29} x_{5}^{6} \\
    & + x_{1} x_{2}^{3} x_{3} x_{4}^{30} x_{5}^{6} + x_{1} x_{2} x_{3}^{3} x_{4}^{30} x_{5}^{6} + x_{1}^{15} x_{2}^{19} x_{5}^{7} + x_{1}^{7} x_{2}^{27} x_{5}^{7} \\
    & + x_{1}^{15} x_{2} x_{3}^{18} x_{5}^{7} + x_{1}^{3} x_{2}^{13} x_{3}^{18} x_{5}^{7} + x_{1} x_{2}^{15} x_{3}^{18} x_{5}^{7} + x_{1}^{15} x_{3}^{19} x_{5}^{7} \\
    & + x_{1} x_{2}^{14} x_{3}^{19} x_{5}^{7} + x_{2}^{15} x_{3}^{19} x_{5}^{7} + x_{1}^{7} x_{2} x_{3}^{26} x_{5}^{7} + x_{1}^{3} x_{2}^{5} x_{3}^{26} x_{5}^{7} \\
    & + x_{1} x_{2}^{7} x_{3}^{26} x_{5}^{7} + x_{1}^{7} x_{3}^{27} x_{5}^{7} + x_{1} x_{2}^{6} x_{3}^{27} x_{5}^{7} + x_{2}^{7} x_{3}^{27} x_{5}^{7} \\
    & + x_{1}^{15} x_{2} x_{3}^{2} x_{4}^{16} x_{5}^{7} + x_{1}^{3} x_{2}^{13} x_{3}^{2} x_{4}^{16} x_{5}^{7} + x_{1} x_{2}^{15} x_{3}^{2} x_{4}^{16} x_{5}^{7} + x_{1}^{3} x_{2}^{5} x_{3}^{10} x_{4}^{16} x_{5}^{7} \\
    & + x_{1}^{3} x_{2} x_{3}^{14} x_{4}^{16} x_{5}^{7} + x_{1} x_{2}^{3} x_{3}^{14} x_{4}^{16} x_{5}^{7} + x_{1} x_{2}^{2} x_{3}^{15} x_{4}^{16} x_{5}^{7} + x_{1}^{15} x_{2} x_{4}^{18} x_{5}^{7} \\
    & + x_{1}^{3} x_{2}^{13} x_{4}^{18} x_{5}^{7} + x_{1} x_{2}^{15} x_{4}^{18} x_{5}^{7} + x_{1}^{15} x_{3} x_{4}^{18} x_{5}^{7} + x_{1} x_{2}^{14} x_{3} x_{4}^{18} x_{5}^{7} \\
    & + x_{2}^{15} x_{3} x_{4}^{18} x_{5}^{7} + x_{1}^{3} x_{2} x_{3}^{12} x_{4}^{18} x_{5}^{7} + x_{1} x_{2}^{3} x_{3}^{12} x_{4}^{18} x_{5}^{7} + x_{1}^{3} x_{3}^{13} x_{4}^{18} x_{5}^{7} \\
    & + x_{1} x_{2}^{2} x_{3}^{13} x_{4}^{18} x_{5}^{7} + x_{2}^{3} x_{3}^{13} x_{4}^{18} x_{5}^{7} + x_{1} x_{2} x_{3}^{14} x_{4}^{18} x_{5}^{7} + x_{1} x_{3}^{15} x_{4}^{18} x_{5}^{7} \\
    & + x_{2} x_{3}^{15} x_{4}^{18} x_{5}^{7} + x_{1}^{15} x_{4}^{19} x_{5}^{7} + x_{1} x_{2}^{14} x_{4}^{19} x_{5}^{7} + x_{2}^{15} x_{4}^{19} x_{5}^{7} \\
    & + x_{1} x_{2}^{2} x_{3}^{12} x_{4}^{19} x_{5}^{7} + x_{1} x_{3}^{14} x_{4}^{19} x_{5}^{7} + x_{2} x_{3}^{14} x_{4}^{19} x_{5}^{7} + x_{3}^{15} x_{4}^{19} x_{5}^{7} \\
    & + x_{1}^{7} x_{2} x_{3}^{2} x_{4}^{24} x_{5}^{7} + x_{1} x_{2}^{7} x_{3}^{2} x_{4}^{24} x_{5}^{7} + x_{1} x_{2}^{2} x_{3}^{7} x_{4}^{24} x_{5}^{7} + x_{1}^{7} x_{2} x_{4}^{26} x_{5}^{7} \\
    & + x_{1}^{3} x_{2}^{5} x_{4}^{26} x_{5}^{7} + x_{1} x_{2}^{7} x_{4}^{26} x_{5}^{7} + x_{1}^{7} x_{3} x_{4}^{26} x_{5}^{7} + x_{1} x_{2}^{6} x_{3} x_{4}^{26} x_{5}^{7} \\
    & + x_{2}^{7} x_{3} x_{4}^{26} x_{5}^{7} + x_{1}^{3} x_{2} x_{3}^{4} x_{4}^{26} x_{5}^{7} + x_{1} x_{2}^{3} x_{3}^{4} x_{4}^{26} x_{5}^{7} + x_{1}^{3} x_{3}^{5} x_{4}^{26} x_{5}^{7} \\
    & + x_{1} x_{2}^{2} x_{3}^{5} x_{4}^{26} x_{5}^{7} + x_{2}^{3} x_{3}^{5} x_{4}^{26} x_{5}^{7} + x_{1} x_{3}^{7} x_{4}^{26} x_{5}^{7} + x_{2} x_{3}^{7} x_{4}^{26} x_{5}^{7} \\
    & + x_{1}^{7} x_{4}^{27} x_{5}^{7} + x_{1} x_{2}^{6} x_{4}^{27} x_{5}^{7} + x_{2}^{7} x_{4}^{27} x_{5}^{7} + x_{1} x_{2}^{2} x_{3}^{4} x_{4}^{27} x_{5}^{7} \\
    & + x_{1} x_{3}^{6} x_{4}^{27} x_{5}^{7} + x_{2} x_{3}^{6} x_{4}^{27} x_{5}^{7} + x_{3}^{7} x_{4}^{27} x_{5}^{7} + x_{1}^{3} x_{2} x_{3}^{2} x_{4}^{28} x_{5}^{7} \\
    & + x_{1} x_{2}^{3} x_{3}^{2} x_{4}^{28} x_{5}^{7} + x_{1} x_{2} x_{3}^{2} x_{4}^{30} x_{5}^{7} + x_{1}^{7} x_{2}^{7} x_{3}^{11} x_{5}^{16} + x_{1}^{15} x_{2}^{7} x_{3} x_{4}^{2} x_{5}^{16} \\
    & + x_{1}^{7} x_{2}^{11} x_{3}^{5} x_{4}^{2} x_{5}^{16} + x_{1}^{15} x_{2} x_{3}^{7} x_{4}^{2} x_{5}^{16} + x_{1}^{3} x_{2}^{13} x_{3}^{7} x_{4}^{2} x_{5}^{16} + x_{1} x_{2}^{15} x_{3}^{7} x_{4}^{2} x_{5}^{16} \\
    & + x_{1}^{7} x_{2}^{7} x_{3}^{9} x_{4}^{2} x_{5}^{16} + x_{1}^{3} x_{2}^{5} x_{3}^{15} x_{4}^{2} x_{5}^{16} + x_{1}^{7} x_{2}^{7} x_{3}^{8} x_{4}^{3} x_{5}^{16} + x_{1}^{15} x_{2}^{3} x_{3} x_{4}^{6} x_{5}^{16} \\
    & + x_{1}^{7} x_{2}^{11} x_{3} x_{4}^{6} x_{5}^{16} + x_{1}^{3} x_{2}^{15} x_{3} x_{4}^{6} x_{5}^{16} + x_{1}^{15} x_{2} x_{3}^{3} x_{4}^{6} x_{5}^{16} + x_{1}^{7} x_{2}^{9} x_{3}^{3} x_{4}^{6} x_{5}^{16} \\
    & + x_{1} x_{2}^{15} x_{3}^{3} x_{4}^{6} x_{5}^{16} + x_{1}^{7} x_{2} x_{3}^{11} x_{4}^{6} x_{5}^{16} + x_{1}^{3} x_{2}^{5} x_{3}^{11} x_{4}^{6} x_{5}^{16} + x_{1} x_{2}^{7} x_{3}^{11} x_{4}^{6} x_{5}^{16} \\
    & + x_{1}^{3} x_{2} x_{3}^{15} x_{4}^{6} x_{5}^{16} + x_{1} x_{2}^{3} x_{3}^{15} x_{4}^{6} x_{5}^{16} + x_{1}^{15} x_{2} x_{3}^{2} x_{4}^{7} x_{5}^{16} + x_{1}^{3} x_{2}^{13} x_{3}^{2} x_{4}^{7} x_{5}^{16} \\
    & + x_{1} x_{2}^{15} x_{3}^{2} x_{4}^{7} x_{5}^{16} + x_{1}^{3} x_{2}^{5} x_{3}^{10} x_{4}^{7} x_{5}^{16} + x_{1}^{3} x_{2} x_{3}^{14} x_{4}^{7} x_{5}^{16} + x_{1} x_{2}^{3} x_{3}^{14} x_{4}^{7} x_{5}^{16} \\
    & + x_{1} x_{2}^{2} x_{3}^{15} x_{4}^{7} x_{5}^{16} + x_{1}^{7} x_{2}^{7} x_{3} x_{4}^{10} x_{5}^{16} + x_{1}^{3} x_{2}^{7} x_{3}^{5} x_{4}^{10} x_{5}^{16} + x_{1}^{7} x_{2}^{7} x_{4}^{11} x_{5}^{16} \\
    & + x_{1}^{3} x_{2}^{5} x_{3}^{6} x_{4}^{11} x_{5}^{16} + x_{1} x_{2}^{7} x_{3}^{6} x_{4}^{11} x_{5}^{16} + x_{1}^{7} x_{3}^{7} x_{4}^{11} x_{5}^{16} + x_{1} x_{2}^{6} x_{3}^{7} x_{4}^{11} x_{5}^{16} \\
    & + x_{2}^{7} x_{3}^{7} x_{4}^{11} x_{5}^{16} + x_{1}^{3} x_{2}^{7} x_{3}^{3} x_{4}^{12} x_{5}^{16} + x_{1}^{7} x_{2} x_{3}^{3} x_{4}^{14} x_{5}^{16} + x_{1}^{3} x_{2}^{5} x_{3}^{3} x_{4}^{14} x_{5}^{16} \\
    & + x_{1} x_{2}^{7} x_{3}^{3} x_{4}^{14} x_{5}^{16} + x_{1} x_{2}^{3} x_{3}^{7} x_{4}^{14} x_{5}^{16} + x_{1}^{3} x_{2}^{5} x_{3}^{2} x_{4}^{15} x_{5}^{16} + x_{1}^{15} x_{2}^{7} x_{3} x_{5}^{18} \\
    & + x_{1}^{15} x_{2}^{3} x_{3}^{5} x_{5}^{18} + x_{1}^{7} x_{2}^{11} x_{3}^{5} x_{5}^{18} + x_{1}^{3} x_{2}^{15} x_{3}^{5} x_{5}^{18} + x_{1}^{15} x_{2} x_{3}^{7} x_{5}^{18} \\
    & + x_{1}^{3} x_{2}^{13} x_{3}^{7} x_{5}^{18} + x_{1} x_{2}^{15} x_{3}^{7} x_{5}^{18} + x_{1}^{15} x_{2}^{7} x_{4} x_{5}^{18} + x_{1}^{15} x_{2} x_{3}^{6} x_{4} x_{5}^{18} \\
    & + x_{1}^{3} x_{2}^{13} x_{3}^{6} x_{4} x_{5}^{18} + x_{1} x_{2}^{15} x_{3}^{6} x_{4} x_{5}^{18} + x_{1}^{15} x_{3}^{7} x_{4} x_{5}^{18} + x_{1} x_{2}^{14} x_{3}^{7} x_{4} x_{5}^{18} \\
    & + x_{2}^{15} x_{3}^{7} x_{4} x_{5}^{18} + x_{1}^{15} x_{2}^{3} x_{3} x_{4}^{4} x_{5}^{18} + x_{1}^{3} x_{2}^{15} x_{3} x_{4}^{4} x_{5}^{18} + x_{1}^{15} x_{2} x_{3}^{3} x_{4}^{4} x_{5}^{18} \\
    & + x_{1}^{7} x_{2}^{9} x_{3}^{3} x_{4}^{4} x_{5}^{18} + x_{1} x_{2}^{15} x_{3}^{3} x_{4}^{4} x_{5}^{18} + x_{1}^{7} x_{2}^{3} x_{3}^{9} x_{4}^{4} x_{5}^{18} + x_{1}^{7} x_{2} x_{3}^{11} x_{4}^{4} x_{5}^{18} \\
    & + x_{1}^{3} x_{2}^{5} x_{3}^{11} x_{4}^{4} x_{5}^{18} + x_{1} x_{2}^{7} x_{3}^{11} x_{4}^{4} x_{5}^{18} + x_{1}^{3} x_{2}^{3} x_{3}^{13} x_{4}^{4} x_{5}^{18} + x_{1}^{3} x_{2} x_{3}^{15} x_{4}^{4} x_{5}^{18} \\
    & + x_{1} x_{2}^{3} x_{3}^{15} x_{4}^{4} x_{5}^{18} + x_{1}^{15} x_{2}^{3} x_{4}^{5} x_{5}^{18} + x_{1}^{7} x_{2}^{11} x_{4}^{5} x_{5}^{18} + x_{1}^{3} x_{2}^{15} x_{4}^{5} x_{5}^{18} \\
    & + x_{1}^{15} x_{2} x_{3}^{2} x_{4}^{5} x_{5}^{18} + x_{1}^{7} x_{2}^{9} x_{3}^{2} x_{4}^{5} x_{5}^{18} + x_{1}^{3} x_{2}^{13} x_{3}^{2} x_{4}^{5} x_{5}^{18} + x_{1} x_{2}^{15} x_{3}^{2} x_{4}^{5} x_{5}^{18} \\
    & + x_{1}^{15} x_{3}^{3} x_{4}^{5} x_{5}^{18} + x_{1} x_{2}^{14} x_{3}^{3} x_{4}^{5} x_{5}^{18} + x_{2}^{15} x_{3}^{3} x_{4}^{5} x_{5}^{18} + x_{1}^{7} x_{2} x_{3}^{10} x_{4}^{5} x_{5}^{18} \\
    & + x_{1}^{3} x_{2}^{5} x_{3}^{10} x_{4}^{5} x_{5}^{18} + x_{1} x_{2}^{7} x_{3}^{10} x_{4}^{5} x_{5}^{18} + x_{1}^{7} x_{3}^{11} x_{4}^{5} x_{5}^{18} + x_{1}^{3} x_{2}^{4} x_{3}^{11} x_{4}^{5} x_{5}^{18} \\
    & + x_{2}^{7} x_{3}^{11} x_{4}^{5} x_{5}^{18} + x_{1}^{3} x_{2} x_{3}^{14} x_{4}^{5} x_{5}^{18} + x_{1} x_{2}^{3} x_{3}^{14} x_{4}^{5} x_{5}^{18} + x_{1}^{3} x_{3}^{15} x_{4}^{5} x_{5}^{18} \\
    & + x_{1} x_{2}^{2} x_{3}^{15} x_{4}^{5} x_{5}^{18} + x_{2}^{3} x_{3}^{15} x_{4}^{5} x_{5}^{18} + x_{1}^{15} x_{2} x_{3} x_{4}^{6} x_{5}^{18} + x_{1}^{3} x_{2}^{13} x_{3} x_{4}^{6} x_{5}^{18} \\
    & + x_{1} x_{2}^{15} x_{3} x_{4}^{6} x_{5}^{18} + x_{1}^{7} x_{2} x_{3}^{9} x_{4}^{6} x_{5}^{18} + x_{1} x_{2}^{3} x_{3}^{13} x_{4}^{6} x_{5}^{18} + x_{1} x_{2} x_{3}^{15} x_{4}^{6} x_{5}^{18} \\
    & + x_{1}^{15} x_{2} x_{4}^{7} x_{5}^{18} + x_{1}^{3} x_{2}^{13} x_{4}^{7} x_{5}^{18} + x_{1} x_{2}^{15} x_{4}^{7} x_{5}^{18} + x_{1}^{15} x_{3} x_{4}^{7} x_{5}^{18} \\
    & + x_{1} x_{2}^{14} x_{3} x_{4}^{7} x_{5}^{18} + x_{2}^{15} x_{3} x_{4}^{7} x_{5}^{18} + x_{1}^{3} x_{2} x_{3}^{12} x_{4}^{7} x_{5}^{18} + x_{1} x_{2}^{3} x_{3}^{12} x_{4}^{7} x_{5}^{18} \\
    & + x_{1}^{3} x_{3}^{13} x_{4}^{7} x_{5}^{18} + x_{1} x_{2}^{2} x_{3}^{13} x_{4}^{7} x_{5}^{18} + x_{2}^{3} x_{3}^{13} x_{4}^{7} x_{5}^{18} + x_{1} x_{2} x_{3}^{14} x_{4}^{7} x_{5}^{18} \\
    & + x_{1} x_{3}^{15} x_{4}^{7} x_{5}^{18} + x_{2} x_{3}^{15} x_{4}^{7} x_{5}^{18} + x_{1}^{7} x_{2}^{7} x_{3} x_{4}^{8} x_{5}^{18} + x_{1}^{3} x_{2}^{7} x_{3}^{5} x_{4}^{8} x_{5}^{18} \\
    & + x_{1}^{7} x_{2}^{3} x_{3}^{4} x_{4}^{9} x_{5}^{18} + x_{1}^{3} x_{2}^{7} x_{3}^{4} x_{4}^{9} x_{5}^{18} + x_{1}^{7} x_{2} x_{3}^{6} x_{4}^{9} x_{5}^{18} + x_{1}^{3} x_{2}^{5} x_{3}^{6} x_{4}^{9} x_{5}^{18} \\
    & + x_{1} x_{2}^{7} x_{3}^{6} x_{4}^{9} x_{5}^{18} + x_{1}^{3} x_{2}^{7} x_{3} x_{4}^{12} x_{5}^{18} + x_{1}^{3} x_{2}^{5} x_{3}^{3} x_{4}^{12} x_{5}^{18} + x_{1}^{3} x_{2}^{5} x_{3}^{2} x_{4}^{13} x_{5}^{18} \\
    & + x_{1}^{3} x_{2} x_{3}^{6} x_{4}^{13} x_{5}^{18} + x_{1}^{3} x_{2}^{5} x_{3} x_{4}^{14} x_{5}^{18} + x_{1} x_{2}^{7} x_{3} x_{4}^{14} x_{5}^{18} + x_{1}^{3} x_{2} x_{3}^{5} x_{4}^{14} x_{5}^{18} \\
    & + x_{1} x_{2} x_{3}^{6} x_{4}^{15} x_{5}^{18} + x_{1}^{15} x_{2}^{7} x_{5}^{19} + x_{1}^{15} x_{2} x_{3}^{6} x_{5}^{19} + x_{1}^{3} x_{2}^{13} x_{3}^{6} x_{5}^{19} \\
    & + x_{1} x_{2}^{15} x_{3}^{6} x_{5}^{19} + x_{1}^{15} x_{3}^{7} x_{5}^{19} + x_{1} x_{2}^{14} x_{3}^{7} x_{5}^{19} + x_{2}^{15} x_{3}^{7} x_{5}^{19} \\
    & + x_{1}^{7} x_{2}^{7} x_{3}^{8} x_{5}^{19} + x_{1}^{7} x_{2} x_{3}^{14} x_{5}^{19} + x_{1}^{3} x_{2}^{5} x_{3}^{14} x_{5}^{19} + x_{1} x_{2}^{7} x_{3}^{14} x_{5}^{19} \\
    & + x_{1}^{15} x_{2} x_{3}^{2} x_{4}^{4} x_{5}^{19} + x_{1}^{3} x_{2}^{13} x_{3}^{2} x_{4}^{4} x_{5}^{19} + x_{1} x_{2}^{15} x_{3}^{2} x_{4}^{4} x_{5}^{19} + x_{1}^{3} x_{2}^{5} x_{3}^{10} x_{4}^{4} x_{5}^{19} \\
    & + x_{1}^{3} x_{2} x_{3}^{14} x_{4}^{4} x_{5}^{19} + x_{1} x_{2}^{3} x_{3}^{14} x_{4}^{4} x_{5}^{19} + x_{1} x_{2}^{2} x_{3}^{15} x_{4}^{4} x_{5}^{19} + x_{1}^{15} x_{2} x_{4}^{6} x_{5}^{19} \\
    & + x_{1}^{3} x_{2}^{13} x_{4}^{6} x_{5}^{19} + x_{1} x_{2}^{15} x_{4}^{6} x_{5}^{19} + x_{1}^{15} x_{3} x_{4}^{6} x_{5}^{19} + x_{1} x_{2}^{14} x_{3} x_{4}^{6} x_{5}^{19} \\
    & + x_{2}^{15} x_{3} x_{4}^{6} x_{5}^{19} + x_{1}^{3} x_{2} x_{3}^{12} x_{4}^{6} x_{5}^{19} + x_{1} x_{2}^{3} x_{3}^{12} x_{4}^{6} x_{5}^{19} + x_{1}^{3} x_{3}^{13} x_{4}^{6} x_{5}^{19} \\
    & + x_{1} x_{2}^{2} x_{3}^{13} x_{4}^{6} x_{5}^{19} + x_{2}^{3} x_{3}^{13} x_{4}^{6} x_{5}^{19} + x_{1} x_{2} x_{3}^{14} x_{4}^{6} x_{5}^{19} + x_{1} x_{3}^{15} x_{4}^{6} x_{5}^{19} \\
    & + x_{2} x_{3}^{15} x_{4}^{6} x_{5}^{19} + x_{1}^{15} x_{4}^{7} x_{5}^{19} + x_{1} x_{2}^{14} x_{4}^{7} x_{5}^{19} + x_{2}^{15} x_{4}^{7} x_{5}^{19} \\
    & + x_{1} x_{2}^{2} x_{3}^{12} x_{4}^{7} x_{5}^{19} + x_{1} x_{3}^{14} x_{4}^{7} x_{5}^{19} + x_{2} x_{3}^{14} x_{4}^{7} x_{5}^{19} + x_{3}^{15} x_{4}^{7} x_{5}^{19} \\
    & + x_{1}^{7} x_{2}^{7} x_{4}^{8} x_{5}^{19} + x_{1}^{3} x_{2}^{5} x_{3}^{6} x_{4}^{8} x_{5}^{19} + x_{1} x_{2}^{7} x_{3}^{6} x_{4}^{8} x_{5}^{19} + x_{1}^{7} x_{3}^{7} x_{4}^{8} x_{5}^{19} \\
    & + x_{1} x_{2}^{6} x_{3}^{7} x_{4}^{8} x_{5}^{19} + x_{2}^{7} x_{3}^{7} x_{4}^{8} x_{5}^{19} + x_{1} x_{2}^{7} x_{3}^{2} x_{4}^{12} x_{5}^{19} + x_{1} x_{2}^{2} x_{3}^{7} x_{4}^{12} x_{5}^{19} \\
    & + x_{1}^{7} x_{2} x_{4}^{14} x_{5}^{19} + x_{1}^{3} x_{2}^{5} x_{4}^{14} x_{5}^{19} + x_{1} x_{2}^{7} x_{4}^{14} x_{5}^{19} + x_{1}^{7} x_{3} x_{4}^{14} x_{5}^{19} \\
    & + x_{1} x_{2}^{6} x_{3} x_{4}^{14} x_{5}^{19} + x_{2}^{7} x_{3} x_{4}^{14} x_{5}^{19} + x_{1} x_{2}^{3} x_{3}^{4} x_{4}^{14} x_{5}^{19} + x_{1}^{3} x_{3}^{5} x_{4}^{14} x_{5}^{19} \\
    & + x_{1} x_{2}^{2} x_{3}^{5} x_{4}^{14} x_{5}^{19} + x_{2}^{3} x_{3}^{5} x_{4}^{14} x_{5}^{19} + x_{1} x_{2} x_{3}^{6} x_{4}^{14} x_{5}^{19} + x_{1} x_{3}^{7} x_{4}^{14} x_{5}^{19} \\
    & + x_{2} x_{3}^{7} x_{4}^{14} x_{5}^{19} + x_{1}^{15} x_{2}^{3} x_{3} x_{4}^{2} x_{5}^{20} + x_{1}^{7} x_{2}^{11} x_{3} x_{4}^{2} x_{5}^{20} + x_{1}^{3} x_{2}^{15} x_{3} x_{4}^{2} x_{5}^{20} \\
    & + x_{1}^{15} x_{2} x_{3}^{3} x_{4}^{2} x_{5}^{20} + x_{1}^{7} x_{2}^{9} x_{3}^{3} x_{4}^{2} x_{5}^{20} + x_{1}^{3} x_{2}^{13} x_{3}^{3} x_{4}^{2} x_{5}^{20} + x_{1} x_{2}^{15} x_{3}^{3} x_{4}^{2} x_{5}^{20} \\
    & + x_{1}^{7} x_{2}^{3} x_{3}^{9} x_{4}^{2} x_{5}^{20} + x_{1}^{3} x_{2}^{3} x_{3}^{13} x_{4}^{2} x_{5}^{20} + x_{1}^{3} x_{2} x_{3}^{15} x_{4}^{2} x_{5}^{20} + x_{1} x_{2}^{3} x_{3}^{15} x_{4}^{2} x_{5}^{20} \\
    & + x_{1}^{3} x_{2}^{5} x_{3}^{10} x_{4}^{3} x_{5}^{20} + x_{1}^{3} x_{2} x_{3}^{14} x_{4}^{3} x_{5}^{20} + x_{1}^{3} x_{2}^{7} x_{3}^{3} x_{4}^{8} x_{5}^{20} + x_{1}^{3} x_{2}^{3} x_{3}^{7} x_{4}^{8} x_{5}^{20} \\
    & + x_{1}^{7} x_{2}^{3} x_{3} x_{4}^{10} x_{5}^{20} + x_{1}^{7} x_{2} x_{3}^{3} x_{4}^{10} x_{5}^{20} + x_{1}^{3} x_{2}^{5} x_{3}^{3} x_{4}^{10} x_{5}^{20} + x_{1}^{3} x_{2}^{3} x_{3}^{5} x_{4}^{10} x_{5}^{20} \\
    & + x_{1} x_{2}^{7} x_{3}^{2} x_{4}^{11} x_{5}^{20} + x_{1} x_{2}^{2} x_{3}^{7} x_{4}^{11} x_{5}^{20} + x_{1}^{3} x_{2}^{3} x_{3} x_{4}^{14} x_{5}^{20} + x_{1} x_{2}^{3} x_{3}^{3} x_{4}^{14} x_{5}^{20} \\
    & + x_{1}^{3} x_{2} x_{3}^{2} x_{4}^{15} x_{5}^{20} + x_{1} x_{2}^{3} x_{3}^{2} x_{4}^{15} x_{5}^{20} + x_{1}^{7} x_{2} x_{3}^{11} x_{5}^{22} + x_{1}^{3} x_{2}^{5} x_{3}^{11} x_{5}^{22} \\
    & + x_{1} x_{2}^{7} x_{3}^{11} x_{5}^{22} + x_{1}^{3} x_{2}^{13} x_{3}^{2} x_{4} x_{5}^{22} + x_{1}^{3} x_{2}^{3} x_{3}^{12} x_{4} x_{5}^{22} + x_{1}^{15} x_{2} x_{3} x_{4}^{2} x_{5}^{22} \\
    & + x_{1}^{7} x_{2}^{9} x_{3} x_{4}^{2} x_{5}^{22} + x_{1}^{3} x_{2}^{13} x_{3} x_{4}^{2} x_{5}^{22} + x_{1} x_{2}^{15} x_{3} x_{4}^{2} x_{5}^{22} + x_{1}^{7} x_{2} x_{3}^{9} x_{4}^{2} x_{5}^{22} \\
    & + x_{1}^{3} x_{2} x_{3}^{13} x_{4}^{2} x_{5}^{22} + x_{1} x_{2} x_{3}^{15} x_{4}^{2} x_{5}^{22} + x_{1} x_{2} x_{3}^{14} x_{4}^{3} x_{5}^{22} + x_{1}^{7} x_{2}^{3} x_{3} x_{4}^{8} x_{5}^{22} \\
    & + x_{1}^{7} x_{2} x_{3}^{3} x_{4}^{8} x_{5}^{22} + x_{1}^{3} x_{2}^{5} x_{3}^{3} x_{4}^{8} x_{5}^{22} + x_{1} x_{2}^{7} x_{3}^{3} x_{4}^{8} x_{5}^{22} + x_{1}^{3} x_{2}^{3} x_{3}^{5} x_{4}^{8} x_{5}^{22} \\
    & + x_{1} x_{2}^{3} x_{3}^{7} x_{4}^{8} x_{5}^{22} + x_{1}^{3} x_{2}^{5} x_{3}^{2} x_{4}^{9} x_{5}^{22} + x_{1}^{3} x_{2} x_{3}^{6} x_{4}^{9} x_{5}^{22} + x_{1} x_{2}^{7} x_{3} x_{4}^{10} x_{5}^{22} \\
    & + x_{1}^{3} x_{2} x_{3}^{5} x_{4}^{10} x_{5}^{22} + x_{1} x_{2} x_{3}^{7} x_{4}^{10} x_{5}^{22} + x_{1}^{7} x_{2} x_{4}^{11} x_{5}^{22} + x_{1}^{3} x_{2}^{5} x_{4}^{11} x_{5}^{22} \\
    & + x_{1} x_{2}^{7} x_{4}^{11} x_{5}^{22} + x_{1}^{7} x_{3} x_{4}^{11} x_{5}^{22} + x_{1} x_{2}^{6} x_{3} x_{4}^{11} x_{5}^{22} + x_{2}^{7} x_{3} x_{4}^{11} x_{5}^{22} \\
    & + x_{1} x_{2}^{3} x_{3}^{4} x_{4}^{11} x_{5}^{22} + x_{1}^{3} x_{3}^{5} x_{4}^{11} x_{5}^{22} + x_{1} x_{2}^{2} x_{3}^{5} x_{4}^{11} x_{5}^{22} + x_{2}^{3} x_{3}^{5} x_{4}^{11} x_{5}^{22} \\
    & + x_{1} x_{2} x_{3}^{6} x_{4}^{11} x_{5}^{22} + x_{1} x_{3}^{7} x_{4}^{11} x_{5}^{22} + x_{2} x_{3}^{7} x_{4}^{11} x_{5}^{22} + x_{1}^{3} x_{2}^{3} x_{3} x_{4}^{12} x_{5}^{22} \\
    & + x_{1} x_{2}^{3} x_{3}^{3} x_{4}^{12} x_{5}^{22} + x_{1}^{3} x_{2} x_{3} x_{4}^{14} x_{5}^{22} + x_{1} x_{2}^{3} x_{3} x_{4}^{14} x_{5}^{22} + x_{1} x_{2} x_{3}^{3} x_{4}^{14} x_{5}^{22} \\
    & + x_{1} x_{2} x_{3}^{2} x_{4}^{15} x_{5}^{22} + x_{1}^{3} x_{2}^{5} x_{3}^{8} x_{4}^{2} x_{5}^{23} + x_{1}^{3} x_{2} x_{3}^{12} x_{4}^{2} x_{5}^{23} + x_{1} x_{2}^{3} x_{3}^{12} x_{4}^{2} x_{5}^{23} \\
    & + x_{1} x_{2} x_{3}^{14} x_{4}^{2} x_{5}^{23} + x_{1}^{3} x_{2}^{5} x_{3}^{2} x_{4}^{8} x_{5}^{23} + x_{1} x_{2} x_{3}^{6} x_{4}^{10} x_{5}^{23} + x_{1}^{3} x_{2} x_{3}^{2} x_{4}^{12} x_{5}^{23} \\
    & + x_{1} x_{2}^{3} x_{3}^{2} x_{4}^{12} x_{5}^{23} + x_{1} x_{2} x_{3}^{2} x_{4}^{14} x_{5}^{23} + x_{1} x_{2}^{7} x_{3}^{7} x_{4}^{2} x_{5}^{24} + x_{1}^{7} x_{2} x_{3}^{6} x_{4}^{3} x_{5}^{24} \\
    & + x_{1}^{3} x_{2}^{5} x_{3}^{6} x_{4}^{3} x_{5}^{24} + x_{1}^{3} x_{2}^{7} x_{3}^{3} x_{4}^{4} x_{5}^{24} + x_{1}^{3} x_{2}^{3} x_{3}^{7} x_{4}^{4} x_{5}^{24} + x_{1}^{7} x_{2}^{3} x_{3} x_{4}^{6} x_{5}^{24} \\
    & + x_{1}^{3} x_{2}^{3} x_{3}^{5} x_{4}^{6} x_{5}^{24} + x_{1}^{3} x_{2} x_{3}^{7} x_{4}^{6} x_{5}^{24} + x_{1}^{7} x_{2} x_{3}^{2} x_{4}^{7} x_{5}^{24} + x_{1} x_{2}^{7} x_{3}^{2} x_{4}^{7} x_{5}^{24} \\
    & + x_{1} x_{2}^{2} x_{3}^{7} x_{4}^{7} x_{5}^{24} + x_{1}^{7} x_{2}^{7} x_{3} x_{5}^{26} + x_{1}^{7} x_{2} x_{3}^{7} x_{5}^{26} + x_{1}^{3} x_{2}^{5} x_{3}^{7} x_{5}^{26} \\
    & + x_{1} x_{2}^{7} x_{3}^{7} x_{5}^{26} + x_{1}^{7} x_{2}^{7} x_{4} x_{5}^{26} + x_{1}^{3} x_{2}^{7} x_{3}^{4} x_{4} x_{5}^{26} + x_{1}^{7} x_{2} x_{3}^{6} x_{4} x_{5}^{26} \\
    & + x_{1} x_{2}^{7} x_{3}^{6} x_{4} x_{5}^{26} + x_{1}^{7} x_{3}^{7} x_{4} x_{5}^{26} + x_{1}^{3} x_{2}^{4} x_{3}^{7} x_{4} x_{5}^{26} + x_{2}^{7} x_{3}^{7} x_{4} x_{5}^{26} \\
    & + x_{1}^{7} x_{2} x_{3}^{3} x_{4}^{4} x_{5}^{26} + x_{1}^{3} x_{2}^{3} x_{3}^{5} x_{4}^{4} x_{5}^{26} + x_{1}^{7} x_{2} x_{3}^{2} x_{4}^{5} x_{5}^{26} + x_{1}^{3} x_{2}^{5} x_{3}^{2} x_{4}^{5} x_{5}^{26} \\
    & + x_{1} x_{2}^{7} x_{3}^{2} x_{4}^{5} x_{5}^{26} + x_{1}^{3} x_{2}^{4} x_{3}^{3} x_{4}^{5} x_{5}^{26} + x_{1} x_{2}^{6} x_{3}^{3} x_{4}^{5} x_{5}^{26} + x_{1} x_{2}^{3} x_{3}^{6} x_{4}^{5} x_{5}^{26} \\
    & + x_{1}^{7} x_{2} x_{3} x_{4}^{6} x_{5}^{26} + x_{1} x_{2}^{3} x_{3}^{5} x_{4}^{6} x_{5}^{26} + x_{1}^{7} x_{2} x_{4}^{7} x_{5}^{26} + x_{1}^{3} x_{2}^{5} x_{4}^{7} x_{5}^{26} \\
    & + x_{1} x_{2}^{7} x_{4}^{7} x_{5}^{26} + x_{1}^{7} x_{3} x_{4}^{7} x_{5}^{26} + x_{1} x_{2}^{6} x_{3} x_{4}^{7} x_{5}^{26} + x_{2}^{7} x_{3} x_{4}^{7} x_{5}^{26} \\
    & + x_{1}^{3} x_{2} x_{3}^{4} x_{4}^{7} x_{5}^{26} + x_{1} x_{2}^{3} x_{3}^{4} x_{4}^{7} x_{5}^{26} + x_{1}^{3} x_{3}^{5} x_{4}^{7} x_{5}^{26} + x_{1} x_{2}^{2} x_{3}^{5} x_{4}^{7} x_{5}^{26} \\
    & + x_{2}^{3} x_{3}^{5} x_{4}^{7} x_{5}^{26} + x_{1} x_{3}^{7} x_{4}^{7} x_{5}^{26} + x_{2} x_{3}^{7} x_{4}^{7} x_{5}^{26} + x_{1}^{7} x_{2}^{7} x_{5}^{27} \\
    & + x_{1}^{7} x_{2} x_{3}^{6} x_{5}^{27} + x_{1}^{3} x_{2}^{5} x_{3}^{6} x_{5}^{27} + x_{1} x_{2}^{7} x_{3}^{6} x_{5}^{27} + x_{1}^{7} x_{3}^{7} x_{5}^{27} \\
    & + x_{1} x_{2}^{6} x_{3}^{7} x_{5}^{27} + x_{2}^{7} x_{3}^{7} x_{5}^{27} + x_{1}^{7} x_{2} x_{3}^{2} x_{4}^{4} x_{5}^{27} + x_{1}^{3} x_{2}^{5} x_{3}^{2} x_{4}^{4} x_{5}^{27} \\
    & + x_{1} x_{2}^{7} x_{3}^{2} x_{4}^{4} x_{5}^{27} + x_{1}^{3} x_{2} x_{3}^{6} x_{4}^{4} x_{5}^{27} + x_{1} x_{2}^{3} x_{3}^{6} x_{4}^{4} x_{5}^{27} + x_{1} x_{2}^{2} x_{3}^{7} x_{4}^{4} x_{5}^{27} \\
    & + x_{1}^{7} x_{2} x_{4}^{6} x_{5}^{27} + x_{1}^{3} x_{2}^{5} x_{4}^{6} x_{5}^{27} + x_{1} x_{2}^{7} x_{4}^{6} x_{5}^{27} + x_{1}^{7} x_{3} x_{4}^{6} x_{5}^{27} \\
    & + x_{1} x_{2}^{6} x_{3} x_{4}^{6} x_{5}^{27} + x_{2}^{7} x_{3} x_{4}^{6} x_{5}^{27} + x_{1}^{3} x_{2} x_{3}^{4} x_{4}^{6} x_{5}^{27} + x_{1} x_{2}^{3} x_{3}^{4} x_{4}^{6} x_{5}^{27} \\
    & + x_{1}^{3} x_{3}^{5} x_{4}^{6} x_{5}^{27} + x_{1} x_{2}^{2} x_{3}^{5} x_{4}^{6} x_{5}^{27} + x_{2}^{3} x_{3}^{5} x_{4}^{6} x_{5}^{27} + x_{1} x_{3}^{7} x_{4}^{6} x_{5}^{27} \\
    & + x_{2} x_{3}^{7} x_{4}^{6} x_{5}^{27} + x_{1}^{7} x_{4}^{7} x_{5}^{27} + x_{1} x_{2}^{6} x_{4}^{7} x_{5}^{27} + x_{2}^{7} x_{4}^{7} x_{5}^{27} \\
    & + x_{1} x_{2}^{2} x_{3}^{4} x_{4}^{7} x_{5}^{27} + x_{1} x_{3}^{6} x_{4}^{7} x_{5}^{27} + x_{2} x_{3}^{6} x_{4}^{7} x_{5}^{27} + x_{3}^{7} x_{4}^{7} x_{5}^{27} \\
    & + x_{1}^{7} x_{2} x_{3}^{3} x_{4}^{2} x_{5}^{28} + x_{1}^{3} x_{2}^{5} x_{3}^{3} x_{4}^{2} x_{5}^{28} + x_{1} x_{2}^{7} x_{3}^{3} x_{4}^{2} x_{5}^{28} + x_{1}^{3} x_{2} x_{3}^{7} x_{4}^{2} x_{5}^{28} \\
    & + x_{1} x_{2}^{7} x_{3}^{2} x_{4}^{3} x_{5}^{28} + x_{1}^{3} x_{2}^{3} x_{3}^{4} x_{4}^{3} x_{5}^{28} + x_{1} x_{2}^{3} x_{3}^{6} x_{4}^{3} x_{5}^{28} + x_{1}^{3} x_{2}^{3} x_{3} x_{4}^{6} x_{5}^{28} \\
    & + x_{1}^{3} x_{2} x_{3}^{3} x_{4}^{6} x_{5}^{28} + x_{1}^{3} x_{2} x_{3}^{2} x_{4}^{7} x_{5}^{28} + x_{1} x_{2}^{3} x_{3}^{2} x_{4}^{7} x_{5}^{28} + x_{1}^{3} x_{2}^{3} x_{3}^{5} x_{5}^{30} \\
    & + x_{1}^{3} x_{2}^{5} x_{3}^{2} x_{4} x_{5}^{30} + x_{1}^{3} x_{2}^{4} x_{3}^{3} x_{4} x_{5}^{30} + x_{1} x_{2}^{6} x_{3}^{3} x_{4} x_{5}^{30} + x_{1}^{3} x_{2}^{3} x_{3}^{4} x_{4} x_{5}^{30} \\
    & + x_{1}^{3} x_{2} x_{3}^{6} x_{4} x_{5}^{30} + x_{1}^{7} x_{2} x_{3} x_{4}^{2} x_{5}^{30} + x_{1} x_{2}^{7} x_{3} x_{4}^{2} x_{5}^{30} + x_{1}^{3} x_{2}^{4} x_{3} x_{4}^{3} x_{5}^{30} \\
    & + x_{1} x_{2}^{6} x_{3} x_{4}^{3} x_{5}^{30} + x_{1}^{3} x_{2} x_{3}^{4} x_{4}^{3} x_{5}^{30} + x_{1} x_{2}^{3} x_{3}^{4} x_{4}^{3} x_{5}^{30} + x_{1} x_{2}^{2} x_{3}^{5} x_{4}^{3} x_{5}^{30} \\
    & + x_{1} x_{2} x_{3}^{6} x_{4}^{3} x_{5}^{30} + x_{1}^{3} x_{2}^{3} x_{3} x_{4}^{4} x_{5}^{30} + x_{1}^{3} x_{2} x_{3}^{3} x_{4}^{4} x_{5}^{30} + x_{1} x_{2}^{3} x_{3}^{3} x_{4}^{4} x_{5}^{30} \\
    & + x_{1}^{3} x_{2}^{3} x_{4}^{5} x_{5}^{30} + x_{1} x_{2}^{3} x_{3}^{2} x_{4}^{5} x_{5}^{30} + x_{1}^{3} x_{3}^{3} x_{4}^{5} x_{5}^{30} + x_{2}^{3} x_{3}^{3} x_{4}^{5} x_{5}^{30} \\
    & + x_{1}^{3} x_{2} x_{3} x_{4}^{6} x_{5}^{30} + x_{1} x_{2}^{3} x_{3} x_{4}^{6} x_{5}^{30} + x_{1} x_{2} x_{3}^{3} x_{4}^{6} x_{5}^{30} + x_{1} x_{2} x_{3}^{2} x_{4}^{7} x_{5}^{30}.
\end{align*}

\subsection{An explicit basis for the space  $(Q_{n_0 = 18}^{\otimes 5})^{\overline{\mathsf{Param}}_{[4]}}$}\label{ap5.1}

From our previous work \cite{D.P3}, we see that $\{[Y_j]_{\overline{\mathsf{Param}}_{[4]}}: 1\leq j\leq 110\}$ is a basis for $(Q_{n_0 = 18}^{\otimes 5})^{\overline{\mathsf{Param}}_{[4]}}$, where the monomials $Y_j$, $1\leq j\leq 110$, are determined as follows:

\begin{longtable}{llll}
$Y_{1} = x_2x_3x_4x_5^{15}$ & $Y_{2} = x_2x_3x_4^{15}x_5$ & $Y_{3} = x_2x_3^{15}x_4x_5$ & $Y_{4} = x_2^{15}x_3x_4x_5$ \\
$Y_{5} = x_1x_3x_4x_5^{15}$ & $Y_{6} = x_1x_3x_4^{15}x_5$ & $Y_{7} = x_1x_3^{15}x_4x_5$ & $Y_{8} = x_1x_2x_4x_5^{15}$ \\
$Y_{9} = x_1x_2x_4^{15}x_5$ & $Y_{10} = x_1x_2x_3x_5^{15}$ & $Y_{11} = x_1x_2x_3^{15}x_5$ & $Y_{12} = x_1x_2x_3x_4^{15}$ \\
$Y_{13} = x_1x_2x_3^{15}x_4$ & $Y_{14} = x_1x_2^{15}x_4x_5$ & $Y_{15} = x_1x_2^{15}x_3x_5$ & $Y_{16} = x_1x_2^{15}x_3x_4$ \\
$Y_{17} = x_1^{15}x_3x_4x_5$ & $Y_{18} = x_1^{15}x_2x_4x_5$ & $Y_{19} = x_1^{15}x_2x_3x_5$ & $Y_{20} = x_1^{15}x_2x_3x_4$ \\
$Y_{21} = x_2x_3x_4^{3}x_5^{13}$ & $Y_{22} = x_2x_3^{3}x_4x_5^{13}$ & $Y_{23} = x_2x_3^{3}x_4^{13}x_5$ & $Y_{24} = x_2^{3}x_3x_4x_5^{13}$ \\
$Y_{25} = x_2^{3}x_3x_4^{13}x_5$ & $Y_{26} = x_2^{3}x_3^{13}x_4x_5$ & $Y_{27} = x_1x_3x_4^{3}x_5^{13}$ & $Y_{28} = x_1x_3^{3}x_4x_5^{13}$ \\
$Y_{29} = x_1x_3^{3}x_4^{13}x_5$ & $Y_{30} = x_1x_2x_4^{3}x_5^{13}$ & $Y_{31} = x_1x_2x_3^{3}x_5^{13}$ & $Y_{32} = x_1x_2x_3^{3}x_4^{13}$ \\
$Y_{33} = x_1x_2^{3}x_4x_5^{13}$ & $Y_{34} = x_1x_2^{3}x_4^{13}x_5$ & $Y_{35} = x_1x_2^{3}x_3x_5^{13}$ & $Y_{36} = x_1x_2^{3}x_3^{13}x_5$ \\
$Y_{37} = x_1x_2^{3}x_3x_4^{13}$ & $Y_{38} = x_1x_2^{3}x_3^{13}x_4$ & $Y_{39} = x_1^{3}x_3x_4x_5^{13}$ & $Y_{40} = x_1^{3}x_3x_4^{13}x_5$ \\
$Y_{41} = x_1^{3}x_3^{13}x_4x_5$ & $Y_{42} = x_1^{3}x_2x_4x_5^{13}$ & $Y_{43} = x_1^{3}x_2x_4^{13}x_5$ & $Y_{44} = x_1^{3}x_2x_3x_5^{13}$ \\
$Y_{45} = x_1^{3}x_2x_3^{13}x_5$ & $Y_{46} = x_1^{3}x_2x_3x_4^{13}$ & $Y_{47} = x_1^{3}x_2x_3^{13}x_4$ & $Y_{48} = x_1^{3}x_2^{13}x_4x_5$ \\
$Y_{49} = x_1^{3}x_2^{13}x_3x_5$ & $Y_{50} = x_1^{3}x_2^{13}x_3x_4$ & $Y_{51} = x_2x_3^{3}x_4^{5}x_5^{9}$ & $Y_{52} = x_2^{3}x_3x_4^{5}x_5^{9}$ \\
$Y_{53} = x_2^{3}x_3^{5}x_4x_5^{9}$ & $Y_{54} = x_2^{3}x_3^{5}x_4^{9}x_5$ & $Y_{55} = x_1x_3^{3}x_4^{5}x_5^{9}$ & $Y_{56} = x_1x_2^{3}x_4^{5}x_5^{9}$ \\
$Y_{57} = x_1x_2^{3}x_3^{5}x_5^{9}$ & $Y_{58} = x_1x_2^{3}x_3^{5}x_4^{9}$ & $Y_{59} = x_1^{3}x_3x_4^{5}x_5^{9}$ & $Y_{60} = x_1^{3}x_3^{5}x_4x_5^{9}$ \\
$Y_{61} = x_1^{3}x_3^{5}x_4^{9}x_5$ & $Y_{62} = x_1^{3}x_2x_4^{5}x_5^{9}$ & $Y_{63} = x_1^{3}x_2x_3^{5}x_5^{9}$ & $Y_{64} = x_1^{3}x_2x_3^{5}x_4^{9}$ \\
$Y_{65} = x_1^{3}x_2^{5}x_4x_5^{9}$ & $Y_{66} = x_1^{3}x_2^{5}x_4^{9}x_5$ & $Y_{67} = x_1^{3}x_2^{5}x_3x_5^{9}$ & $Y_{68} = x_1^{3}x_2^{5}x_3x_4^{9}$ \\
$Y_{69} = x_1^{3}x_2^{5}x_3^{9}x_5$ & $Y_{70} = x_1^{3}x_2^{5}x_3^{9}x_4$ & $Y_{71} = x_1x_2x_3x_4x_5^{14}$ & $Y_{72} = x_1x_2x_3x_4^{14}x_5$ \\
$Y_{73} = x_1x_2x_3^{14}x_4x_5$ & $Y_{74} = x_1x_2x_3x_4^{2}x_5^{13}$ & $Y_{75} = x_1x_2x_3^{2}x_4x_5^{13}$ & $Y_{76} = x_1x_2x_3^{2}x_4^{13}x_5$ \\
$Y_{77} = x_1x_2x_3^{2}x_4^{5}x_5^{9}$ & $Y_{78} = x_1x_2x_3x_4^{3}x_5^{12}$ & $Y_{79} = x_1x_2x_3^{3}x_4x_5^{12}$ & $Y_{80} = x_1x_2x_3^{3}x_4^{12}x_5$ \\
$Y_{81} = x_1x_2x_3^{3}x_4^{4}x_5^{9}$ & $Y_{82} = x_1x_2x_3^{3}x_4^{5}x_5^{8}$ & $Y_{83} = x_1x_2^{3}x_3x_4x_5^{12}$ & $Y_{84} = x_1x_2^{3}x_3x_4^{12}x_5$ \\
$Y_{85} = x_1x_2^{3}x_3x_4^{4}x_5^{9}$ & $Y_{86} = x_1x_2^{3}x_3x_4^{5}x_5^{8}$ & $Y_{87} = x_1x_2^{3}x_3^{5}x_4x_5^{8}$ & $Y_{88} = x_1x_2^{3}x_3^{5}x_4^{8}x_5$ \\
$Y_{89} = x_1x_2^{14}x_3x_4x_5$ & $Y_{90} = x_1x_2^{2}x_3x_4x_5^{13}$ & $Y_{91} = x_1x_2^{2}x_3x_4^{13}x_5$ & $Y_{92} = x_1x_2^{2}x_3x_4^{5}x_5^{9}$ \\
$Y_{93} = x_1x_2^{2}x_3^{13}x_4x_5$ & $Y_{94} = x_1x_2^{2}x_3^{5}x_4x_5^{9}$ & $Y_{95} = x_1x_2^{3}x_3^{12}x_4x_5$ & $Y_{96} = x_1x_2^{3}x_3^{4}x_4x_5^{9}$ \\
$Y_{97} = x_1x_2^{2}x_3^{5}x_4^{9}x_5$ & $Y_{98} = x_1x_2^{3}x_3^{4}x_4^{9}x_5$ & $Y_{99} = x_1^{3}x_2x_3x_4x_5^{12}$ & $Y_{100} = x_1^{3}x_2x_3x_4^{12}x_5$ \\
$Y_{101} = x_1^{3}x_2x_3x_4^{4}x_5^{9}$ & $Y_{102} = x_1^{3}x_2x_3x_4^{5}x_5^{8}$ & $Y_{103} = x_1^{3}x_2x_3^{5}x_4x_5^{8}$ & $Y_{104} = x_1^{3}x_2x_3^{5}x_4^{8}x_5$ \\
$Y_{105} = x_1^{3}x_2^{5}x_3x_4x_5^{8}$ & $Y_{106} = x_1^{3}x_2^{5}x_3x_4^{8}x_5$ & $Y_{107} = x_1^{3}x_2x_3^{12}x_4x_5$ & $Y_{108} = x_1^{3}x_2x_3^{4}x_4x_5^{9}$ \\
$Y_{109} = x_1^{3}x_2^{5}x_3^{8}x_4x_5$ & $Y_{110} = x_1^{3}x_2x_3^{4}x_4^{9}x_5.$ &  &  \\
\end{longtable}

\subsection{The nonadmissible monomials $T_i$, $1\leq i\leq 1685$, with $\mathsf{Param}(T_i) = (3,3,2,1,1)$}\label{ap5.2}

The complete list of 1685 nonadmissible monomials $T_i$ is available at Zenodo: \url{https://doi.org/10.5281/zenodo.17622854}.

\subsection{Output of the algorithm for a basis of $Q_{n_0 = 18}^{\otimes 5}$ and its invariants}\label{ap5.3}

Computational data for $Q_{n_0 = 18}^{\otimes 5}$ and its invariants are available at Zenodo:\\ 
\url{https://doi.org/10.5281/zenodo.17613944}.

\subsection{Output of the algorithm for a basis of $Q_{n_1 = 41}^{\otimes 5},$ and its invariants}\label{ap5.4}

Computational data for $Q_{n_1 = 41}^{\otimes 5}$ and its invariants are available at Zenodo:\\ 
\url{https://doi.org/10.5281/zenodo.17613995}.

\end{document}